\documentclass[a4paper,12pt,leqno]{amsart}
\usepackage{latexsym}
\usepackage{amsmath}
\usepackage{amssymb}
\usepackage{amsthm}
\usepackage{indentfirst}
\usepackage[english]{babel} 
\usepackage[T1]{fontenc}
\usepackage{breqn}
\usepackage{cite}    
\usepackage{mathrsfs}                   
\usepackage[colorlinks=true]{hyperref}

\theoremstyle{plain} 
\newtheorem{thm}{Theorem}[section] 
\newtheorem{cor}[thm]{Corollary} 
\newtheorem{lem}[thm]{Lemma} 
\newtheorem{prop}[thm]{Proposition} 
\newtheorem{rmk}[thm]{Remark}

\theoremstyle{definition} 
\newtheorem{defi}{Definition}[section] 
\newcommand{\eps}{\varepsilon}

\numberwithin{equation}{section}

\author[D.~Mastrostefano]{Daniele Mastrostefano}
\address{University of Warwick, Mathematics Institute, Zeeman Building, Coventry, CV4 7AL, UK}
\email{Daniele.Mastrostefano@warwick.ac.uk}


\begin{document}
\title[Lower bounds for the variance of generalized divisor functions in APs]
      {A lower bound for the variance of generalized divisor functions in arithmetic progressions}\thanks{The author is funded by a Departmental Award and by an EPSRC Doctoral Training Partnership Award. The present work has been conducted when the author was a first year PhD student at the University of Warwick.}

\begin{abstract}
We prove that for a large class of multiplicative functions, referred to as generalized divisor functions, it is possible to find a lower bound for the corresponding variance in arithmetic progressions. As a main corollary, we deduce such a result for any $\alpha$-fold divisor function, for any complex number $\alpha\not\in \{1\}\cup-\mathbb{N}$, even when considering a sequence of parameters $\alpha$ close in a proper way to $1$. Our work builds on that of Harper and Soundararajan, who handled the particular case of $k$-fold divisor functions $d_k(n)$, with $k\in\mathbb{N}_{\geq 2}$. 
\end{abstract}

\maketitle

\section{Introduction}
Let $f$ be a complex arithmetic function. It is believed that many $f$ are roughly uniformly distributed in arithmetic progressions, or equivalently that there is an approximation
$$\sum_{\substack{n\leq N\\ n\equiv a\pmod{q}}}f(n)\approx \frac{1}{\phi(q)}\sum_{\substack{n\leq N\\ (n,q)=1}}f(n),$$ 
for any $a\ (\bmod\ {q})$ with $(a,q)=1.$ Here $N$ is a large positive integer and $\phi(q)$ indicates the Euler totient function, which counts the number of reduced residue classes $\bmod\ {q}$, i.e. classes $a=1,\dots, q$ with the greatest common divisor $(a,q)=1$. In order to understand whether this point of view may be correct or not we study the variance
$$V_q(f)=\frac{1}{\phi(q)}\sum_{\substack{a=1,\dots,q\\ (a,q)=1}}\bigg|\sum_{\substack{n\leq N\\ n\equiv a\pmod{q}}}f(n)-\frac{1}{\phi(q)}\sum_{\substack{n\leq N\\ (n,q)=1}}f(n)\bigg|^2.$$
In a series of works, Elliott \cite{E1, E2} and Hildebrand \cite{H} found upper bounds for $V_q(f)$, for single moduli $q$, when $f$ is a multiplicative function with absolute value bounded by $1$. Their results were improved by Balog, Granville and Soundararajan\cite{BGS}, who understood the asymptotic of such variance, at worst in terms of certain exceptional moduli that naturally arise in the context of distribution of functions in arithmetic progressions. We would also like to draw the reader's attention to a very recent result of Klurman, Mangerel and Ter\"{a}v\"{a}inen \cite{KMT} about the variance of $1$--bounded multiplicative functions in short arithmetic progressions. 

For the specific case of $d_2(n)=\sum_{d|n}1$, the function which counts the number of divisors of an integer $n$, such problem has been tackled in the paper of Banks, Heath-Brown and Shparlinski \cite{BHS}. More generally, for the $k$-th divisor function $d_k(n)=\sum_{e_1e_2\cdots e_k=n}1$, which counts all the possible ways of decomposing $n$ into a product of $k$ positive integers, a conjecture on the asymptotic behaviour of its variance in arithmetic progressions has been suggested in the work of Keating, Rodgers, Roditty-Gershon and Rudnick \cite{KRGR}. 

We now introduce an averaged version of the previous variance, defined in the following way.
\begin{defi}
\label{def1}
We define the variance of $f$ in arithmetic progressions by
\begin{equation}
\label{variancedef}
V(Q,f)=\sum_{q\leq Q}\sum_{h|q}\sum_{\substack{a\ \bmod{q}\\ (a,q)=h}}\bigg|\sum_{\substack{n\leq N\\ n\equiv a\ \bmod{q}}}f(n)-\frac{1}{\phi(q/h)}\sum_{\substack{n\leq N\\ (n,q)=h}}f(n)\bigg|^{2}.
\end{equation}
\end{defi}
An asymptotic equality for $V(Q, d_2)$ has been established by Motohashi \cite{M}, whereas for $V(Q, d_k)$ by de la Bret\`{e}che and Fiorilli \cite{BF}; for a smooth version of $V(Q, d_k)$, in which the function $d_k(n)$ is twisted with a smooth weight, the result is contained in the paper of Rodgers and Soundararajan \cite{RS}. It is important to note that the last two articles deal only with values of $Q$ lying in a limited range. More specifically, for any $\delta>0$ it is roughly required that $N^{1/(1+2/k-\delta)}\leq Q\leq N^{1/\delta}$.
\subsection{Statement of the main results}
The present paper studies the question of finding a lower bound for the quantity $V(Q,f)$ when considering suitable generalizations $f$ of the divisor functions introduced above. The main reference on this problem is the Harper and Soundararajan's paper \cite{HS}, in which the authors set up the bases for the study of lower bounds of variances of complex sequences in arithmetic progressions. More precisely, they showed that for a wide class of functions with a controlled growth we can lower bound the variance \eqref{variancedef} with the $L^2$-norm of the exponential sum with coefficients $f(n)$ over a large portion of the circle (namely, the so called union of minor arcs). Since for functions that fluctuate like random we usually expect that such integral makes the largest contribution compared to that on the complementary portion (note how here the situation is the opposite of what happens in some common additive problems, like in the three--primes problem), we are led to the following heuristic 
$$V(Q,f)\gg Q\int_{0}^{1}|\sum_{n\leq N}f(n)e(n\varphi)|^2d\varphi,$$
where $e(n\varphi)$ stands for $e^{2\pi in\varphi}$, which by Parseval's identity can be rewritten as
\begin{equation}
\label{conjestimate}
V(Q,f)\gg Q\sum_{n\leq N}|f(n)|^2.
\end{equation}
This lower bound has been proven to hold when $f(n)=\Lambda(n)$ and $f(n)=d_k(n)$, for $k\geq 2$ a positive integer and $Q$ in the range $N^{1/2+\delta}\leq Q\leq N$, for any small $\delta>0$, by Theorem 1 and Theorem 2 in \cite{HS}. Here $\Lambda(n)$ indicates as usual the Von Mangoldt function. 
 
The aim of this paper is to show that the new powerful method introduced in \cite{HS} can be used to prove the validity of \eqref{conjestimate} for every $\alpha$--fold divisor function $d_\alpha(n)$, defined as the $n$-th coefficient in the Dirichlet series of $\zeta(s)^{\alpha}$ on the half plane $\Re(s)>1$, except for some specific values of $\alpha$. The result is contained in the following theorem.
\begin{thm}
\label{cor1}
Let $\delta>0$ sufficiently small and consider $N^{1/2+\delta}\leq Q\leq N$. For any complex number $\alpha\not\in -\mathbb{N}\cup \{1\}$, we have
\begin{equation}
\label{statementcor1}
V(Q,d_\alpha)\gg_{\alpha,\delta} Q\sum_{n\leq N}|d_\alpha(n)|^2,
\end{equation}
if $N$ large enough with respect to $\alpha$ and $\delta$.
\end{thm}
We now introduce the following class of functions, which extends that of divisor functions.
\begin{defi}
\label{defgendivfun}
A \textit{generalized divisor function} is a multiplicative function for which there exist a complex number $\alpha$ and positive real numbers $\beta, A_1, A_2$ such that the following statistics hold
\begin{align}
\label{mainstatistic1}
&\sum_{p\leq x}f(p)\log p=\alpha x+O\bigg(\frac{x}{(\log x)^{A_1}}\bigg)\ \ \ (2\leq x\leq N),\\
\label{mainstatistic2}
&\sum_{p\leq x}|f(p)-1|^2\log p=\beta x+O\bigg(\frac{x}{(\log x)^{A_2}}\bigg)\ \ \ (2\leq x\leq N)
\end{align}
and such that $|f(n)|\leq d_{\kappa}(n)$, for a constant $\kappa>0$ and every $N$-smooth positive integer $n$, i.e. for any $n$ divisible only by prime numbers smaller than $N$.
\end{defi}
In \eqref{mainstatistic1} and \eqref{mainstatistic2} the sums are over prime numbers $p$ and we will keep such notation throughout the rest of this paper. 

\begin{rmk}
From \eqref{mainstatistic1} we deduce that when $\alpha\neq 0$
\begin{align*}
|\alpha|N\bigg(1+O\bigg(\frac{1}{(\log N)^{A_1}}\bigg)\bigg)&=|\sum_{p\leq N}f(p)\log p|\\
&\leq \sum_{p\leq N}|f(p)|\log p\\
&\leq \sum_{p\leq N}\kappa\log p\\
&=\kappa N\bigg(1+O\bigg(\frac{1}{(\log N)^{A_1}}\bigg)\bigg),
\end{align*}
by the prime number theorem (see e.g. \cite[Theorem 6.9]{MV}). We conclude that for any $\alpha$ we have $|\Re(\alpha)|\leq |\alpha|\leq \kappa(1+O(1/(\log N)^{A_1}))$. Similarly, but using \eqref{mainstatistic2}, we get $\beta\leq (\kappa+1)^2(1+O(1/(\log N)^{A_2}))$. In particular, we deduce that $|\alpha|\leq \kappa+1$ and $\beta\leq (\kappa+2)^2$, if $N$ large enough in terms of $\kappa,A_1,A_2$ and the implicit constants \eqref{mainstatistic1}--\eqref{mainstatistic2}. By the monotonicity of $d_{\kappa}(n)$ as function of $\kappa>0$ and by replacing $\kappa$ with $\kappa+1$, we may thus assume that $\kappa>1$ and $|\alpha|\leq \kappa$ and $\beta\leq (\kappa+1)^2$.
\end{rmk}

We define for future reference and for the sake of readiness the quantity 
$$\kappa(\alpha,\beta):=(\kappa+1)^2+\kappa-\Re(\alpha)-\beta+4\geq 4.$$ 

We observe that when $f=d_\alpha(n)$ the equations \eqref{mainstatistic1}--\eqref{mainstatistic2} are trivially satisfied by the prime number theorem, whose reference above, with $\beta=|\alpha-1|^2$, $\kappa=|\alpha|+2$ and any $A_1,A_2 >0$. When $\alpha\neq 0$, Theorem \ref{cor1} is then a corollary of the following \emph{main} result.
\begin{thm}
\label{thmalpha}
Let $\delta$ be a sufficiently small positive real number and $N$ be a large positive integer. Suppose that $N^{1/2+\delta}\leq Q\leq N$. Let $f(n)$ be a generalized divisor function as in Definition \ref{defgendivfun} with $\alpha\not\in -\mathbb{N}\cup \{0\}$. Furthermore, assume that 
\begin{align}
\label{relationbetaA_1}
&A_1>\max\{\kappa(\alpha,\beta),\kappa+2\};\\
&A_2>A_1-\kappa(\alpha,\beta)+1;\nonumber\\
&\beta\geq (\log N)^{\kappa(\alpha,\beta)-A_1};\nonumber\\
& |\Gamma(\alpha)|\leq \log N,\nonumber
\end{align}
where $\Gamma(\alpha)$ is the Gamma function. Finally, let
$$c_0=\prod_{p\leq N}\bigg(1+\frac{f(p)}{p}+\frac{f(p^2)}{p^2}+\cdots\bigg)\bigg(1-\frac{1}{p}\bigg)^{\alpha}$$
and suppose that
\begin{equation}
\label{gammaassumpt}
(\log N)^{1-\delta}|c_0|\geq 1.
\end{equation}
Then we have
\begin{equation}
\label{variance1}
V(Q,f)\gg\bigg|\frac{c_0\beta}{\Gamma(\alpha)}\bigg|^2 Q\sum_{n\leq N}|f(n)|^2.
\end{equation}
The implicit constant above may depend on $\delta,\kappa,A_1,A_2$ and the implicit constants in \eqref{mainstatistic1}--\eqref{mainstatistic2} and we take $N$ large enough depending on all of these parameters.
\end{thm}
\begin{rmk} We note that we clearly have $A_1,A_2>1$. Also, the implicit constant in \eqref{variance1} \emph{does not} depend on $\alpha$. 
\end{rmk}
\begin{rmk} The last condition in \eqref{relationbetaA_1} has been inserted to avoid the scenario in which $\alpha$ is too close to a pole of the Gamma function, which would make us losing control on the average of $f(n)$ over integers $n\leq N$, thus precluding us from producing a lower bound for $V(Q,f)$. However, by slightly modifying the conditions on $A_1,A_2$ and $\beta$ in \eqref{relationbetaA_1}, as well as the definition of other parameters involved in the proof, it may be possible to relax such restriction to make $|\Gamma(\alpha)|$ indeed smaller than a suitable larger power of $\log N$. 
\end{rmk}
\begin{rmk}
Condition \eqref{gammaassumpt} makes sure that $f$ looks nice on the primes $p\leq \kappa$, thus excluding certain patterns of $f(p)$ where $c_0$ is too close to $0$. However, in this last situation, one could still be able to replace \eqref{variance1} with a lower bound of a different shape by carefully understanding some non-trivial derivatives of the Euler product of the Dirichlet series of $f$. 
\end{rmk}
Theorem \ref{thmalpha} is quite technical, but it does not merely represent an improvement upon Theorem \ref{cor1}. Indeed it allows us to lower bound the variance of multiplicative functions that arise from divisor functions, such as for instance positive integer powers of $d_2(n)$ or products of divisor functions as $d_2(n)d_3(n)$, but also, and most importantly, of those that behave very differently from the simple divisor functions. As a concrete example of this last case we state here the following corollary.
\begin{cor}
\label{sumoftwosquares}
Let $\delta>0$ sufficiently small and consider $N^{1/2+\delta}\leq Q\leq N$. Let $S$ be the set of all integer sums of two squares. Then we have
\begin{equation}
\label{variancesumoftwosquares}
V(Q, \textbf{1}_{S})\gg_{\delta} \frac{QN}{\sqrt{\log N}},
\end{equation}
if $N$ is large enough with respect to $\delta$.
\end{cor} 
\begin{rmk} Note that the lower bound \eqref{variancesumoftwosquares} is consistent with the Parseval heuristic \eqref{conjestimate}. 
\end{rmk}
The distribution of sums of two squares in arithmetic progressions has been studied by a number of authors, among which Fiorilli \cite{F}, Iwaniec \cite{I}, Lin--Zhan \cite{LZ} and Rieger \cite{RI1, RI2}.
\subsection{About the variance of certain multiplicative functions in arithmetic progressions.}
For suitably chosen parameters $K,Q$ and $Q_0$, we are going to define the so called set of major arcs ${\frak M} = {\frak M}(Q_0,Q;K)$, consisting of those $\varphi \in {\Bbb R}/{\Bbb Z}$ having an approximation $|\varphi-a/q| \le K/(qQ)$, with $q\le K Q_0$ and $(a,q)=1$. Let ${\frak m}$, the minor arcs, denote the complement of the major arcs in ${\Bbb R}/{\Bbb Z}$. Clearly, this last set occupies almost the totality of the circle, depending on $K,Q$ and $Q_0$, and it consists of real numbers well approximated by rational fractions with large denominator. 

The idea exploited in \cite{HS} was to connect the variance in arithmetic progressions with the minor arc contribution of the exponential sum with coefficients $f(n)$. This point of view was already widespread and present in the literature, like for example in the works of Liu \cite{LI1, LI2} and Perelli \cite{PE}. However, as explained in \cite{HS}, the previous arguments relied on the connection between character sums and exponential sums (similarly as in the usual deductions of the multiplicative large sieve inequality), which can only be made to work for the $\Lambda$-function sequence or for other sequences without small prime factors. In contrast, as pointed out in \cite{HS}, Harper and Soundararajan avoided the use of Dirichlet characters in favour of Hooley's approach, connecting the variance of $f(n)$ in
arithmetic progressions with the variance of the exponential sums $\sum_{n\leq N}f(n)e(na/q)$. By positivity
of the variance one can discard the major arc contribution to the latter, leaving only a minor arc contribution and some terms
involving Ramanujan sums $c_d(n)$ with $d$ fairly large. Those sums are simply defined as 
$$c_d(n)=\sum_{\substack{a=1,\dots,d\\ (a,d)=1}}e(an/d)$$
and characterized by their main property that we will make use of several times in the future
\begin{equation}
\label{mainpropc_q}
c_d(n)=\sum_{k|(n,d)}k\mu(d/k),
\end{equation} 
where $\mu(n)$ is the Mobius function. In details, the result proved in \cite{HS} is the following.
\begin{prop} 
\label{Prop5.1} 
Let $N$ be large positive integer, $K\ge 5$ be a parameter and $Q,Q_0$ be such that 
\begin{equation} 
\label{eq0} 
K \sqrt{N\log N} \le Q \le N\ \textrm{and}\ \ \frac{N \log N}{Q} \le Q_0 \le \frac{Q}{K^2}. 
\end{equation} 
Keeping notations as above, we then have 
\begin{align}
\label{estimateprop1}
V(Q,f) &\ge Q\Big(1+ O\Big(\frac{\log K}{K}\Big)\Big) \int_{\frak m} |{\mathcal F}(\varphi)|^2d\varphi + O \Big( \frac{NK}{Q_0} \sum_{n\le N} |f(n)|^2 \Big)\\
&+O\bigg(\sum_{q\le Q } \frac{1}{q} \sum_{\substack {d|q \\ d>Q_0}} \frac{1}{\phi(d)} \Big| \sum_{n\leq N} f(n) c_d(n)\Big|^2\bigg),\nonumber
\end{align}
where $\mathcal{F}(\varphi):=\sum_{n\leq N}f(n)e(n\varphi)$.
\end{prop} 
It is clear that Proposition \ref{Prop5.1} boils the variance question down to find a lower bound for the integral over the minor arcs of $|\mathcal{F}(\varphi)|^2$. The particular form of the above result suggests a heuristic leading to \eqref{conjestimate}. Indeed, if the contribution of $\mathcal{F}(\varphi)$ on the minor arcs exceeds that on the major arcs, we can approximate the integral in \eqref{estimateprop1} with the integral over all the circle, obtaining 
$$\int_{\frak m} |{\mathcal F}(\varphi)|^2d\varphi\approx \sum_{n\leq N}|f(n)|^2,$$
by Parseval's identity. For instance, we expect such behaviour for all the functions $f(n)=d_{\alpha}(n)$ with $\alpha\neq 1$. Indeed, we believe that those divisor functions do not correlate with the exponential phase. However, when $\alpha$ equals a negative integer, the situation is possibly subtle, as shown by the following discussion. Consider for instance $\alpha=-1$. Now the sum $\mathcal{F}(\varphi)$ over the major arcs would be close to $\sum_{n\leq N}\mu(n)$ and by the squareroot cancellation principle it should be at most $N^{1/2+\eps}$, for an arbitrary small $\eps$. On the other hand, $\mathcal{F}(\varphi)$ roughly coincides with $\sum_{n\leq N}\mu(n)e(n\varphi)$ over the minor arcs, where by the random fluctuations of the prime numbers we expect a value of roughly $\sqrt{N}$, being a sum of $N$ pseudorandom phases. In conclusion, it is difficult to distinguish between major and minor arc contribution in the case of the Mobius function.

When instead $\alpha$ is either $0$ or $1$ the simple form of those divisor functions allows us to elementarily study the associated variance. The result is contained in the next proposition, which in the latter case highlights a different variance behaviour from the Parseval heuristic compared to the other divisor functions.
\begin{prop}
\label{thm0,1}
For any $Q\geq 1$, we have
\begin{align*}
\label{variance0,1}
V(Q,d_{0})&=Q+O(\log Q),\\
V(Q,d_{1})&\ll Q^2.\nonumber
\end{align*}
\end{prop}
\begin{proof}
Let us start with the function $d_0(n)$. It is clear that the contribution to \eqref{variancedef} is not zero only if $h=1$ and in such case the term inside the square reduces to $\textbf{1}_{a\equiv 1\pmod{q}}(a)-1/\phi(q)$. Thus we get
$$V(Q,d_0)=\sum_{q\leq Q}\left(1-\frac{1}{\phi(q)}\right)^{2}+\sum_{q\leq Q}\sum_{\substack{a\pmod{q}\\ (a,q)=1\\ a\not\equiv 1\pmod{q}}}\frac{1}{\phi(q)^{2}}=\sum_{q\leq Q}\left(1-\frac{1}{\varphi(q)}\right)=Q+O(\log Q),$$
by Landau's result \cite[p. 184]{L}. This concludes the proof of the first part of Proposition \ref{thm0,1} and of the case $\alpha=0$ in Theorem \ref{cor1}.

Let us now consider the function $d_1(n)\equiv 1$, for all $n\in\mathbb{N}$. We immediately see that
$$\sum_{\substack{n\leq N\\ n\equiv a\pmod{q}}}d_1(n)=\frac{N}{q}+O(1).$$
On the other hand, we have
$$\sum_{\substack{n\leq N\\ (n,q)=h}}d_1(n)=\frac{N}{q}\phi(q/h)+O(d_2(q/h)),$$
using \cite[ch. I, Theorem 2.8]{T} and \cite[ch. I, Theorem 2.13]{T}. Inserting these two identities in \eqref{variancedef}, we get
$$V(Q,f)\ll \sum_{q\leq Q}\sum_{h|q}\sum_{\substack{a\pmod{q}\\ (a,q)=h}}1\leq \sum_{q\leq Q} q\ll Q^2,$$
where we used that $d_2(n)/\phi(n)\ll 1$, for any $n\in\mathbb{N}$. 

This concludes the proof of Proposition \ref{thm0,1}.
\end{proof}
We insert here a discussion about the case $\alpha=0$, in which the variance of a specific class of functions can be easily estimated and its value does not match the Parseval heuristic.

Indeed, all the positive monotone non-increasing multiplicative functions satisfy \eqref{mainstatistic1} and \eqref{mainstatistic2} with $\alpha=0$ and $\beta=1$, since such functions must be of the shape $f(n)=n^{-\gamma}$, for a certain $\gamma>0$. Thus we observe that \eqref{conjestimate}, if true, holds only for certain values of $Q$ in the range $N^{1/2+\delta}\leq Q\leq N$. In fact, if we assume $\gamma<1/2$, then the variance \eqref{variancedef} can be seen to be $\ll Q^2$, which can be fairly small compared to $Q\sum_{n\leq N}f(n)^2\gg QN^{1-2\gamma}$. This is because for such class of functions we have the following bound
$$\bigg|\sum_{\substack{n\leq N\\ n\equiv a\pmod{q}}}f(n)-\frac{1}{\phi(q/h)}\sum_{\substack{n\leq N\\ (n,q)=h}}f(n)\bigg|\ll f(1)=1,$$
for any triple $a,q,h$ with $(a,q)=h$. Similar considerations hold if we replace non-increasing with non-decreasing.

When $\alpha=1$, Proposition \ref{thm0,1} states that the corresponding variance in arithmetic progression is at most a possibly large constant times $Q^2$, whereas the Parseval heuristic would suggest a lower bound of $QN$. However, one might wonder what happens for a sequence of divisor functions $f=d_{\alpha_N}$ for values of $\alpha_N$ close to $1$. In this case, Theorem \ref{thmalpha} still gives us a lower bound for the variance.
\begin{thm}
\label{uniformalpha}
Let $A>0$ be a real number and $\alpha_N=1+1/R(N)$, where $R(N)$ is a real non-vanishing function such that $|R(N)|\leq (\log N)^A$. Let $\delta>0$ small enough and $N^{1/2+\delta}\leq Q\leq N$. Then there exists a constant $C>0$ such that if $|R(N)|\geq C$ we have \footnote{In a forthcoming paper we will prove a stronger lower bound valid for larger values of $R(N)$, by pursuing a slightly different approach from what has been done here. In particular, whenever $|R(N)|\geq C(\log\log N)^2$, we are able to improve \eqref{varianceuniform} to:\\
$$V(Q,d_{\alpha_N})\gg_{\delta}\frac{QN}{R(N)^{2}}\log\bigg(\frac{\log N}{\log(N/Q)}\bigg),$$
if $N$ is large enough with respect to $\delta$ and $A$, which we expect to be best possible.}
\begin{equation}
\label{varianceuniform}
V(Q,d_{\alpha_N})\gg_{\delta,A}\frac{QN}{R(N)^{4}}\exp\bigg(\bigg(2+\frac{1}{R(N)}\bigg)\frac{\log\log N}{R(N)}\bigg),
\end{equation}
if $N$ is large enough with respect to $\delta$ and $A$.
\end{thm}
\subsection{Lower bounding the integral over the minor arcs}
The two error terms in \eqref{estimateprop1} can be fairly easily estimated and I will insert the details in due course. Regarding the main term we first follow the idea introduced in \cite{HS} to apply the Cauchy--Schwarz inequality to get
\begin{equation}
\label{C-S*}
\int_{\frak m} |{\mathcal F}(\varphi)|^2d\varphi\geq \bigg(\int_{\frak m} |{\mathcal F}(\varphi)\tilde{\mathcal{F}}(\varphi)|d\varphi\bigg)^2\bigg(\int_{\frak m} |\tilde{\mathcal {F}}(\varphi)|^2d\varphi\bigg)^{-1},
\end{equation}
where
$$\tilde{\mathcal{F}}(\varphi)=\sum_{n\leq N}\bigg(\sum_{\substack{r|n\\ r\leq R}}g(r)\bigg)e(n\varphi),$$
for a suitable function $g(r)$. In this way we are reduced to estimate integrals of exponential sums in which the coefficients are an opportune approximation of the function $f(n)$. The heart of the proof in \cite{HS} was in the observation that we can move from integrals of exponential sums with coefficients $f(n)$ to sums where $f(n)$ is twisted by a Ramanujan sum $c_q(n)$. The result is contained in the following proposition, which is a minor modification and a simplified version of \cite[Proposition 3]{HS}, in which a smooth weight in the average of $f$ has been removed by introducing a small error term.
\begin{prop}
\label{newprop10} 
Keep notations as above, and assume that $KQ_0 \le R\le \sqrt{N}$ and $|f(n)| \ll_{\epsilon} N^{\epsilon}$ for any $\epsilon > 0$ and $n\leq N$.  Then 
\begin{equation}
\label{eq81}
\int_{\frak m} |{\mathcal F}(\varphi) \tilde {\mathcal F}(\varphi) | d\varphi \ge \sum_{KQ_0 < q\le R} \Big| \sum_{\substack{r\le R \\ q|r}} \frac{g(r)}{r} \Big| \Big| \sum_{n\le N} 
f(n) c_q(n) \Big|+ O_{\epsilon}(\Delta RN^{\frac 12 +\epsilon}  +N/\log^{B}N), 
\end{equation}
where $\Delta=\max_{r\leq R}|g(r)|$ and $B$ any positive real constant.
\end{prop} 
In \cite{HS} Proposition \ref{newprop10} has been applied to deduce the lower bound for the variance of primes in arithmetic progressions as well as for the sequence of divisor functions $d_k(n)$, for a positive integer $k\geq 2$. In the former case, the computations are relatively straightforward because $(n, q) = 1$ for almost all prime (or prime power) values of $n$, so the Ramanujan sum $c_q(n)$ takes the value $\mu(q)$ for almost all such values. In the latter, the argument is more intricate and requires a lot more work. The authors showed that it is possible to find a lower bound for $\sum_{n\le N}d_k(n) c_q(n)$, when suitably restricting the range in which $q$ varies. However, their techniques do not extend to the case of any general divisor function of parameter $\alpha\in\mathbb{C}$. 

Indeed, the main difference between our approach and that of Harper and Soundararajan is in the computation of the mean value of $f(n)$ twisted with a Ramanujan sum. More precisely, Harper and Soundararajan used the classical approach of rewriting the sum in question as an integral of the corresponding Dirichlet series, by means of Perron's formula and then concluding the estimate by using a residue computation. This strategy is admissible since the aforementioned Dirichlet series is a slight variation of $\zeta(s)^k$; since $k$ is a positive integer, it can be extended to a meromorphic function on the whole complex plane with just one pole at $1$. On the other hand, when replacing $k$ with any complex number $\alpha$, the function $\zeta(s)^{\alpha}$ may have an essential singularity at $s=1$, which makes the previous approach inapplicable. More generally, for the class of functions $f(n)$ introduced in Definition \ref{defgendivfun}, the corresponding Dirichlet series can only be defined on the half plane of complex numbers with real part greater or equal then $1$ and there represents a smooth function. One possible way around this is to apply the Selberg--Delange's method (or better, a suitable generalization for multiplicative functions proved by Granville and Koukoulopoulos \cite{GK}) to compute asymptotically the sum $\sum_{n\le N}f(n) c_q(n)$. Clearly, the product $f(n)c_q(n)$ is not a multiplicative function and this is an obstruction to an immediate application of such a result. To overcome this, the idea is to break the above sum down into smaller pieces that are easier to understand and in particular, to reduce ourselves to apply the Selberg--Delange's method to the much more manageable average of $f$ over a coprimality  condition. More precisely, we notice that
\begin{equation}
\label{firstmanipulation*}
\sum_{n\leq N}f(n)c_q(n)=\sum_{\substack{b\leq N\\ p|b\Rightarrow p|q}}f(b)c_q(b)\sum_{\substack{a\leq N/b\\ (a,q)=1}}f(a),
\end{equation}
using the substitution $n=ab$, with $(a,q)=1$ and $b=n/a$, which is unique, and properties of the Ramanujan sums. Thanks to this decomposition we can apply the Selberg--Delange's method to the simple average of $f(a)$ over the coprimality condition $(a,q)=1$. 

Since we are seeking for a lower bound of the integral in \eqref{eq81}, we can restrict the sum over $q$ to a subset of integers between $KQ_0$ and $R$ satisfying certain conditions that will help us to compute a lower bound for \eqref{firstmanipulation*}. In the following we set for future reference all the conditions we ask $q$ to be subject to. Let $\eps$ be a small positive real number to be chosen at the end in terms of $\delta,\kappa,A_1,A_2$ and the implicit constants in \eqref{mainstatistic1}--\eqref{mainstatistic2}. Moreover, let $A,B,C$ and $D$ positive real constants to be chosen in due course. Then we ask that:
\begin{enumerate}
\item $q\in [KQ_0,N^{1/2-3\delta/4}]$ squarefree.
\item $\omega(q)\leq A\log\log N$, where $\omega(.)$ is the prime divisors counting function. Equivalently, we are asking that the number of prime factors of $q$ is bounded by the expected one.
\item $q=tss'$, with 
\begin{enumerate}
\item $p|s\Rightarrow p\leq (\log N)^{B}$, i.e. $s$ is $(\log N)^B$--smooth.
\item $s'\leq N^{\eps}$, with $p|s'\Rightarrow p>(\log N)^B$, i.e. $s'$ smaller than a suitably small power of $N$ and $(\log N)^B$--rough.
\item $s'\in\mathcal{A}'$, where 
\begin{equation*}
\label{mainassumption}
\mathcal{A}':=\bigg\{s': \sum_{p|s'}\frac{\log p}{\min\{|f(p)-1|,1\}}\leq \frac{\eps\log N}{\kappa}\bigg\}.
\end{equation*}
It is equivalent to ask that $f$ is \emph{never} too close to $1$ on several prime factors of $s'$.
\item $t$ a prime in $[N^{1/2-3\delta/4-\eps},N^{1/2-3\delta/4-\eps/2}]$, supposing the existence of a \textit{unique large prime factor} in the prime factorization of $q$.
\end{enumerate}
\item for any prime $p|q$ we have $p>C$, i.e. $q$ does not have any very small prime factor.
\item
\begin{enumerate}
\item $|f(p)-1|>1/\sqrt{\log\log N}$, if $p|ss'$, i.e. on those primes $f$ is never too close to $1$.
\item $p>C/|f(p)-1|$, for any $p|ss'$.
\item $f(t)\neq 1$.
\end{enumerate}
\item To avoid the scenario in which $q$ has lots of small prime factors, we require $q\in\mathcal{A},$ where 
$$\mathcal{A}:=\bigg\{q: \sum_{p|q}\frac{(\log p)^{A_1+1}}{p^{3/4}}\leq D\bigg\}.$$
\end{enumerate}
Under the restrictions $(1)$, $(4)$ and $(6)$ on $q$, we can develop the average of $f$ under the coprimality condition $(a,q)=1$ as 
\begin{equation}
\label{SDasymptotic*}
\sum_{\substack{a\leq N/b\\ (a,q)=1}}f(a)=\frac{N}{b}(\log(N/b))^{\alpha-1}\left(\sum_{j=0}^{J}\frac{\lambda_j}{(\log(N/b))^{j}}+O((\log(N/b))^{\kappa-A_1-1}(\log\log N))\right),
\end{equation}
for any $b\leq N$, where $J=\lfloor A_1\rfloor$ and $\lambda_j=\lambda_j(f,\alpha,q)$ are certain coefficients with a controlled growth on average over $q$. Here we see the need to restrict ourselves to values of $\alpha\not\in-\mathbb{N}\cup \{0\}$. Indeed, each $\lambda_j$ turns out to be a multiple of the reciprocal of $\Gamma(\alpha-j)$ and $-\mathbb{N}\cup \{0\}$ is exactly the set of poles of the Gamma function. If $\alpha$ belongs to it, all the terms in the sum over $j$ vanish and we no longer have an asymptotic expansion for the average of $f$. Losing control on \eqref{SDasymptotic*} does not allow us to find an explicit lower bound for the variance \eqref{variancedef}, with the method developed here. 

Plugging \eqref{SDasymptotic*} into \eqref{firstmanipulation*}, we are basically left to evaluate the truncated Dirichlet series 
$$\sum_{\substack{b\leq N:\\ p|b\Rightarrow p|q}}\frac{f(b)c_q(b)}{b}(\log(N/b))^{\tilde{\alpha}},$$
with $\tilde{\alpha}\in\{\alpha-1,\alpha-2,\dots,\alpha-J-1\}$. Heuristically, we might expect they behave like
\begin{equation}
\label{heuristicform}
(\log N)^{\tilde{\alpha}}\sum_{b|q}f(b)\mu(q/b).
\end{equation}
Indeed, the above sums have their major contribution coming from the \textit{squarefree} $b\leq N$ sharing their prime factors only with $q$. On those integers they reduce to a sum over the divisors of $q$. Inserting now \eqref{mainpropc_q}, swapping summations and assuming that $(\log(N/b))^{\tilde{\alpha}}$ may be replaced by $(\log N)^{\tilde{\alpha}}$ on average over $b\leq N$, we arrive for them to an expression like
$$(\log N)^{\tilde{\alpha}}\sum_{b|q}f(b)\mu(q/b)\sum_{k|\frac{q}{b}}\frac{f(k)}{k}.$$
Here the innermost sum is
$$\sum_{k|\frac{q}{b}}\frac{f(k)}{k}=\prod_{p|\frac{q}{b}}\bigg(1+\frac{f(p)}{p}\bigg),$$
which thanks to the uniform boundedness of $f$ on primes does not affect our computation on average over $q$, thus leaving us essentially with understanding the behaviour of \eqref{heuristicform}. However, this is a very rough prediction based on arithmetic properties of the Ramanujan sums and of our generalized divisor functions. Technically speaking, there are several details to take into account which make the evaluation of those sums quite complicated, as for instance the presence of possibly very large divisors of $q$, for which the value of $(\log(N/b))^{\tilde{\alpha}}$ cannot be approximated with $(\log N)^{\tilde{\alpha}}$. Thus, the idea is to exploit the structure of the Ramanujan sums, leading to a useful decomposition of these truncated series given by splitting the integers $q=rs\leq N^{1/2}$, with $s$ supported only on small prime numbers as in condition $(3.a)$, whereas $r$ will be written as $r=ts'$ later on, with $t$ and $s'$ subject to conditions $(3.b)-(3.d)$. In view of this factorization and using the Dirichlet hyperbola method we have the following identity:
\begin{align}
\label{multofc_q*}
\sum_{\substack{b\leq N\\ p|b\Rightarrow p|q}}\frac{f(b)c_q(b)}{b}(\log(N/b))^{\tilde{\alpha}}&=\sum_{\substack{b_1\leq \sqrt{N}\\ p|b_1\Rightarrow p|r}}\frac{f(b_1)c_r(b_1)}{b_1}\sum_{\substack{b_2\leq N/b_1\\ p|b_2\Rightarrow p|s}}\frac{f(b_2)c_s(b_2)}{b_2}(\log(N/b_1 b_2))^{\tilde{\alpha}}\\
&+\sum_{\substack{b_2\leq \sqrt{N}\\ p|b_2\Rightarrow p|s}}\frac{f(b_2)c_s(b_2)}{b_2}\sum_{\substack{\sqrt{N}<b_1\leq N/b_2\\ p|b_1\Rightarrow p|r}}\frac{f(b_1)c_r(b_1)}{b_1}(\log(N/b_1 b_2))^{\tilde{\alpha}},\nonumber
\end{align}
since by multiplicativity of $c_q(n)$ as function of $q$ and definition of $r,s$ we have 
$$c_q(b)=c_r(b)c_s(b)=c_r(b_1)c_s(b_2).$$ 
We will show that the contribution from the second double sum on the right hand side of \eqref{multofc_q*} is negligible, because the innermost sum there is a tail of a convergent series. We are left then with finding an estimate for the first one. Regarding the innermost sum there, this can be done by expanding the $\tilde{\alpha}$-power of the logarithm $\log(N/b_1b_2)$, using the generalized binomial theorem. In this way we obtain a sum of successive derivatives of the Dirichlet series in question, which can be handled by means of several applications of the Fa\`a di Bruno's formula, which is a combinatorial expression for the derivative of the composition of two functions (see for instance Roman's paper \cite{R}). It remains to estimate the outermost sum twisted again with a fractional power of $\log(N/b_1)$. In order to compute this we insert a key hypothesis on the structure of $q$, i.e. to be divisible by an extremely large prime number of roughly the size of $\sqrt{N}$, as in condition $(3.d).$ Indeed, it seems crucial to avoid the situation in which $r$ has several large divisors, thus gaining more control on the factor $\log(N/b_1)$.  This is another main difference with the approach employed in \cite{HS}, where the restriction on $q$ consisted only on taking $N^{\eps}$--smooth numbers, for a carefully chosen small $\eps>0$. Under our assumption, the aforementioned sums can be handled by using the multinomial coefficient formula, which gives the expansion for a positive integer power of a multinomial sum (see for example Netto \cite{N}).

We will end up with
$$|\sum_{n\leq N}f(n)c_q(n)|\gg |c_0|\frac{N(\log N)^{\Re(\alpha)-1}}{|\Gamma(\alpha)|}|(f\ast \mu)(q)|,$$
where $c_0$ is as in the statement of Theorem \ref{thmalpha}. Here we indicate with $f\ast g$ the Dirichlet convolution between any two functions $f$ and $g$. Inserting this final lower bound in \eqref{eq81}, plugging this in Proposition \ref{Prop5.1} and estimating the remaining minor quantities, we deduce that
\begin{equation}
\label{finalestimlowerboundvariance11}
V(Q,f) \gg |c_0|\frac{QN(\log N)^{-\beta+2(\Re(\alpha)-1)}}{|\Gamma(\alpha)|^2}\left(\sum_{q\leq N}'\frac{|f\ast \mu(q)|^2}{q}\right)^2
\end{equation}
where the sum $\sum^{'}$ is over all the integers $q$ satisfying the restrictions $(1)-(6)$. The proof now ends after showing that 
$$\sum_{q\leq N}'\frac{|f\ast \mu(q)|^2}{q}\gg \beta(\log N)^{\beta}\ \textrm{and}\ \sum_{n\leq N}|f(n)|^2\ll N(\log N)^{\beta+2(\Re(\alpha)-1)},$$
so that to deduce
$$V(Q,f) \gg \bigg|\frac{c_0\beta}{\Gamma(\alpha)}\bigg|^2QN(\log N)^{\beta+2(\Re(\alpha)-1)}\gg \bigg|\frac{c_0\beta}{\Gamma(\alpha)}\bigg|^2 Q \sum_{n\leq N}|f(n)|^2.$$
\section{Proof of corollaries of Theorem \ref{thmalpha}}
This section is devoted to the proof of some applications of our main theorem that have already been mentioned in the introduction.
\begin{proof}[Proof of Theorem \ref{cor1}]
The case $\alpha=0$ has already been handled by Proposition \ref{thm0,1}. When instead $\alpha\not\in -\mathbb{N}\cup \{0,1\}$, Theorem \ref{thmalpha} can be applied.

Notice than that $c_0=1$ as we can see from the following identities:
$$\sum_{k\geq 0}\frac{d_{\alpha}(p^k)}{p^k}=\bigg(1-\frac{1}{p}\bigg)^{-\alpha},$$
for any prime $p$ and any complex number $\alpha$. Thus, assumption \eqref{gammaassumpt} is satisfied.

Moreover, since $\alpha\not\in -\mathbb{N}\cup \{0,1\}$ and $\beta=|\alpha-1|^2>0$ are constant and equations \eqref{mainstatistic1}--\eqref{mainstatistic2} are satisfied with any $A_1,A_2>0$, also the relations \eqref{relationbetaA_1} hold. 

Theorem \ref{thmalpha} now gives the thesis.
\end{proof}
\begin{proof}[Proof of Corollary \ref{sumoftwosquares}]
The function $\textbf{1}_{S}$ is multiplicative and satisfies \eqref{mainstatistic1} and \eqref{mainstatistic2} with $\alpha=\beta=1/2$ and any $A_1,A_2>0$, by the prime number theorem for the arithmetic progression $1\ (\bmod\ {4})$ (see e.g. \cite[Corollary 11.20]{MV}). Thus, the inequalities \eqref{relationbetaA_1} hold. 

By Mertens' theorem for the arithmetic progressions $1\ (\bmod\ {4})$ and $3\ (\bmod\ {4})$ (see e.g. \cite[Corollary 4.12]{MV}) we have $c_0\gg 1$, thus implying  assumption \eqref{gammaassumpt}.

We again conclude by using \eqref{variance1}.
\end{proof}
\begin{proof}[Proof of Theorem \ref{uniformalpha}]
We let $f(n)=d_{\alpha_N}(n)$, with $\alpha_N=1+1/R(N)$ as in the statement. By choosing $C$ large enough, we may assume $1/2\leq \alpha_N^2\leq 3/2$. Using \cite[Theorem 1]{GK} with $A_1=4$ we see that
\begin{equation}
\label{firstasymptoticforf(n)^2}
\sum_{n\leq N}f^2(n)=\sum_{j=0}^{3}\frac{c_j(\alpha_N ^2)}{\Gamma(\alpha_N ^2-j)}N(\log N)^{\alpha_N ^2-j-1}+O\bigg(\frac{N\log\log N}{\log N}\bigg),
\end{equation} 
with coefficients $c_j(\alpha_N^2)$ defined by
$$c_j(\alpha_N^2)=\frac{d^j}{dz^j}\frac{(z-1)^{\alpha_N ^2}F(z)}{z}\bigg|_{z=1},$$
where $F(z)$ is the Dirichlet series of $f^2(n)$. By adapting the proof of the $C^4$-continuation of $F(z)(z-1)^{\alpha_N^2}$ to the half-plane $\Re(z)\geq 1$ at the start of \cite[section 2]{GK}, we can easily check that every $c_j(\alpha_N^2)$ is uniformly bounded, for every $1/2\leq \alpha_N^2\leq 3/2$.

Moreover, for any $j\geq 1$
$$\Gamma(\alpha_N^2-j)=\frac{\Gamma(\alpha_N^2)}{(\alpha_N^2-j)(\alpha_N^2-j+1)\cdots(\alpha_N^2-1)},$$
from which we deduce that
$$ |\Gamma(\alpha_N^2)|\asymp1\ \textrm{and}\  |\Gamma(\alpha_N^2-j)|\gg 1$$
thanks to the continuity of $\Gamma(z)$ and our hypothesis on $\alpha_N$. Hence, we conclude that
\begin{equation}
\label{asymptoticforf(n)^2}
\sum_{n\leq N}f^2(n)\gg N(\log N)^{\alpha_N^2-1}= N\exp\bigg(\bigg(2+\frac{1}{R(N)}\bigg)\frac{\log\log N}{R(N)}\bigg),
\end{equation}
if $N$ large enough, where we also used that 
$$c_0(\alpha_N^2)=\prod_{p}\bigg(1+\frac{f(p)^2}{p}+\frac{f(p^2)^2}{p^2}+\cdots\bigg)\bigg(1-\frac{1}{p}\bigg)^{\alpha_N ^2}\gg 1.$$
Similarly, we have
$$c_0=\prod_{p\leq N}\bigg(1+\frac{f(p)}{p}+\frac{f(p^2)}{p^2}+\cdots\bigg)\bigg(1-\frac{1}{p}\bigg)^{\alpha_N}\gg 1.$$
We notice that relations \eqref{relationbetaA_1} are trivially satisfied, with $\beta=1/R(N)^2$, since $R(N)$ is allowed to grow at most as a large power of $\log N$ and we can take $A_1,A_2$ arbitrarily large. 

An application of Theorem \ref{thmalpha}, together with equation \eqref{asymptoticforf(n)^2}, leads to \eqref{varianceuniform}, since again by the continuity of the Gamma function we have $|\Gamma(\alpha_N)|\ll 1$.
\end{proof} 
\section{Mean value of multiplicative functions under a coprimality condition}
In this section we will show how to handle averages of multiplicative functions satisfying \eqref{mainstatistic1} under a coprimality condition. Since we are going to use the full strength of \cite[Theorem 1]{GK}, we report it here for the sake of readiness and in a form more suitable for our purposes.
\begin{thm}
\label{GKthm}
Let $f$ be a multiplicative function satisfying \eqref{mainstatistic1} and such that there exists $\kappa>1$ with $|f(n)|\leq d_\kappa(n)$, for any $N$-smooth positive integer $n$. Let $J$ be the largest integer $<A_1$ and the coefficients $c_j=c_j(f,\alpha)$ defined by
$$c_j=\frac{1}{j!}\frac{d^j}{dz^j}\bigg(\zeta_N(z)^{-\alpha}F(z)\frac{((z-1)\zeta(z))^{\alpha}}{z}\bigg)_{z=1},\ \textrm{for\ any}\ j\leq J,$$
with
$$F(z)=\sum_{\substack{n:\\ p|n\Rightarrow p\leq N}}\frac{f(n)}{n^z},\ \ \ \zeta_N(z)=\sum_{\substack{n:\\ p|n\Rightarrow p\leq N}}\frac{1}{n^z}.$$
Then we have
\begin{equation}
\label{maineqGKthm}
\sum_{n\leq x}f(n)=x\sum_{j=0}^{J}c_j\frac{(\log x)^{\alpha-j-1}}{\Gamma(\alpha-j)}+O(x(\log x)^{\kappa-1-A_1}(\log\log x)),\ \ \ (2\leq x\leq N).
\end{equation} 
The big-Oh constant depends at most on $\kappa, A_1$ and the implicit constant in \eqref{mainstatistic1}. The dependence on $A_1$ comes from both its size and its distance from the nearest integer. Moreover, the condition $|f(n)|\leq d_{\kappa}(n)$ can be relaxed to the following two ones on average over prime powers
\begin{align}
\label{extraconditions}
&\sum_{\substack{p\leq x}}\frac{|f(p)|\log p}{p}\leq \kappa \log x+O(1)\ \ \ (2\leq x\leq N);\\
&\sum_{\substack{p\leq x\\ j\geq 1}}\frac{|f(p^j)|^2}{p^j}\leq \kappa^2\log\log x+O(1)\ \ \ (2\leq x\leq N),\nonumber
\end{align}
where the big-Oh terms here depend only on $\kappa$.
\end{thm} 
\begin{proof}
We first note that \cite[Theorem 1]{GK} gives an asymptotic for the mean value of multiplicative functions for which we know their behaviour on average over all the prime numbers, including those much larger than $N$, whereas here we are interested only in the value of $f(p^k)$, for prime powers $p^k$ with $p\leq N$. However, we can freely replace $f$ with the function equal to $f$ itself on such prime powers and such that
$$f(p^k)=d_\alpha(p^k),\ \textrm{for\ any}\ p>N\ \textrm{and}\ k\geq 1.$$
Then Theorem \ref{GKthm} readily follows from \cite[Theorem 1]{GK}. 
Indeed, it is clear that the statistic \cite[equation 1.2]{GK} corresponds to \eqref{mainstatistic1} here. Moreover, the condition $|f(n)|\leq d_{\kappa}(n)$, for every $n\leq N$, trivially translates to our condition only on $N$-smooth numbers, since it is equivalent to the corresponding one on prime powers. Same considerations for the statistics \eqref{extraconditions}, which are slightly weaker than the corresponding conditions \cite[eq. (7.1)--(7.2)]{GK}.

The only main difference is in the representation of the coefficients $c_j$. Indeed, in \cite{GK} such coefficients are defined as
$$c_j=\frac{1}{j!}\frac{d^j}{dz^j}\frac{(z-1)^{\alpha}\tilde{F}(z)}{z}\bigg|_{z=1},$$
with $\tilde{F}(z)$ the Dirichlet series of $f(n)$. Here we multiply and divide the above expression by $\zeta(z)^{-\alpha}$ and notice that $$\zeta(z)^{-\alpha}\tilde{F}(z)=\zeta_N(z)^{-\alpha}F(z),$$
where $\zeta_N(z)$ and $F(z)$ are defined as in the statement of the theorem. Since the function $((z-1)\zeta(z))^{\alpha}$ is a holomorphic function on $\Re(z)\geq 1$, for any $\alpha\in\mathbb{C}$, we see that each coefficient $c_j$ is basically the $j$-th derivative of an Euler product. In particular, we have 
\begin{equation*}
c_0=\prod_{p\leq N}\bigg(1+\frac{f(p)}{p}+\frac{f(p^2)}{p^2}+\cdots\bigg)\bigg(1-\frac{1}{p}\bigg)^{\alpha}.
\end{equation*}
Potentially, the coefficients $c_j$ could grow together with $N$ and $\alpha$. However, the next lemma shows that under our hypotheses on $f$ they are indeed uniformly bounded. 
\begin{lem}
\label{lemderivativeseulerproduct0}
Let $f$ be a multiplicative function satisfying \eqref{mainstatistic1} for some $\alpha\in\mathbb{C}$ and such that $|f(n)|\leq d_{\kappa}(n)$, for some $\kappa>1$ and every $N$-smooth number $n$. Then
$$c_j\ll 1,\ \textrm{for\ any}\ 0\leq j\leq J,$$
where the implicit constant may depend on $\kappa,A_1$ and the implicit constant in \eqref{mainstatistic1} and we take $N$ large enough with respect to these parameters.
\end{lem}
\begin{proof}
It is clear that $c_0$ is uniformly bounded in $N$ and $\alpha$. Indeed, since by hypothesis $f(p^k)\leq d_{\kappa}(p^k)$, for any prime $p\leq N$ and integer $k\geq 0$, either $c_0=0$ or we can write
$$c_0=\exp\bigg(\sum_{p\leq N}\bigg(\frac{f(p)-\alpha}{p}+O_\kappa\bigg(\frac{1}{p^2}\bigg)\bigg)\bigg)\asymp 1,$$
by partial summation from \eqref{mainstatistic1}, where the implicit constant may depend on $\kappa,A_1$ and the implicit constant in \eqref{mainstatistic1} and we take $N$ large enough with respect to these parameters.

It is not that straightforward though to show that each $c_j$, for $j\geq 1$, is uniformly bounded in $N$ and $\alpha$.
To this aim we employ the following procedure borrowing some ideas from the discussion in \cite[Section 2]{GK}. Since $c_j$ is the $j$-th derivative at $z=1$ of the product between $H(z)=\zeta_N(z)^{-\alpha}F(z)$ and $Z_\alpha(z)=((z-1)\zeta(z))^{\alpha}/z$, we only need to show that all the $l$-derivatives of $H(z)$ at $z=1$ are uniformly bounded, for any $l\leq J$. Indeed, this is certainly true for all the $m$-derivatives of $Z_\alpha(z)$ at $z=1$, for any $m\leq J$, and we have
$$(H(z)Z_\alpha(z))^{(j)}(1)=\sum_{l+m=j}\binom{j}{l}H^{(l)}(1)Z_\alpha^{(m)}(1).$$
We have $F(z)=F_1(z)F_2(z)$, where
$$F_1(z)=\sum_{\substack{n\geq 1: \\ p|n\Rightarrow p\leq N}}\frac{d_f(n)}{n^z}=\prod_{p\leq N}\bigg(1-\frac{1}{p^z}\bigg)^{-f(p)},$$
where $d_f(n)$ is the multiplicative function satisfying 
$$d_f(p^k)=\binom{f(p)+k-1}{k}$$
over all the prime powers $p^k$, with $p\leq N$, and 
$$F_2(z)=\sum_{\substack{n\geq 1: \\ p|n\Rightarrow p\leq N}}\frac{R_f(n)}{n^z},$$
with $f(n)=d_f\ast R_f(n).$ Since $R_f$ is supported only on square-full integers, $|R_f(n)|\leq d_{2\kappa}(n)$ and for every $l\geq 0$ we have
$$F_2^{(l)}(1)=\sum_{\substack{n\geq 1: \\ p|n\Rightarrow p\leq N}}\frac{R_f(n)(-\log n)^l}{n},$$
it is clear that all the derivatives of $F_2$ at $z=1$ are uniformly bounded. 

Arguing similarly as before, we are left with showing that all the derivatives of 
$$H_1(z)=\zeta_N(z)^{-\alpha}F_1(z)=\prod_{p\leq N}\bigg(1-\frac{1}{p^z}\bigg)^{-f(p)+\alpha}$$
at $z=1$ are uniformly bounded.

To this aim, for any $1\leq l\leq J$ we use the Fa\`a di Bruno's formula \cite[p. 807, Theorem 2]{R} to find
\begin{equation}
\label{firstfaadibruno}
H_1^{(l)}(1)=H_1(1)l!\sum_{m_1+2m_2+...+lm_l=l}\frac{\prod_{i=1}^{l}(h^{(i-1)}(1))^{m_i}}{1!^{m_1}m_1!2!^{m_2}m_2!\cdots l!^{m_l}m_l!},
\end{equation}
where 
\begin{align*}
h(z)=\frac{H_1'}{H_1}(z)&=\sum_{p\leq N}(\alpha-f(p))\log p\sum_{k=0}^{\infty}\frac{d_{\alpha-f}(p^k)}{p^{(k+1)z}},
\end{align*}
where as before $d_{\alpha-f}(n)$ is the multiplicative function satisfying 
$$d_{\alpha-f}(p^k)=\binom{\alpha-f(p)+k-1}{k}$$
over all the prime powers $p^k$, with $p\leq N$. From this we deduce that
\begin{align*}
h^{(i-1)}(1)&=\sum_{p\leq N}(\alpha-f(p))(\log p)^{i}\sum_{k=0}^{\infty}\frac{d_{\alpha-f}(p^k)(-k-1)^{i-1}}{p^{(k+1)}}\\
&=(-1)^{i-1}\sum_{p\leq N}\frac{(\alpha-f(p))(\log p)^{i}}{p}+\sum_{p\leq N}(\alpha-f(p))(\log p)^{i}\sum_{k=1}^{\infty}\frac{d_{\alpha-f}(p^k)(-k-1)^{i-1}}{p^{(k+1)}}\\
&=(-1)^{i-1}\sum_{p\leq N}\frac{(\alpha-f(p))(\log p)^{i}}{p}+O_{i,\kappa}\bigg(\sum_{p\leq N}\frac{(\log p)^{i}}{p^2}\bigg)\\
&=(-1)^{i-1}\sum_{p\leq N}\frac{(\alpha-f(p))(\log p)^{i}}{p}+O_{i,\kappa}(1),
\end{align*}
where we used that $|d_{\alpha-f}(n)|\leq d_{2\kappa}(n)$. The last sum above can be estimated with a partial summation argument from \eqref{mainstatistic1}, for any $1\leq i\leq J$. We thus find $|h^{(i-1)}(1)|\ll 1$, with an implicit constant depending on $\kappa,A_1,i$ and the implicit constant in \eqref{mainstatistic1}. Inserting this into \eqref{firstfaadibruno} also gives $H^{(l)}(1)\ll 1$, with now an implicit constant depending on $\kappa,A_1,l$ and the implicit constant in \eqref{mainstatistic1}, since we can prove that $H(1)$ is uniformly bounded in much the same way as we did for $c_0$. Together with previous considerations this concludes the proof of the lemma.
\end{proof}
The previous lemma shows that the coefficients in the asymptotic expansion \eqref{maineqGKthm} are well defined and indeed uniformly bounded independently of $N$ and $\alpha$, for given $A_1$. This will also turn out to be useful in several future applications of Theorem \ref{GKthm}, in which in order to make sure that the first term in the asymptotic expansion \eqref{maineqGKthm} dominates, we will need a careful control on the other terms. 

Together with previous observations, it proves this version of \cite[Theorem 1]{GK}.
\end{proof}
We are going to apply the above theorem to prove its slight variation about sums restricted to those integers up to $x$ coprime with a parameter $q$, satisfying certain suitable properties.
\begin{thm}
\label{GKvariant}
Let $f(n)$ be a multiplicative function with complex values such that there exists $\kappa>1$ with $|f(n)|\leq d_{\kappa}(n)$, for any $N$-smooth positive integer $n$, and satisfying \eqref{mainstatistic1} with $\alpha\in\mathbb{C}\setminus\{\{0\}\cup -\mathbb{N}\}$. Moreover, suppose that $q$ is a positive squarefree number smaller than $N$ satisfying condition $(4)$, i.e. for any prime $p|q$ we have $p>C$, where $C>\kappa^2$ will be chosen later on in terms of $\delta,\kappa,A_1$ and the implicit constant in \eqref{mainstatistic1}.
Then for any $4\leq x\leq N$ we have
\begin{align}
\label{eqGKvariant}
\sum_{\substack{n\leq x\\ (n,q)=1}}f(n)&=x(\log x)^{\alpha-1}\sum_{j=0}^{J}\frac{\lambda_j}{(\log x)^{j}}+O(|\tilde{G}_q^{(2\lfloor A_1\rfloor+2)}(1)|x(\log x)^{\kappa-A_1-1}(\log\log x))\\
&+O\bigg(x^{3/4}\sum_{d|q}\frac{d_{\kappa}(d)}{d^{3/4}}\bigg),\nonumber
\end{align}
where $J$ is the largest integer $<A_1$, $\lfloor A_1\rfloor$ is the integer part of $A_1$ and we define 
$$\lambda_j=\lambda_j(f,\alpha,q)=\frac{1}{\Gamma(\alpha-j)}\sum_{l+h=j}\frac{(H_q^{-1})^{(h)}(1)c_l}{h!},$$
with
$$H_q(z)=\prod_{p|q}\bigg(1+\frac{f(p)}{p^z}+\frac{f(p^2)}{p^{2z}}+\cdots\bigg)\ \textrm{and}\ \tilde{G}_q(z)=\prod_{p|q}\bigg(1+\frac{|f(p)|}{p^z}\bigg)$$
on $\Re(z)\geq 1$. Here the big-Oh constant depends on $\kappa,A_1$ and the implicit constant in \eqref{mainstatistic1}.
\end{thm}
\begin{proof}
To begin with, let us define an auxiliary multiplicative function $\tilde{f}$ such that
$$\tilde{f}(p^{j})=\left\{ \begin{array}{ll}
        f(p^{j}) & \mbox{if $p\nmid q$};\\
        f(p)^j & \mbox{otherwise}.\end{array} \right.$$ 
Then we may rewrite the sum in question as
\begin{align}
\label{removingcoprimality}
\sum_{\substack{n\leq x\\ (n,q)=1}}\tilde{f}(n)=\sum_{\substack{n\leq x}}\tilde{f}(n)\sum_{d|n,d|q}\mu(d)=\sum_{d|q}\mu(d)\sum_{\substack{n\leq x\\ d|n}}\tilde{f}(n)&=\sum_{d|q}\mu(d)\sum_{\substack{k\leq x/d}}\tilde{f}(dk)\nonumber\\
&=\sum_{\substack{d|q\\ d\leq x/2}}\mu(d)\sum_{\substack{k\leq x/d}}\tilde{f}(dk)+\sum_{\substack{d|q\\ x/2<d\leq x}}\mu(d)\tilde{f}(d).
\end{align}
The completely multiplicative structure of $\tilde{f}$ on the numbers divisible only by prime factors of $q$ allows us to rewrite the first double sum in \eqref{removingcoprimality} as
\begin{equation}
\label{usingtildef}
\sum_{\substack{d|q\\ d\leq x/2}}\mu(d)\tilde{f}(d)\sum_{\substack{k\leq x/d}}\tilde{f}(k).
\end{equation}
Moreover, since $\tilde{f}$ equals $f$ on the primes, we have
$\sum_{p\leq x}\tilde{f}(p)\log p=\sum_{p\leq x}f(p)\log p$
and it is not difficult to show that the two conditions \eqref{extraconditions} hold for $\tilde{f}$ as well, if $C>\kappa^2$. Thus, an application of Theorem \ref{GKthm} leads to an evaluation of \eqref{usingtildef} as
\begin{align}
\label{usingGKthm}
&=x\sum_{l=0}^{J}\frac{\tilde{c}_l}{\Gamma(\alpha-l)}\sum_{\substack{d|q\\ d\leq x/2}}\frac{\mu(d)\tilde{f}(d)}{d}(\log(x/d))^{\alpha-l-1}\\
&+O\bigg(x\sum_{\substack{d|q\\ d\leq x/2}}\frac{|\tilde{f}(d)|}{d}(\log(x/d))^{\kappa-A_1-1}(\log\log x)\bigg),\nonumber
\end{align}
where analogously to the definition of $c_l$ we define
$$\tilde{c}_l= \frac{1}{l!}\frac{d^l}{dz^l}\bigg(\zeta_N(z)^{-\alpha}\tilde{F}(z)\frac{((z-1)\zeta(z))^{\alpha}}{z}\bigg)_{z=1},$$
with 
$$\tilde{F}(z):=G_q(z)^{-1}\prod_{\substack{p\leq N\\ p\nmid q:}}\sum_{k=0}^{\infty}\frac{f(p^{k})}{p^{kz}}\ \textrm{and}\ G_q(z):=\prod_{p|q}\bigg(1-\frac{f(p)}{p^z}\bigg).$$
The second double sum in \eqref{removingcoprimality} instead is upper bounded by
\begin{equation}
\label{largedivofq}
\ll x^{3/4}\sum_{d|q}\frac{|\tilde{f}(d)|}{d^{3/4}}\leq x^{3/4}\sum_{d|q}\frac{d_{\kappa}(d)}{d^{3/4}},
\end{equation}
since $q$ is squarefree. We see that we may rewrite $\tilde{F}(z)$ as $F(z)G_q(z)^{-1}H_q(z)^{-1}$. 

Hence, $\tilde{c}_l$ will be
$$\tilde{c}_l=\sum_{k=0}^{l}\frac{\frac{d^k}{dz^k}(G_q(z)^{-1}H_q(z)^{-1})|_{z=1}}{k!}c_{l-k}.$$
By Lemma \ref{lemderivativeseulerproduct0} each $c_{l-k}(\alpha)$ is uniformly bounded by a constant depending on $\kappa,A_1$ and the implicit constant in \eqref{mainstatistic1}. The coefficients $\tilde{c}_l$ may potentially depend on $q$. However, we have
$$G_q(z)H_q(z)=\prod_{p|q}\bigg(1-\frac{f(p)}{p^z}\bigg)\bigg(1+\frac{f(p)}{p^z}+\frac{f(p^2)}{p^{2z}}+\cdots\bigg)=\prod_{p|q}\bigg(1+O_{\kappa}\bigg(\frac{1}{p^{2\Re(z)}}\bigg)\bigg),$$
as we can see from: $|f(p^j)-f(p^{j-1})f(p)|\leq (\kappa+1)(d_{\kappa}(p^{j})+d_{\kappa}(p^{j-1})),$
for any $j\geq 2$. 

We deduce that $G_q(z)H_q(z)$ defines a non-vanishing analytic function on $\Re(z)\geq 1$ and so does its inverse. This shows the possibility to estimate the coefficient $\tilde{c}_l$ with a bound free on the dependence of $q$. Another way to show this could be to argue as in the proof of Lemma \ref{lemderivativeseulerproduct0}, because $(G_q(z)H_q(z))^{-1}$ coincides with the Dirichlet series of a function with a controlled growth and supported only on square-full integers.

Let us now focus on studying the sums over $d$ in the main term of \eqref{usingGKthm}. By the generalized binomial expansion (see e.g. the first paragraph in chapter II.5 of \cite{T}), we find for any $0\leq l\leq J$
\begin{align}
\label{usingbinexp}
\sum_{\substack{d|q\\ d\leq x/2}}\frac{\mu(d)\tilde{f}(d)}{d}(\log(x/d))^{\alpha-l-1}&=(\log x)^{\alpha-l-1}\sum_{h=0}^{J-l}\frac{\binom{\alpha-l-1}{h}(-1)^h}{\log^h x}\sum_{\substack{d|q\\ d\leq x/2}}\frac{\mu(d)\tilde{f}(d)}{d}\log^h d\\
&+O_{\kappa,J}\bigg((\log x)^{\Re(\alpha)-J-2}\sum_{d|q}\frac{|\tilde{f}(d)|}{d}(\log d)^{J-l+1}\bigg).\nonumber
\end{align}
Completing the above sums to all the divisors of $q$ gives an error in \eqref{usingbinexp} of at most
$$\ll_{\kappa,J}2^E(\log x)^{\Re(\alpha)-l-1-E}\sum_{d|q}\frac{|\tilde{f}(d)|}{d}(\log d)^{J-l+E},$$
for any $E>0$, since $x/2\geq \sqrt{x}$ on $x\geq 4$. Similarly, the error term in \eqref{usingGKthm} can be estimated with
\begin{align}
\label{errorforlarged}
\ll_{\kappa,A_1}(\log x)^{\kappa-A_1-1}(\log\log x)\bigg(\sum_{\substack{d|q}}\frac{|\tilde{f}(d)|}{d}+\sum_{\substack{d|q}}\frac{|\tilde{f}(d)|}{d}\frac{\log d}{\log x}\bigg).
\end{align}
Next, since $q$ is squarefree we have
$$\sum_{d|q}\frac{|\tilde{f}(d)|}{d^z}=\prod_{p|q}\bigg(1+\frac{|f(p)|}{p^z}\bigg)=\tilde{G}_q(z)$$
and we can rewrite \eqref{usingbinexp} as
\begin{align}
\label{usingbinexp2}
&=(\log x)^{\alpha-l-1}\sum_{h=0}^{J-l}\frac{\binom{\alpha-l-1}{h}}{\log^h x}G_q^{(h)}(1)\\
&+O_{\kappa,J}\bigg((\log x)^{\Re(\alpha)-J-2}|\tilde{G}_q^{(J-l+1)}(1)|\bigg)\nonumber\\
&+O_{\kappa,J}\bigg(2^E(\log x)^{\Re(\alpha)-l-1-E}|\tilde{G}_q^{(J-l+E)}(1)|\bigg)\nonumber
\end{align}
and \eqref{errorforlarged} as
\begin{equation}
\label{errorforlarged2}
\ll_{\kappa,A_1}(\log x)^{\kappa-A_1-1}(\log\log x)\bigg(|\tilde{G}_q(1)|+\frac{|\tilde{G}_q^{(1)}(1)|}{\log x}\bigg).
\end{equation}
Inserting \eqref{usingbinexp2} and \eqref{errorforlarged2} into \eqref{usingGKthm} and rearranging, we have overall found
\begin{align}
\label{maintermGK}
\sum_{\substack{n\leq x\\ (n,q)=1}}f(n)&=x\sum_{j=0}^{J}\frac{(\log x)^{\alpha-j-1}}{\Gamma(\alpha-j)}\sum_{h+l=j}\frac{G_q^{(h)}(1)\tilde{c}_l}{h!}\\
&+O\bigg(x(\log x)^{\Re(\alpha)-J-2}\sum_{l=0}^{J}\frac{|\tilde{c}_l\tilde{G}_q^{(J-l+1)}(1)|}{|\Gamma(\alpha-l)|}\bigg)\nonumber\\ 
&+O\bigg(2^E x(\log x)^{\Re(\alpha)-1-E}\sum_{l=0}^{J}\frac{|\tilde{c}_l\tilde{G}_q^{(J-l+E)}(1)|}{|\Gamma(\alpha-l)|(\log x)^l}\bigg)\nonumber\\
&+O\bigg(x(\log x)^{\kappa-A_1-1}(\log\log x)\bigg(|\tilde{G}_q(1)|+\frac{|\tilde{G}_q^{(1)}(1)|}{\log x}\bigg)\bigg)\nonumber\\
&+O\bigg(x^{3/4}\sum_{d|q}\frac{d_{\kappa}(d)}{d^{3/4}}\bigg)\nonumber
\end{align}
By definition of $\tilde{c}_l$, the $j$-th coefficient in the sum in the main term in the displayed equation above can be rewritten as
\begin{align*}
&\frac{1}{\Gamma(\alpha-j)}\sum_{h=0}^{j}\frac{G_q^{(h)}(1)}{h!}\sum_{k=0}^{j-h}\frac{(G_q(z)^{-1}H_q(z)^{-1})^{(k)}(1)}{k!}c_{j-h-k}\\
&=\frac{1}{\Gamma(\alpha-j)}\sum_{l=0}^{j}c_{j-l}\sum_{k+h=l}\frac{G_q^{(h)}(1)}{h!}\frac{(G_q(z)^{-1}H_q(z)^{-1})^{(k)}(1)}{k!}\\
&=\frac{1}{\Gamma(\alpha-j)}\sum_{l=0}^{j}\frac{c_{j-l}}{l!}(H_q(z)^{-1})^{(l)}(1)\\
&=\frac{1}{\Gamma(\alpha-j)}\sum_{l+h=j}\frac{(H_q^{-1})^{(l)}(1)c_h}{l!}\\
\end{align*}
and in this way is presented as in the statement of the theorem.

Regarding the error term instead, by Lemma \ref{lemderivativeseulerproduct0} and previous considerations, we can prove that all the coefficients $\tilde{c}_l$ are uniformly bounded, thus finding an upper bound of
\begin{align*}
&\ll_{\kappa,J}|\tilde{G}_q^{(J+1)}(1)|x(\log x)^{\Re(\alpha)-J-2}+2^E|\tilde{G}_q^{(J+E)}(1)|x(\log x)^{\Re(\alpha)-1-E}\\
&+(|\tilde{G}_q(1)|+|\tilde{G}_q^{(1)}(1)|)x(\log x)^{\kappa-A_1-1}(\log\log x)+x^{3/4}\sum_{d|q}\frac{d_{\kappa}(d)}{d^{3/4}},
\end{align*}
where we also used the continuity of $\Gamma(\alpha-l)^{-1}$ over the compact set $|\alpha|\leq \kappa$. Moreover, since we have
$$|\tilde{G}_q^{(h)}(1)|=\sum_{d|q}\frac{|\tilde{f}(d)|\log^h d}{d},$$
for any $h\geq 0$, it is clear that $|\tilde{G}_q^{(a)}(1)|\leq |\tilde{G}_q^{(b)}(1)|$, for any $0\leq a\leq b$. Thanks to this last inequality we may simplify the error term in \eqref{maintermGK} further to 
$$\ll_{\kappa,A_1} |\tilde{G}_q^{(2\lfloor A_1\rfloor+2)}(1)|x(\log x)^{\kappa-A_1-1}(\log\log x)+x^{3/4}\sum_{d|q}\frac{d_{\kappa}(d)}{d^{3/4}},$$
if we let $E:=\lfloor A_1\rfloor+2\geq A_1+1+|\Re(\alpha)|-\kappa$. Hence also the error term above is in the form contained in the statement of the theorem, thus concluding its proof.
\end{proof}
\section{The error terms in Proposition \ref{Prop5.1}}
From now on we are going to specialize the function $f$ to be as in the statement of Theorem \ref{thmalpha}, considering $\alpha,\beta,A_1,A_2$ as in \eqref{relationbetaA_1}. To begin with, we start with the following useful lemma.
\begin{lem}
\label{lemmaarithmquadrineq}
Under the usual notation, we have $|\alpha-1|^2\leq \beta+O_{\kappa}((\log N)^{-\min\{A_1,A_2\}})\ll \beta$, if $N$ large enough with respect to $A_1,A_2,\kappa$ and the implicit constants \eqref{mainstatistic1}--\eqref{mainstatistic2}.
\end{lem}
\begin{proof}
If $\alpha=1$, the result is trivial. Assume then $\alpha\neq 1$. By an application of the Cauchy--Schwarz inequality, we have
\begin{align*}
\bigg|\sum_{p\leq N}(f(p)-1)\log p\bigg|^2&\leq \sum_{p\leq N}|f(p)-1|^2\log p\sum_{p\leq N}\log p\\
&=(\beta N+O(N(\log N)^{-A_2}))(N+O(Ne^{-c\sqrt{\log N}})),
\end{align*}
for a suitable $c>0$, by the prime number theorem and equation \eqref{mainstatistic2}. The left hand side of the above inequality instead is $|\alpha-1|^2N^2+O_{\kappa}(N^2(\log N)^{-A_1})$, by \eqref{mainstatistic1}. This implies the thesis if we assume $N$ as in the statement of the lemma and thanks to conditions \eqref{relationbetaA_1}.  
\end{proof}

An application of \cite[Theorem 2.14]{MV} leads to
$$ \sum_{n\leq N}|f(n)|^{2}\ll_\kappa \frac{N}{\log N}\sum_{n\leq N}\frac{|f(n)|^{2}}{n}\leq \frac{N}{\log N}\prod_{p\leq N}\bigg(1+\frac{|f(p)|^2}{p}+\frac{|f(p^2)|^2}{p^2}+\cdots\bigg).$$
By Mertens' theorem we deduce
\begin{equation}
\label{averagefsquare}
\sum_{n\leq N}|f(n)|^{2}\ll_\kappa c_0(|f|^2,\beta+2\Re(\alpha)-1)N(\log N)^{\beta+2\Re(\alpha)-2},
\end{equation}
where 
$$c_0(|f|^2,\beta+2\Re(\alpha)-1)=\prod_{p\leq N}\bigg(1+\frac{|f(p)|^2}{p}+\frac{|f(p^2)|^2}{p^{2}}+\cdots\bigg)\bigg(1-\frac{1}{p}\bigg)^{\beta+2\Re(\alpha)-1}$$
is a positive constant. In particular, it is uniformly bounded in terms of $\kappa,A_1,A_2$ and the implicit constants in \eqref{mainstatistic1}--\eqref{mainstatistic2}, as we may see by applying partial summation from \eqref{mainstatistic1}--\eqref{mainstatistic2} and considering the relations \eqref{relationbetaA_1}.
In conclusion, the first error term in \eqref{estimateprop1} is
\begin{equation}
\label{firsterrortemrprop1.4}
\ll \frac{KN^2(\log N)^{\beta+2(\Re(\alpha)-1)}}{Q_0},
\end{equation}
with the implicit constant depending on all the aforestated parameters.

We now turn to the estimate of the second error term in \eqref{estimateprop1}, but first let us state the following result which we will make use of several times later on. 
\begin{lem}
\label{lemrankinestimate}
For any non-negative multiplicative function $g(n)$ uniformly bounded on the prime numbers by a positive real constant $B$ and such that the sum $S=\sum_{q}g(q)/q$ over all the prime powers $q=p^{k}$, with $k\geq 2$, converges, we have
\begin{equation}
\label{Rankinestimate}
1\ll_{B,S} \sum_{n\leq x}\frac{g(n)}{n}\prod_{p\leq x}\bigg(1+\frac{g(p)}{p}\bigg)^{-1}\ll_{B,S} 1,\ \ \ (x\geq 1).
\end{equation}
\end{lem}
\begin{proof}
This is \cite[Lemma 20]{EK} of Elliott and Kish.
\end{proof}
\begin{prop}
\label{propupperboundfc_d}
We have
\begin{equation}
\label{upperboundfc_d1}
\sum_{q\le Q } \frac{1}{q} \sum_{\substack {d|q \\ d>Q_0}} \frac{1}{\phi(d)} \Big| \sum_{n} f(n) c_d(n)\Big|^2\ll_{\kappa} \frac{N^2(\log N)^{\kappa^2+4\kappa+2}}{Q_0}.
\end{equation}
\end{prop}
\begin{proof}
We initially observe that
\begin{equation}
\label{eq1}
\Big| \sum_{n\leq N} f(n) c_d(n)\Big|\leq \sum_{n\leq N} d_{\kappa} (n) \sum_{e|(n,d)}e\leq \sum_{e|d}e\sum_{\substack{n\leq N\\ e|n}}d_{\kappa}(n),
\end{equation}
by \eqref{mainpropc_q}. Since $\kappa>1$
\begin{equation*}
\sum_{\substack{n\leq N\\ e|n}}d_{\kappa}(n)=\sum_{k\leq N/e}d_{\kappa}(ek)\leq \frac{d_{\kappa}(e)}{e}N\sum_{k\leq N/e}\frac{d_{\kappa}(k)}{k}\ll_{\kappa} \frac{N}{e}d_{\kappa}(e)(\log N)^{\kappa},
\end{equation*}
by Lemma \ref{lemrankinestimate}, which inserted in \eqref{eq1} gives
\begin{equation*}
\Big| \sum_{n\leq N} f(n) c_d(n)\Big|^2\ll_{\kappa} N^2(\log N)^{2\kappa} d_{\kappa+1}^{2}(d).
\end{equation*}
From this we deduce that the left hand side in \eqref{upperboundfc_d1} is
\begin{align*}
\ll_{\kappa} N^2(\log N)^{2\kappa}\sum_{Q_0<d\leq Q}\frac{d_{\kappa+1}^2(d)}{\phi(d)}\sum_{\substack{q\leq Q\\ d|q}}\frac{1}{q}&\leq N^2(\log N)^{2\kappa}\sum_{Q_0<d\leq Q}\frac{d_{\kappa+1}^{2}(d)}{d\phi(d)}\sum_{\substack{q\leq Q/d}}\frac{1}{q}\\
&\ll N^2(\log N)^{2\kappa+1}\sum_{Q_0<d\leq Q}\frac{d_{\kappa+1}^{2}(d)}{d\phi(d)}\\
&\ll_{\kappa} \frac{N^2(\log N)^{\kappa^2+4\kappa+2}}{Q_0},
\end{align*}
again by Lemma \ref{lemrankinestimate}. 
\end{proof}
From now on, we consider $K$ as a large constant so that the term $(\log K)/K$ in Proposition \ref{Prop5.1} is small enough. Since we are assuming $N^{1/2+\delta}\leq Q\leq N$, with $\delta>0$ small, we let $R:=N^{1/2-\delta/2}.$
\section{Lower bounding the integral over the minor arcs}
It remains to lower bound the integral in \eqref{estimateprop1}. To this aim, following the idea and the notations introduced in \cite{HS}, we apply the Cauchy--Schwarz inequality to get
\begin{equation}
\label{C-S}
\int_{\frak m} |{\mathcal F}(\varphi)|^2d\varphi\geq \bigg(\int_{\frak m} |{\mathcal F}(\varphi)\tilde{\mathcal{F}}(\varphi)|d\varphi\bigg)^2\bigg(\int_{\frak m} |\tilde{\mathcal {F}}(\varphi)|^2d\varphi\bigg)^{-1},
\end{equation}
where $\tilde{\mathcal {F}}(\varphi)=\sum_{n\leq N}\tilde{f}(n)e(n\varphi)$ with 
$$\tilde{f}(n)=\sum_{\substack{r|n\\ r\leq R}}g(r)\Phi\bigg(\frac{n}{N}\bigg).$$
Here $g(n)$ is a suitable arithmetic function and $\Phi(t)$ is a suitable smooth function, compactly supported in $[0, 1]$, with $0\leq \Phi(t)\leq 1$ for 
all $0\leq t\leq 1$. The choice of $g$ is fundamental for succeeding in the proof of our main result. We consider a multiplicative function supported on the squarefree integers and zero on all the prime numbers smaller than $C$, where $C$ is as in condition $(4)$ on $q$. On the prime numbers $C<p\leq N$ we put $g(p)=f(p)-1$, if $\Re(\alpha)\geq 1$, and $g(p)=1-f(p)$, otherwise.

We observe that
\begin{equation}
\label{trivparseval}
\int_{\frak m} |\tilde{\mathcal {F}}(\varphi)|^2d\varphi\leq \sum_{n\leq N}|\tilde{f}(n)|^2\leq\sum_{n\leq N}\bigg|\sum_{\substack{r|n\\ r\leq R}}g(r) \bigg|^2,
\end{equation}
by Parseval's indentity. Regarding the integral of ${\mathcal F}(\varphi)\tilde{\mathcal{F}}(\varphi)$ we first report the full proposition proved in \cite{HS}, that was used to find a manageable lower bound for its integral over minor arcs.
\begin{prop}
\label{newprop1} 
Keep notations as above and assume that $KQ_0 \le R\le \sqrt{N}$ and $|f(n)| \ll_{\epsilon} N^{\epsilon}$ for any $\epsilon > 0$ when $n\leq N$.  Then 
\begin{equation*}
\int_{\frak m} |{\mathcal F}(\varphi) \tilde {\mathcal F}(\varphi) | d\varphi \ge \sum_{KQ_0 < q\le R} \Big| \sum_{\substack{r\le R \\ q|r}} \frac{g(r)}{r} \Big| \Big| \sum_{n\le N} 
f(n) c_q(n)\Phi(n/N) \Big|+ O_{\epsilon,\Phi}(\Delta RN^{\frac 12 +\epsilon}  ), 
\end{equation*}
where $\Delta=\max_{r\leq R}|g(r)|$ and $\Phi(t)$ is a suitable smooth function as in \cite[Proposition 3]{HS}.
\end{prop} 
Next, we show how to remove in our set up the smooth cut-off in the sum above.
\begin{lem}
\label{lemparseval-shiu}
Let $R=N^{1/2-\delta/2}$ and assume that $KQ_0 \le R$. Then there exists a smooth function $\Phi(t)$ satisfying the hypotheses in \cite[Proposition 3]{HS} such that
\begin{align*}
\int_{\frak m} |{\mathcal F}(\varphi) \tilde {\mathcal F}(\varphi) | d\varphi \geq \sum_{KQ_0 < q\le R} \Big| \sum_{\substack{r\le R \\ q|r}} \frac{g(r)}{r} \Big| \Big| \sum_{n\le N} 
f(n) c_q(n) \Big|+O_{\delta,\kappa}(N^{1-\delta/11}),
\end{align*}
if $N$ large enough in terms of $\delta$ and $\kappa$.
\end{lem}
\begin{proof}
Let $\Psi(t):\mathbb{R}\rightarrow [0,1]$ be a smooth function compactly supported on $[-1,1]$ with 
$$\int_{\mathbb{R}}\Psi(t)dt=1.$$
Then consider the following convolution
$$\Phi(t)=T\textbf{1}_{[1/T, 1-1/T]}(t)\ast \Psi(Tt)=T\int_{1/T}^{1-1/T}\Psi(T(s-t))ds,$$
for any real number $T\geq 4$.
A quick analysis of this integral reveals that $\Phi(t)$ is a smooth function such that
\[\Phi(t)= \left\{ \begin{array}{lll}
         1& \mbox{if $2/T\leq t\leq 1-2/T$};\\
         \in [0,1] & \mbox{if $1-2/T\leq t\leq 1$ or $0\leq t\leq 2/T$};\\
         0 & \mbox{if $t\geq 1$ or $t\leq 0$}.\end{array} \right. \]
In particular, $\Phi(t)$ is a smooth function, compactly supported in $[0, 1]$, with $0\leq \Phi(t)\leq 1$ for 
all $0\leq t\leq 1$, and with
$$\int_{0}^{1}\Phi(t)dt\geq 1-\frac{4}{T}.$$
It is easy to see that
$$\Phi^{(k)}(t)\ll T^{k}||\Psi^{(k)}||_{L^1},$$
for every $k\geq 0$. Let 
$$F(\xi)=\int_{\mathbb{R}}\Phi(t)e^{-2\pi i \xi t}dt$$
be the Fourier transform of $\Phi(t)$. Then $F$ is continuous and $F(0)=\int_{\mathbb{R}}\Phi(t)dt<\infty$. Moreover, by using $k$ times integration by parts and the definition of $\Phi(t)$ we immediately deduce that
$$F(\xi)=\frac{1}{(2\pi i\xi)^{k}}\int_{1-2/T}^{1}\Phi^{(k)}(t)e^{-2\pi i \xi t}dt+\frac{1}{(2\pi i\xi)^{k}}\int_{0}^{2/T}\Phi^{(k)}(t)e^{-2\pi i \xi t}dt\ll \frac{T^{k-1}||\Psi^{(k)}||_{L^1}}{(2\pi|\xi|)^k},$$
where the implicit constant is absolute. In particular, we get for all $\xi\in\mathbb{R}$ that $F(\xi)\ll T(1+|\xi|)^{-2}$. This is equivalent to say that $\Phi$ as defined satisfies the conditions of the smooth weight introduced in \cite{HS}. Moreover, since in the proof of \cite[Proposition 3]{HS}, which corresponds to Proposition \ref{newprop1} here, it was only used the bound $F(\xi)\ll (1+|\xi|)^{-2}$, where the implicit constant here is directly proportional to that in the error term of Proposition \ref{newprop1}, we may conclude that this last one is indeed
$$\ll_\epsilon T\Delta RN^{1/2+\epsilon}.$$
For any $4\leq T\leq \sqrt{N}$, we may write  
$$\sum_{n\le N} 
f(n) c_q(n)\Phi\left(\frac{n}{N}\right)=\sum_{n\le N} 
f(n) c_q(n)+O\bigg(\sum_{N(1-2/T)<n\le N} d_{\kappa}(n)(q,n)+\sum_{n\le 2N/T} d_{\kappa}(n)(q,n)\bigg).$$
We can estimate the first sum in the big-Oh term with
$$\ll \sum_{e|q}ed_{\kappa}(e)\sum_{N(1-2/T)/e<l\leq N/e}d_{\kappa}(l)\ll_{\delta} d_{\kappa+1}(q)\frac{N}{T\log N}\exp\bigg(\sum_{p\leq N}\frac{{\kappa}}{p}\bigg)\ll_{\delta,\kappa} d_{\kappa+1}(q)\frac{N(\log N)^{\kappa-1}}{T},$$
for every $4\leq T\leq \sqrt{N}$, using Shiu's theorem \cite[Theorem 1]{S}, Mertens' theorem and considering $\delta$ small enough. On average over $q$ in Lemma \ref{lemparseval-shiu}, by upper bounding $|g(r)|\leq d_{\kappa+1}(r)$, it will contribute
$$\ll_{\delta,\kappa} \frac{N(\log N)^{\kappa-1}}{T}\sum_{q\le R}\frac{d_{\kappa+1}(q)d_{\kappa+1}(q)}{q} \sum_{\substack{r\le R}} \frac{d_{\kappa+1}(r)}{r} \ll_{\kappa} \frac{N(\log N)^{(\kappa+1)^2+2\kappa}}{T}=\frac{N(\log N)^{\kappa^2+4\kappa+1}}{T},$$
say, for any $4\leq T\leq \sqrt{N}$. The second sum in the big-Oh term above can be estimated similarly, but replacing the application of Shiu's theorem with an application of Lemma \ref{lemrankinestimate}, and gives the same contribution.

Finally, observing that $\Delta$ satisfies $\Delta\ll R^{O_\kappa(1/\log\log R)}$ (see e.g. \cite[ch. I, Theorem 5.4]{T} for the case of $d_2$, which can be easily generalized to a general $d_{\kappa}$), it is easy to see that letting $\epsilon:=\delta/4$, say, the error term in Proposition \ref{newprop1} becomes 
$$\ll_\delta TN^{1-\delta/4+O_\kappa(1/\log\log N)}\leq TN^{1-\delta/5}\leq N^{1-\delta/10},$$ if $N$ large enough in terms of $\delta$ and $\kappa$, by letting $T:=N^{\delta/10}$. Putting the above considerations together we can now deduce the lemma from Proposition \ref{newprop1}.
\end{proof}

We now find an upper bound for the sum of $|\tilde{f}(n)|^2$, but before we state the next lemma which will be useful later.
\begin{lem}
\label{lemmaparticularaverage}
Let $g(n)$ be a multiplicative function supported on squarefree integers such that $|g(n)|\leq d_{\kappa+1}(n)$ and
\begin{equation}
\label{applmainstatistic2}
\sum_{p\leq x}|g(p)|^2\log p=\beta x+O\bigg(\frac{x}{(\log x)^{A_2}}\bigg)\ \ \ (2\leq x\leq R),
\end{equation}
with $\kappa, \beta, A_2$ and $R$ as usual. Then we have
\begin{equation}
\label{particularaverage}
\sum_{q\leq R}\frac{|g(q)|^2}{q}\bigg(\sum_{d|q}\frac{d_{\kappa+1}(d)}{d^{3/4}}\bigg)^2\ll (\log N)^{\beta},
\end{equation}
with an implicit constant depending on $\kappa,A_2$ and that in \eqref{applmainstatistic2}.
\end{lem}
\begin{proof}
Expanding the square out and swapping summation we find that the sum in \eqref{particularaverage} is
\begin{align*}
&\sum_{d_1,d_2\leq R\ \textrm{squarefree}}\frac{d_{\kappa+1}(d_1)d_{\kappa+1}(d_2)}{d_1^{3/4}d_2^{3/4}}\sum_{\substack{q\leq R\\ q\equiv 0\pmod{[d_1,d_2]}}}\frac{|g(q)|^2}{q}\\
&\leq \sum_{d_1,d_2\leq R}\frac{d_{\kappa+1}(d_1)d_{\kappa+1}(d_2)|g([d_1,d_2])|^2(d_1,d_2)}{d_1^{7/4}d_2^{7/4}}\sum_{k\leq R}\frac{|g(k)|^2}{k},
\end{align*}
where $[a,b]$ stands for the least common multiple of integers $a$ and $b$. The innermost sum is $\ll (\log N)^{\beta}$, by Lemma \ref{lemrankinestimate} and partial summation from \eqref{applmainstatistic2}, with an implicit constant depending on $\kappa,A_2$ and that of \eqref{applmainstatistic2}. On the other hand, the double sum over $d_1,d_2$ is
$$\leq \sum_{d_1,d_2\leq R}\frac{d_{\kappa+1}(d_1)^3d_{\kappa+1}(d_2)^3(d_1,d_2)}{d_1^{7/4}d_2^{7/4}}\leq \sum_{e\leq R}\frac{d_{\kappa+1}(e)^6}{e^{5/2}}\bigg(\sum_{k}\frac{d_{\kappa+1}(k)^3}{k^{7/4}}\bigg)^2.$$
Since 
$$\sum_{k}\frac{d_{\kappa+1}(k)^3}{k^{7/4}}\ll_{\kappa} 1,$$
by using e.g. $d_{\kappa+1}(k)\ll_{\kappa}k^{3/24}$, we obtain that the final double sum above is 
$$\ll_{\kappa} \sum_{e\leq R}\frac{d_{\kappa+1}(e)^6}{e^{5/2}}\ll_{\kappa} 1.$$
Collecting the above estimate together we get \eqref{particularaverage}.
\end{proof}
\begin{prop}
\label{propinttildeF}
Let $R=N^{1/2-\delta/2}$ as before. Then we have 
\begin{align}
\label{finalestimateintg}
\sum_{n\leq N}|\tilde{f}(n)|^2\ll N(\log N)^{\beta+2|\Re(\alpha)-1|},
\end{align}
where the implicit constant may depend on $\delta,\kappa,A_1,A_2$ and that in \eqref{mainstatistic1}--\eqref{mainstatistic2}.
\end{prop}
\begin{rmk}
It is crucial to have a \emph{sharp} upper bound for the sum in \eqref{finalestimateintg} to guarantee a sharp lower bound for the variance in arithmetic progressions. Indeed, \eqref{finalestimateintg} provides an upper bound for the integral in \eqref{trivparseval} which coincides with the denominator in \eqref{C-S}. Finding a sharp lower bound for the $L^2$--integral in \eqref{C-S} is a key step towards proving Theorem \ref{thmalpha}.
\end{rmk}
\begin{proof}
To begin with, we expand the square in \eqref{trivparseval} out and swap summations to find that the average square of $\tilde{f}$ is
\begin{equation}
\label{expandsquare}
\leq N\sum_{r,r'\leq R}\frac{g(r)\bar{g}(r')}{[r,r']}+O\bigg(\sum_{r\leq R}|g(r)|\bigg)^2.
\end{equation}
Regarding the error term in \eqref{expandsquare}, we notice the sum may be upper bounded by
\begin{equation}
\label{errorg}
R\sum_{r\leq R}\frac{|g(r)|}{r}\ll_\kappa R(\log R)^{\kappa+1}\leq N^{1/2-\delta/2}(\log N)^{\kappa+1},
\end{equation}
by an application of Lemma \ref{lemrankinestimate}. 

Using a manipulation which occurs in the work of Dress, Iwaniec and Tenenbaum (see e.g. \cite[equation 1]{DIT}),  we may rearrange the sum in the main term of \eqref{expandsquare} as
\begin{align}
\label{maintermg}
\sum_{r,r'\leq R}\frac{g(r)\bar{g}(r')}{rr'}\sum_{q|r,q|r'}\phi(q)&=\sum_{q\leq R}\phi(q)\bigg|\sum_{\substack{r\leq R\\ q|r}}\frac{g(r)}{r}\bigg|^2\\
&=\sum_{q\leq R}\frac{\phi(q)|g(q)|^2}{q^2}\bigg|\sum_{\substack{k\leq R/q\\ (q,k)=1}}\frac{g(k)}{k}\bigg|^2.\nonumber
\end{align}
We now need a careful estimate for the innermost sum in the second line of \eqref{maintermg}. We restrict first to the case $\Re(\alpha)\geq 1$. If $\alpha=1$, we define the  auxiliary multiplicative function $\tilde{g}$ such that
$$\tilde{g}(p^{j})=\left\{ \begin{array}{ll}
        g(p^{j}) & \mbox{if $p\nmid q$};\\
        g(p)^j & \mbox{otherwise}.\end{array} \right.$$ 
In this way the innermost sum above may be rewritten as
$$\sum_{\substack{k\leq R/q\\ (q,k)=1}}\frac{\tilde{g}(k)}{k} =\sum_{d|q}\frac{g(d)\mu(d)}{d}\sum_{l\leq R/dq}\frac{\tilde{g}(l)}{l},$$
since $q$ is squarefree and arguing as at the start of the proof of Theorem \ref{GKvariant}. We notice that since $\alpha=1$, we have
$$\sum_{p\leq x} \tilde{g}(p)\log p=\sum_{p\leq x} g(p)\log p\ll x/\log^{A_1}x\ \ \ (2\leq x\leq N).$$
By Theorem \ref{GKthm}, we deduce that
$$\sum_{n\leq x}\tilde{g}(l)\ll x(\log x)^{\kappa-A_1-1}\log\log x\ \ \ (2\leq x\leq N)$$
and by partial summation, remembering $A_1>\kappa+2$ from the hypothesis of Theorem \ref{thmalpha}, that
$$\sum_{l\leq R/dq}\frac{\tilde{g}(l)}{l}\ll 1,\ \textrm{for\ any}\ d|q\ \textrm{and}\ q\leq R,$$
thus concluding that
$$\sum_{q\leq R}\frac{\phi(q)|g(q)|^2}{q^2}\bigg|\sum_{\substack{k\leq R/q\\ (q,k)=1}}\frac{g(k)}{k}\bigg|^2\ll \sum_{q\leq R}\frac{|g(q)|^2}{q}\bigg(\sum_{d|q}\frac{|g(d)|}{d}\bigg)^2\ll (\log N)^{\beta},$$
for any $\beta>0$, with an implicit constant depending on $\kappa,A_1,A_2$ and that in \eqref{mainstatistic1}--\eqref{mainstatistic2}. The last estimate follows from Lemma \ref{lemmaparticularaverage}.

From now on we will work under the hypothesis $\alpha\neq 1$ and $\Re(\alpha)\geq 1$. We first note that \eqref{maintermg} is bounded by
\begin{align}
\label{maintermg2}
&\ll_\kappa\sum_{q\leq R}\frac{|g(q)|^2}{q}+\sum_{q\leq R/4}\frac{|g(q)|^2}{q}\bigg|\sum_{\substack{4\leq k\leq R/q\\ (q,k)=1}}\frac{g(k)}{k}\bigg|^2\\
&=O_\kappa((\log N)^{\beta})+\sum_{q\leq R/4}\frac{|g(q)|^2}{q}\bigg|\sum_{\substack{4\leq k\leq R/q\\ (q,k)=1}}\frac{g(k)}{k}\bigg|^2.\nonumber
\end{align}
By Theorem \ref{GKvariant}, we have
\begin{align}
\label{meanvaluegovercoprimalitycond}
\sum_{\substack{4\leq k\leq x\\(k,q)=1}}g(k)&=x(\log x)^{\alpha-2}\sum_{j=0}^{J}\frac{\lambda_j}{(\log x)^{j}}+O(|\tilde{G}_q^{(2\lfloor A_1\rfloor+2)}(1)|x(\log x)^{\kappa-A_1}(\log\log x)),\\
&+O\bigg(x^{3/4}\sum_{d|q}\frac{d_{\kappa+1}(d)}{d^{3/4}}\bigg),\nonumber
\end{align}
where 
$$\lambda_j=\lambda_j(g,\alpha,q)=\frac{1}{\Gamma(\alpha-1-j)}\sum_{l+h=j}\frac{(H_q^{-1})^{(h)}(1)c_l}{h!}=:\frac{\lambda_j'}{\Gamma(\alpha-1-j)},$$
with
$$H_q(z)=\prod_{p|q}\bigg(1+\frac{g(p)}{p^z}\bigg),\ \ \ c_l=\frac{1}{l!}\frac{d^l}{dz^l}\bigg(\zeta_N(z)^{-(\alpha-1)}G(z)\frac{((z-1)\zeta(z))^{\alpha-1}}{z}\bigg)_{z=1}$$
and 
$$G(z)=\sum_{\substack{n:\\ p|n\Rightarrow p\leq N}}\frac{g(n)}{n^z},\ \ \ \tilde{G}_q(z)=\sum_{d|q}\frac{|g(d)|}{d^z}=\prod_{p|q}\bigg(1+\frac{|g(p)|}{p^z}\bigg)$$
on $\Re(z)\geq 1$. Here each $c_l=c_l(g,\alpha)$ is uniformly bounded on $|\alpha|\leq \kappa$, thanks to an application of Lemma \ref{lemderivativeseulerproduct0} with $f$ replaced by $g$ here and $\alpha$ by $\alpha-1$.

Using partial summation, we get
\begin{equation}
\label{resultofpartialsumm}
\bigg|\sum_{\substack{4\leq k\leq R/q\\ (k,q)=1}}\frac{g(k)}{k}\bigg|\ll \frac{(\log (R/q))^{\Re(\alpha)-1}}{|\Gamma(\alpha)|}\bigg(\sum_{j=0}^{J}|\lambda_j'|+|\tilde{G}_q^{(2\lfloor A_1\rfloor+2)}(1)|+\sum_{d|q}\frac{d_{\kappa+1}(d)}{d^{3/4}}\bigg),
\end{equation}
where the implicit constant depends on $\kappa,A_1$ and the implicit constant in $\eqref{mainstatistic1}.$ Here we used that $\Gamma(\alpha)^{-1}$ is an entire function on the whole complex plane satisfying two main properties:
$$|\Gamma(\alpha)|\leq \Gamma(\Re(\alpha))\ \textrm{and}\ \Gamma(\alpha-l)=\frac{\Gamma(\alpha)}{(\alpha-l)\cdots(\alpha-1)},$$
for any $l\geq 1$ and $\alpha\in\mathbb{C}$ such that $\Re(\alpha)\geq 1$ and $|\alpha|\leq \kappa$. We can pretty easily deduce that $|\lambda_j'|\ll_{\kappa,j} \sum_{h=0}^{j}|(H_q^{-1})^{(h)}(1)|$. Likewise as in the proof of Theorem \ref{GKvariant}, we can write 
$${H}_q^{-1}(z)=\prod_{p|q}\bigg(1-\frac{g(p)}{p^z}\bigg)\prod_{p|q}\bigg(1+\frac{g(p)}{p^z}\bigg)^{-1}\bigg(1-\frac{g(p)}{p^z}\bigg)^{-1}:=\tilde{H}_q(z)\tilde{\tilde{H}}_q(z)$$
and show that we can bound all the derivatives of $\tilde{\tilde{H}}_q(z)$ with a constant independent of $q$. By linearity, all the derivatives of ${H}_q^{-1}$ will be a linear combination with complex coefficients of those of $\tilde{H}_q$, which are given by
$$(\tilde{H}_q)^{(h)}(1)=(-1)^{h}\sum_{d|q}\frac{\mu(d)g(d)}{d}(\log d)^{h}\ll \sum_{d|q}\frac{|g(d)|}{d}(\log d)^{h},$$
for any $0\leq h\leq J$. Hence 
\begin{align*}
\sum_{j=0}^{J}|\lambda_j'|\ll_{\kappa,J} \sum_{j=0}^{J}\sum_{h=0}^{j}\sum_{d|q}\frac{|g(d)|}{d}(\log d)^{h}&\ll_{\kappa,J}\sum_{j=0}^{J}\sum_{d|q}\frac{|g(d)|}{d}((\log d)^{j}+1)\\
&\ll_{\kappa,J}\sum_{d|q}\frac{|g(d)|}{d}((\log d)^{J}+1).
\end{align*}
Thus we deduce \eqref{resultofpartialsumm} is
$$\ll \frac{(\log (R/q))^{\Re(\alpha)-1}}{|\Gamma(\alpha)|}\bigg(\sum_{d|q}\frac{|g(d)|}{d}((\log d)^{2A_1+2}+1)+\sum_{d|q}\frac{d_{\kappa+1}(d)}{d^{3/4}}\bigg),$$
where the implicit constant depends on $\kappa,A_1$ and that one in $\eqref{mainstatistic1}.$
We conclude that \eqref{maintermg2} is
\begin{align*}
&\ll (\log N)^{\beta}+\frac{(\log N)^{2(\Re(\alpha)-1)}}{|\Gamma(\alpha)|^2}\sum_{q\leq R}\frac{|g(q)|^2}{q}\bigg(\sum_{d|q}\frac{|g(d)|}{d}((\log d)^{2A_1+2}+1)+\sum_{d|q}\frac{d_{\kappa+1}(d)}{d^{3/4}}\bigg)^2\\
&\ll \frac{(\log N)^{\beta+2(\Re(\alpha)-1)}}{|\Gamma(\alpha)|^2}\ll_\kappa (\log N)^{\beta+2(\Re(\alpha)-1)},
\end{align*}
by Lemma \ref{lemmaparticularaverage}, with an implicit constant depending on $\delta,\kappa,A_1,A_2$ and that in \eqref{mainstatistic1}--\eqref{mainstatistic2}. This concludes the proof when $\Re(\alpha)\geq 1$, since the error \eqref{expandsquare} will be negligible, thanks to \eqref{errorg}. 

When instead $\Re(\alpha)<1$, by definition of $g$ we now get from Theorem \ref{GKvariant}
\begin{align*}
\sum_{\substack{4\leq k\leq x\\(k,q)=1}}g(k)&=x(\log x)^{-\alpha}\sum_{j=0}^{J}\frac{\lambda_j}{(\log x)^{j}}+O(|\tilde{G}_q^{(2\lfloor A_1\rfloor+2)}(1)|x(\log x)^{\kappa-A_1}(\log\log x)),\\
&+O\bigg(x^{3/4}\sum_{d|q}\frac{d_{\kappa+1}(d)}{d^{3/4}}\bigg),
\end{align*}
where 
$$\lambda_j=\lambda_j(g,\alpha,q)=\frac{1}{\Gamma(1-\alpha-j)}\sum_{l+h=j}\frac{(H_q^{-1})^{(h)}(1)c_l}{h!}=:\frac{\lambda_j'}{\Gamma(1-\alpha-j)},$$
with
$$c_l=\frac{1}{l!}\frac{d^l}{dz^l}\bigg(\zeta_N(z)^{-(1-\alpha)}G(z)\frac{((z-1)\zeta(z))^{1-\alpha}}{z}\bigg)_{z=1}$$
and $G(z),\tilde{G}_q(z)$ and $H_q(z)$ defined as before. Again by partial summation we get
\begin{equation*}
\bigg|\sum_{\substack{4\leq k\leq R/q\\ (k,q)=1}}\frac{g(k)}{k}\bigg|\ll \frac{(\log (R/q))^{1-\Re(\alpha)}}{|\Gamma(2-\alpha)|}\bigg(\sum_{j=0}^{J}|\lambda_j'|+|\tilde{G}_q^{(2\lfloor A_1\rfloor+2)}(1)|+\sum_{d|q}\frac{d_{\kappa+1}(d)}{d^{3/4}}\bigg),
\end{equation*}
from which we can conclude as before, since all the other considerations and computations carry over exactly the same.
\end{proof}
Regarding the summation over $r$ in Lemma \ref{lemparseval-shiu} we are going to prove the following proposition.
\begin{prop}
\label{lemsumr}
Let $q$ be a positive integer with $KQ_0\leq q\leq N^{1/2-3\delta/4}$ satisfying condition $(4)$, i.e. any prime divisor of $q$ is larger than $C$, and $(6)$. Then we have
\begin{equation}
\label{estimatesumr1}
\Big| \sum_{\substack{r\le R \\ q|r}} \frac{g(r)}{r} \Big|\gg \frac{|g(q)|}{q}(\log N)^{|\Re(\alpha)-1|},
\end{equation}
where the implicit constant may depend on $\delta,\kappa,A_1,D$ and the implicit constant in \eqref{mainstatistic1}. Moreover, we are assuming $N$ and $C$ sufficiently large with respect to all of these parameters.
\end{prop}
Before starting with the proof we insert here a lemma which will be useful later.
\begin{lem}
\label{lemderivativeseulerproduct}
Let $g$ be a multiplicative function supported on the squarefree numbers and such that $|g(n)|\leq d_{\kappa+1}(n)$, for a certain real positive constant $\kappa>1$ and any $N$-smooth integer $n$. Assume moreover that $g(p)=0$, for any prime $p\leq C$, and define the following Euler products 
$$H_q(z)=\prod_{p|q}\bigg(1+\frac{g(p)}{p^z}\bigg),\ \tilde{G}_q(z)=\prod_{p|q}\bigg(1+\frac{|g(p)|}{p^z}\bigg)\ \ \ (\Re(z)\geq 1),$$
where $q$ is a squarefree positive integer smaller than $N$ satisfying conditions $(4)$ and $(6)$. Then for every positive integer $h$ we have 
$$\max\{|(H_q^{-1})^{(h)}(1)|,|\tilde{G}_q^{(h)}(1)|\}\ll_{h,\kappa,D}C^{-1/5},$$
if $C=C(\kappa,h)>\kappa+1$ is sufficiently large. Moreover, under our assumptions on $q$ we also have 
$$\max\{|H_q^{-1}(1)|,|\tilde{G}_q(1)|\}\asymp_{\kappa,D} 1.$$
\end{lem}
\begin{proof}
Let us focus on $\tilde{G}_q$, since similar computations also hold for $H_q^{-1}$. For values of $h\geq 1$ we use the Fa\`a di Bruno's formula \cite[p. 807, Theorem 2]{R} to find
$$\tilde{G}_q^{(h)}(1)=\tilde{G}_q(1)h!\sum_{m_1+2m_2+...+hm_h=h}\frac{\prod_{i=1}^{h}(\gamma_q^{(i-1)}(1))^{m_i}}{1!^{m_1}m_1!2!^{m_2}m_2!\cdots h!^{m_h}m_h! },$$
where 
$$\gamma_q(z)=\frac{\tilde{G}_q'}{\tilde{G}_q}(z)=-\sum_{n\geq 1}\frac{\Lambda_{\tilde{g}}(n)}{n^z}=-\sum_{p|q}\sum_{k=1}^{\infty}\frac{\Lambda_{\tilde{g}}(p^{k})}{p^{kz}},$$
if we indicate with $\tilde{g}(n)=|g(n)|\textbf{1}_{n|q}$ and define $\Lambda_{\tilde{g}}(n)$ exactly as the $n$-th coefficient in the Dirichlet series corresponding to minus the logarithmic derivative of $\tilde{G}_q(z)=\sum_{n}\frac{\tilde{g}(n)}{n^z}$.
Analysing the values of the $\Lambda_{\tilde{g}}$ function, we see that it is supported only on prime powers for primes dividing $q$. More precisely, on those powers we have the following relation 
$$\Lambda_{\tilde{g}}(p^k)=(-1)^{k+1}|g(p)|^k\log p,$$
which in turn follows from 
\begin{equation}
\label{fundamentalequationLambda}
\tilde{g}(n)\log(n)=\Lambda_{\tilde{g}}\ast \tilde{g}(n),\ \textrm{for any}\ n.
\end{equation}
The above also shows that $\Lambda_{\tilde{g}}(p^k)=0$ whenever $p\leq C$, by the support of $g$, and choosing $C=C(\kappa)>\kappa+1$ large enough makes the series over $k$ on $\Re(z)\geq 1$ convergent. We clearly obtain
$$\gamma_q^{(i)}(1)=-\sum_{p|q}\sum_{k=1}^{\infty}\frac{(-1)^{k+i+1}|g(p)|^{k}(k\log p)^{i+1}}{kp^{k}}\ll_{i,\kappa} \sum_{\substack{p|q\\ p>C}}\frac{(\log p)^{i+1}}{p}\leq \frac{1}{C^{1/5}}\sum_{p|q}\frac{(\log p)^{i+1}}{p^{4/5}},$$
since $|g(p)|$ is uniformly bounded by $\kappa+1>0$ and supported only on large primes. Remembering that $q\in\mathcal{A}$ by condition $(6)$, we immediately deduce that the last sum is bounded for $i\leq h-1$, implying that $\tilde{G}_q^{(h)}(1)\ll_{h,\kappa,D} \tilde{G}_q(1)/C^{1/5}$, for any $h\geq 1$. However, $\tilde{G}_q(1)$ is itself bounded, because 
$$\tilde{G}_q(1)=\exp\bigg(\sum_{p|q}\frac{|g(p)|}{p}+O(1)\bigg)= \exp(O_{\kappa,D}(1))\asymp_{\kappa,D} 1,$$ 
by condition $(6)$ and $C>\kappa+1$. Similarly, we can show the same for all the derivatives of $H_q^{-1}$, by first showing that the bound for those of its logarithmic derivative coincides with the bound for the derivative of $\gamma_q$. Indeed, since we have
$$\frac{d}{dz}\log(H_q^{-1}(z))=\sum_{p|q}\frac{g(p)\log p}{p^{z}+g(p)}=\sum_{p|q}g(p)\log p\sum_{k=0}^{\infty}\frac{(-g(p))^k}{p^{(k+1)z}},$$
where the series converges since $\Re(z)\geq 1$ and $p>C>\kappa+1$, its corresponding $j$-th derivative is
$$\sum_{p|q}g(p)(\log p)^{j+1}\sum_{k=0}^{\infty}\frac{(-g(p))^k(-k-1)^j}{p^{(k+1)z}},$$
from which by taking the absolute value we recover the analogous bound for $\gamma_q^{(j)}$. Finally, note that since $g$ vanishes on the primes smaller than a large constant $C$ the product $H_q(1)$ is not zero. Moreover, we note that
$$H_q^{-1}(1)=\exp\bigg(-\sum_{p|q}\frac{g(p)}{p}+O(1)\bigg)=\exp(O_{\kappa,D}(1))\asymp_{\kappa,D} 1,$$
again by condition $(6)$ on $q$ and $C>\kappa+1$.
\end{proof}
\begin{proof}[Proof of Proposition \ref{lemsumr}]
First of all, note that
\begin{equation}
\label{gcoprimalitycond}
\sum_{\substack{r\le R \\ q|r}} \frac{g(r)}{r}=\frac{g(q)}{q}\sum_{\substack{1\leq r\le R/q \\ (q,r)=1}} \frac{g(r)}{r}=\frac{g(q)}{q}\bigg(1+\sum_{\substack{C\leq r\le R/q \\ (q,r)=1}} \frac{g(r)}{r}\bigg),
\end{equation}
since $g$ is supported on squarefree numbers larger than $C$, and on $1$ where $g(1)=1$.

In order to evaluate the last sum on the right hand side of \eqref{gcoprimalitycond} we apply again Theorem \ref{GKvariant}, as was done in Proposition \ref{propinttildeF}, and conclude with a partial summation argument. In this case our task is facilitated by restricting $q$ to lie in the subset $\mathcal{A}\subset [KQ_0, RN^{-\delta/4}]$, as in condition $(6)$. In particular, since $q\leq RN^{-\delta/4}$ we notice that $\log(R/q)\asymp_{\delta} \log N$ and the condition $q\in\mathcal{A}$ allows us to simplify the asymptotic expansion of the average of $g(n)$. However, since here we are looking for a lower bound, some difficulties arise when $\Re(\alpha)$ is near $1$, for which we will need to invoke condition $(4)$ on $q$ and divide the argument into two different cases according to the size of $|\Re(\alpha)-1|$. 

We first restrict our attention to the case $\Re(\alpha)\geq 1$ in which case we can compute the average of $g$ over the coprimality condition using Theorem \ref{GKvariant}. Assuming $C>4$ sufficiently large, we obtain
$$\sum_{\substack{C\leq k\leq x\\(k,q)=1}}g(k)= \sum_{j=0}^{J}\frac{\lambda_j'}{\Gamma(\alpha-1-j)}(\log x)^{\alpha-2-j}+O\bigg(x(\log x)^{\kappa-A_1}(\log\log x)\bigg),$$
where the big-Oh term depends on $\kappa,A_1,D$ and the implicit constant in \eqref{mainstatistic1} and the $\lambda_j'$ are as in \eqref{meanvaluegovercoprimalitycond}. Here we simplified the expression in the error term by using Lemma \ref{lemderivativeseulerproduct} and noticing that
$$\sum_{d|q}\frac{d_{\kappa}(d)}{d^{3/4}}=\prod_{p|q}\bigg(1+\frac{\kappa}{p^{3/4}}\bigg)=\exp\bigg(\sum_{p|q}\frac{\kappa}{p^{3/4}}+O_{\kappa}(1)\bigg)\ll_{\kappa,D}1,$$
by the conditions $(4)$ and $(6)$ on $q$, if $C=C(\kappa,A_1)>\kappa^{4/3}.$ Moreover, note that
\begin{equation}
\label{lowerboundlambda0}
\lambda_0'=\prod_{p|q}\bigg(1+\frac{g(p)}{p}\bigg)^{-1}\prod_{p\leq N}\bigg(1+\frac{g(p)}{p}\bigg)\bigg(1-\frac{1}{p}\bigg)^{\alpha-1}\asymp 1,
\end{equation}
with an implicit constant depending on $\kappa,A_1,C,D$ and that in \eqref{mainstatistic1}, since $g(p)=0$ when $p\leq C$ with $C>\kappa+1$, and using Lemma \ref{lemderivativeseulerproduct} and partial summation from \eqref{mainstatistic1}, as at the start of the proof of Lemma \ref{lemderivativeseulerproduct0}.
By partial summation and similar considerations to those employed in the proof of Proposition \ref{propinttildeF}, remembering the hypothesis $\eqref{relationbetaA_1}$ on $A_1$, we deduce that
\begin{equation}
\label{lastapprox}
\bigg|\sum_{\substack{C\leq r\leq R/q\\ (q,r)=1}}\frac{g(r)}{r}\bigg|\geq E\bigg|\frac{\lambda_0'}{\Gamma(\alpha)}\bigg|(\log(R/q))^{\Re(\alpha)-1},
\end{equation}
where $E$ depends on $\delta,\kappa,A_1,D$ and the implicit constant in \eqref{mainstatistic1}, if we think of $C$ as large enough in terms of $\delta,\kappa,A_1,D$ and take $N$ sufficiently large to all of these parameters. 

Now, write $\alpha=1+L/\log\log N+i\tau$, with $L\geq 0$, and suppose that $L>L_0$, where
$$L_0:=\min\bigg\{l\in\mathbb{R}^{+}: e^{l-1}\geq \frac{2}{E}\bigg|\frac{\Gamma(\kappa)}{\lambda_0'}\bigg|\bigg\}$$
depends on $\delta,\kappa,A_1,D$ and the implicit constant in \eqref{mainstatistic1}.
Then \eqref{lastapprox} is clearly
$$= E \bigg|\frac{\lambda_0'}{\Gamma(\alpha)}\bigg|e^{L+O(|\log \delta|/\log\log N)}\geq 2,$$
if we take $N$ large enough. This, together with \eqref{lowerboundlambda0}, concludes the proof in this subcase.

Suppose now $0\leq L\leq L_0$. We remark that when $\tau$ is either $0$ or a possibly small function of $N$ and $\Re(\alpha)$ is suitably close to $1$, the above partial summation argument could lose its efficiency. For this reason, a direct argument is needed, one in which only the value of the $\Re(\alpha)$ counts. Hence, we start again from \eqref{gcoprimalitycond} and note that
\begin{equation}
\label{casealpha1}
\sum_{\substack{1\leq r\le R/q \\ (q,r)=1}} \frac{g(r)}{r}=\sum_{\substack{(q,r)=1}} \frac{g(r)}{r}-\sum_{\substack{r>R/q\\ (q,r)=1}} \frac{g(r)}{r}
\end{equation}
where the complete series in \eqref{casealpha1} converges, since it is equal to
$$\prod_{\substack{p\leq N\\p\nmid q}}\bigg(1+\frac{g(p)}{p}\bigg)=H_q^{-1}(1)\prod_{p\leq N}\bigg(1+\frac{g(p)}{p}\bigg)\asymp 1,$$
with an implicit constant depending on $\kappa,A_1,D$ and that in \eqref{mainstatistic1}. Indeed, Lemma \ref{lemderivativeseulerproduct} gives 
$$H_q^{-1}(1)\asymp_{\kappa,D}1$$
and we have
\begin{align*}
\bigg|\prod_{p\leq N}\bigg(1+\frac{g(p)}{p}\bigg)\bigg|=\exp\bigg(\Re\bigg(\sum_{C<p\leq N}\bigg(\frac{g(p)}{p}+O_\kappa\bigg(\frac{1}{p^2}\bigg)\bigg)\bigg)\bigg)\asymp_\kappa \exp\bigg(\sum_{C<p\leq N}\frac{\Re(g(p))}{p}\bigg)\asymp 1,
\end{align*}
with an implicit constant depending on $\delta,\kappa,A_1,D$ and that in \eqref{mainstatistic1}, if $C,N$ sufficiently large in terms of those parameters. The last estimate follows through partial summation from \eqref{mainstatistic1}.

Finally, since $R/q\geq N^{\delta/4}$, the tail of the series in \eqref{casealpha1} can be made arbitrary small if we choose $N$ large enough. Therefore, we simply have  
$$\sum_{\substack{1\leq r\leq R\\ q|r}} \frac{g(r)}{r}\gg \frac{|g(q)|}{q},$$
with an implicit constant depending on $\delta,\kappa,A_1,D$ and that in \eqref{mainstatistic1}, if $C,N$ large enough in terms of those parameters, which matches the expression in \eqref{estimatesumr1}, since 
$$|\Gamma(1+L/\log\log N+i\tau)|\leq \Gamma(1+L/\log\log N)\ll 1,$$
choosing $N$ sufficiently large, by the continuity of $\Gamma(\alpha)$. 

This concludes the proof in the case $\Re(\alpha)\geq 1$.

In the complementary case, i.e. $\Re(\alpha)<1$, we just note that $g$ has average $1-\alpha\neq 0$ over the primes. All the above computations then carry over, with the opportune modifications already explained at the end of the proof of Proposition \ref{propinttildeF}, and the overall  result may be written as in \eqref{estimatesumr1}. 
\end{proof}

\section{Twisting with Ramanujan's sums}
By inserting the conclusion of Proposition \ref{propupperboundfc_d} and estimate \eqref{averagefsquare} in Proposition \ref{Prop5.1}, so far we have found 
$$V(Q,f) \ge Q\Big(1+ O\Big(\frac{\log K}{K}\Big)\Big) \int_{\frak m} |{\mathcal F}(\varphi)|^2d\varphi+O_{\kappa,K}\bigg(\frac{N^2(\log N)^{\kappa^2+4\kappa+2}}{Q_0}+\frac{N^2(\log N)^{\beta+2\Re(\alpha)-2}}{Q_0}\bigg),$$
where $\int_{\frak m} |{\mathcal F}(\varphi)|^2d\varphi$ may be lower bounded using the results of Lemma \ref{lemparseval-shiu}, Proposition \ref{propinttildeF} and Proposition \ref{lemsumr}, with $K$ a large constant. Hence, we have proved that 
\begin{align}
\label{startingpoint1}
V(Q,f)&\gg \frac{Q}{N(\log N)^{\beta}}\bigg(\sum_{KQ_0\leq q\leq RN^{-\delta/4}}^{'}\frac{|g(q)|}{q}\Big| \sum_{n\le N} 
f(n) c_q(n)\Big|+O_{\delta,\kappa}(N^{1-\delta/11})\bigg)^{2}\\
&+O_\kappa\bigg(\frac{N^2(\log N)^{\kappa^2+4\kappa+2}}{Q_0}+\frac{N^2(\log N)^{\beta+2\Re(\alpha)-2}}{Q_0}\bigg),\nonumber
\end{align}
where $\Sigma^{'}$ indicates a sum over all the squarefree $KQ_0\leq q\leq RN^{-\delta/4}$ under the restrictions $(4)$ and $(6)$ and the $\gg$ constant may depend on $\delta,\kappa,A_1,A_2,D$ and the implicit constants in \eqref{mainstatistic1}--\eqref{mainstatistic2}. Moreover, we are assuming $N$ sufficiently large with respect to all of them as well as to $C$. 

In the rest of this section we explain how to deal with the average of $f(n)$ twisted with Ramanujan's sums, which is indeed the heart of the proof. We begin with the following observation
\begin{equation}
\label{firstmanipulation}
\sum_{n\leq N}f(n)c_q(n)=\sum_{\substack{b\leq N\\ p|b\Rightarrow p|q}}f(b)c_q(b)\sum_{\substack{a\leq N/b\\ (a,q)=1}}f(a),
\end{equation}
using the substitution $n=ab$, with $(a,q)=1$ and $b=n/a$, which is unique, and noticing that
$$c_q(n)=\frac{\mu(q/(n,q))\phi(q)}{\phi(q/(n,q))}=\frac{\mu(q/(b,q))\phi(q)}{\phi(q/(b,q))}=c_q(b),$$
which can be deduced from \eqref{mainpropc_q}. For any $b\leq N/4$ we can apply Theorem \ref{GKvariant} to find
\begin{equation}
\label{SDasymptotic}
\sum_{\substack{a\leq N/b\\ (a,q)=1}}f(a)=\frac{N}{b}(\log(N/b))^{\alpha-1}\sum_{j=0}^{J}\frac{\lambda_j}{(\log(N/b))^{j}}+O\left(\frac{N}{b}(\log(N/b))^{\kappa-1-A_1}(\log\log N)\right),
\end{equation}
where we simplified the expression in the error term by using Lemma \ref{lemderivativeseulerproduct} and noticing that
$$\sum_{d|q}\frac{d_{\kappa}(d)}{d^{3/4}}=\prod_{p|q}\bigg(1+\frac{\kappa}{p^{3/4}}\bigg)=\exp\bigg(\sum_{p|q}\frac{\kappa}{p^{3/4}}+O_{\kappa}(1)\bigg)\ll_{\kappa,D}1,$$
by the conditions $(4)$ and $(6)$ on $q$, if $C=C(\kappa,A_1)>\kappa^{4/3}.$

Here $J$ is the largest integer $<A_1$ and $\lambda_j$ as in the statement of Theorem \ref{GKvariant}.
Since the asymptotic holds only when $N/b\geq 4$ we need to estimate
\begin{equation}
\label{eq11}
\sum_{\substack{N/4<b\leq N\\ p|b\Rightarrow p|q}}d_\kappa(b)|c_q(b)|\leq N^{3/4}\sum_{\substack{p|b\Rightarrow p|q}}\frac{d_\kappa(b)(q,b)}{b^{3/4}}.
\end{equation}
In future we will make use several times of the following lemma.
\begin{lem}
\label{lemmatail}
For all positive integers $q\in\mathcal{A}$, as in condition $(6)$, we have
$$\sum_{\substack{p|b\Rightarrow p|q}}\frac{d_\kappa(b)(q,b)}{b^{3/4}}\ll_{\kappa,D} q^{1/4}d_{\kappa+1}(q),\ \ \sum_{\substack{p|b\Rightarrow p|q}}\frac{d_\kappa(b)(q,b)}{b}\ll_{\kappa,D} d_{\kappa+1}(q).$$
\end{lem}
\begin{proof}
The first sum in the statement is upper bounded by
\begin{equation}
\label{lemmatail1}
\sum_{e|q}e\sum_{\substack{e|b\\ p|b\Rightarrow p|q}}\frac{d_{\kappa}(b)}{b^{3/4}}.
\end{equation}
Since $\kappa>1$, it is
$$\leq \sum_{e|q}e^{1/4}d_{\kappa}(e)\sum_{\substack{f\\ p|f\Rightarrow p|q}}\frac{d_{\kappa}(f)}{f^{3/4}}\leq q^{1/4}d_{\kappa+1}(q)\prod_{p|q}\bigg(1-\frac{1}{p^{3/4}}\bigg)^{-\kappa}\ll_{\kappa,D} q^{1/4}d_{\kappa+1}(q).$$
Similarly the second sum in the statement is bounded by
\begin{equation}
\label{lemmatail2}
\sum_{e|q}e\sum_{\substack{e|b\\ p|b\Rightarrow p|q}}\frac{d_{\kappa}(b)}{b},
\end{equation}
which is 
\begin{equation*}
\leq \sum_{e|q}d_{\kappa}(e)\sum_{\substack{f\\ p|f\Rightarrow p|q}}\frac{d_{\kappa}(f)}{f}\leq d_{\kappa+1}(q)\prod_{p|q}\bigg(1-\frac{1}{p}\bigg)^{-\kappa}\ll_{\kappa,D} d_{\kappa+1}(q).\qedhere
\end{equation*}
\end{proof}
Therefore, we conclude that \eqref{eq11} is
\begin{equation*}
\ll N^{3/4}q^{1/4}d_{\kappa+1}(q)\ll_{\kappa,D} N^{7/8-3\delta/16} d_{\kappa+1}(q).
\end{equation*}
Plugging \eqref{SDasymptotic} into \eqref{firstmanipulation} we get
\begin{align}
\label{eq13}
\sum_{n\leq N}f(n)c_q(n)&=N\sum_{\substack{b\leq N/4\\ p|b\Rightarrow p|q}}\frac{f(b)c_q(b)}{b}(\log(N/b))^{\alpha-1}\bigg(\sum_{j=0}^{J}\frac{\lambda_j}{(\log(N/b))^{j}}\bigg)\\
&+O\bigg(N\sum_{\substack{b\leq N/4\\ p|b\Rightarrow p|q}}\frac{d_{\kappa}(b)|c_q(b)|}{b}(\log (N/b))^{\kappa-1-A_1}(\log\log N)\bigg)\nonumber\\
&+O(N^{7/8-3\delta/16} d_{\kappa+1}(q)).\nonumber
\end{align}
To estimate the sum in the error term here we use same considerations employed in the case of \eqref{eq11}. Since in the hypothesis of Theorem \ref{thmalpha} we assumed $A_1>\kappa+2$, the function $(\log(N/b))^{\kappa-1-A_1}$ is an increasing function of $b$. Therefore, by an application of Lemma \ref{lemmatail} we immediately deduce that the sum in the error term corresponding to values of $b\leq \sqrt{N}$ is
$$\ll (\log N)^{\kappa-1-A_1}(\log\log N)\sum_{\substack{p|b\Rightarrow p|q}}\frac{d_{\kappa}(b)(q,b)}{b}\ll_{\kappa,D} (\log N)^{\kappa-1-A_1}(\log\log N) d_{\kappa+1}(q).$$
On the other hand, the one corresponding to $b>\sqrt{N}$ is simply
$$\ll N^{-1/8}(\log\log N)\sum_{\substack{p|b\Rightarrow p|q}}\frac{d_{\kappa}(b)(q,b)}{b^{3/4}}\ll_{\kappa,D} N^{-3\delta/16}(\log\log N)d_{\kappa+1}(q),$$
again by an application of Lemma \ref{lemmatail}. We conclude that
\begin{align}
\label{firtsestimatef(n)c_q(n)}
\sum_{n\leq N}f(n)c_q(n)&= N\sum_{\substack{b\leq N/4\\ p|b\Rightarrow p|q}}\frac{f(b)c_q(b)}{b}(\log(N/b))^{\alpha-1}\bigg(\sum_{j=0}^{J}\frac{\lambda_j}{(\log(N/b))^{j}}\bigg)\\
&+O\bigg(N(\log N)^{\kappa-1-A_1}(\log\log N)d_{\kappa+1}(q)\bigg),\nonumber
\end{align}
where the constant in the big-Oh term may depend on $\delta,\kappa,A_1, D$ and the implicit constant in \eqref{mainstatistic1} and we take $N$ large enough with respect to these parameters.

The principal aim from now on is to evaluate the following family of sums
\begin{equation}
\label{mainsum}
\sum_{\substack{b\leq N/4\\ p|b\Rightarrow p|q}}\frac{f(b)c_q(b)}{b}\log^{\tilde{\alpha}}(N/b),
\end{equation}
with $\tilde{\alpha}\in\{\alpha-1,\alpha-2,\dots,\alpha-J-1\}.$
In order to do that, we employ condition $(3.a)$ and write $q=rs$, with $p|r\Rightarrow p>(\log N)^{B}$ and $p|s\Rightarrow p\leq (\log N)^{B}$, for a large constant $B>0$ to be chosen later. In view of this factorization we have the following identity:
\begin{align}
\label{multofc_q}
\sum_{\substack{b\leq N/4\\ p|b\Rightarrow p|q}}\frac{f(b)c_q(b)}{b}\log^{\tilde{\alpha}}(N/b)&=\sum_{\substack{b_1\leq \sqrt{N}\\ p|b_1\Rightarrow p|r}}\frac{f(b_1)c_r(b_1)}{b_1}\sum_{\substack{b_2\leq N/4b_1\\ p|b_2\Rightarrow p|s}}\frac{f(b_2)c_s(b_2)}{b_2}\log^{\tilde{\alpha}}(N/b_1 b_2)\\
&+\sum_{\substack{b_2\leq \sqrt{N}\\ p|b_2\Rightarrow p|s}}\frac{f(b_2)c_s(b_2)}{b_2}\sum_{\substack{N^{1/2}<b_1\leq N/4b_2\\ p|b_1\Rightarrow p|r}}\frac{f(b_1)c_r(b_1)}{b_1}\log^{\tilde{\alpha}}(N/b_1 b_2),\nonumber
\end{align}
since by multiplicativity of $c_q(n)$ as function of $q$ and definition of $r,s$ we have 
$$c_q(b)=c_r(b)c_s(b)=c_r(b_1)c_s(b_2).$$ 
We inserted the above structural information on $q$ to reduce the estimate of \eqref{mainsum} to that of a sum over smooth integers, which is easier to handle, and one over integers divisible only by large primes which will turn out to be basically over squarefree integers, notably simplifying its computation. Let us focus our attention on the second double sum on the right hand side of \eqref{multofc_q}. By Lemma \ref{lemmatail}, the innermost sum there is
\begin{align}
\ll \frac{(\log N)^{\max\{\Re(\alpha)-1,0\}}}{N^{1/8}}\sum_{\substack{p|b_1\Rightarrow p|r}}\frac{d_{\kappa}(b_1)|c_r(b_1)|}{b_1^{3/4}}&\ll_{\kappa,D} \frac{(\log N)^{\max\{\Re(\alpha)-1,0\}}}{N^{1/8}}r^{1/4}d_{\kappa+1}(r)\\
&\ll \frac{(\log N)^{\max\{\Re(\alpha)-1,0\}}}{N^{3\delta/16}}d_{\kappa+1}(r).\nonumber
\end{align}
Since this bound is independent of $b_2$ we only need to consider
$$\sum_{\substack{b_2\leq \sqrt{N}\\ p|b_2\Rightarrow p|s}}\frac{|f(b_2)c_s(b_2)|}{b_2}\leq \sum_{\substack{p|b_2\Rightarrow p|s}}\frac{d_{\kappa}(b_2)|c_s(b_2)|}{b_2}\ll_{\kappa,D} d_{\kappa+1}(s),$$
again by Lemma \ref{lemmatail}. In conclusion, the contribution from the second double sum in \eqref{multofc_q} is
\begin{equation}
\label{eq14}
\ll_{\kappa,D}\frac{d_{\kappa+1}(q)(\log N)^{\max\{\Re(\alpha)-1,0\}}}{N^{3\delta/16}}.
\end{equation}
\section{The contribution from small prime factors}
We are left with the estimate of the first double sum in \eqref{multofc_q}. For brevity, let us write $M=N/b_1$. We need to consider first 
\begin{align*}
\sum_{\substack{b_2\leq M/4\\ p|b_2\Rightarrow p|s}}\frac{f(b_2)c_s(b_2)}{b_2}\log^{\tilde{\alpha}}(M/b_2)&=\sum_{k=0}^{K}\binom{-\tilde{\alpha}+k-1}{k}(\log M)^{\tilde{\alpha}-k}\sum_{\substack{b_2\leq M/4\\ p|b_2\Rightarrow p|s}}\frac{f(b_2)c_s(b_2)}{b_2}\log^{k} b_2\\&+O_{\kappa,K}\bigg((\log M)^{\Re(\tilde{\alpha})-K-1}\sum_{\substack{b_2\leq M/4\\ p|b_2\Rightarrow p|s}}\frac{|f(b_2)c_s(b_2)|}{b_2}(\log b_2)^{K+1}\bigg),
\end{align*}
for a constant $K$ that will be chosen in terms of $A_1$ later on. Let us move on now to estimate the sums 
$$\sum_{b\leq M/4,\ p|b\Rightarrow p|s}\frac{f(b)c_s(b)}{b}\log^{k} b,\ \ \ (\forall\ 0\leq k\leq K).$$
First, we observe that we can remove the condition $b\leq M/4$, because using the monotonicity of $b^{1/4}/(\log b)^{k}$, for fixed $k$ and $b$ large, and applying Lemma \ref{lemmatail}, we may deduce that
$$\sum_{\substack{b> M/4\\ p|b\Rightarrow p|s}}\frac{f(b)c_s(b)}{b}\log^{k} b\ll_k \frac{\log^{k}M}{M^{1/4}}\sum_{\substack{p|b\Rightarrow p|s}}\frac{d_{\kappa}(b)|c_s(b)|}{b^{3/4}}\ll_{k,\kappa,D} \frac{\log^{k}M}{M^{1/4}}s^{1/4}d_{\kappa+1}(s).$$
Therefore, the error in replacing the finite sums above with the complete series is
\begin{equation*}
\ll_{k,\kappa,D} \frac{\log^{\Re(\tilde{\alpha})}M}{M^{1/4}}s^{1/4}d_{\kappa+1}(s)\ll_{\delta,\kappa,A_1} \frac{\log^{\Re(\tilde{\alpha})}N}{N^{3\delta/16}}d_{\kappa+1}(s),
\end{equation*}
for any $\tilde{\alpha}$, using that $M\geq N^{1/2}$ and $s\leq q\leq N^{1/2-3\delta/4}$. We have obtained so far
\begin{align}
\label{asymptoticsumb_2}
\sum_{\substack{b_2\leq M/4\\ p|b_2\Rightarrow p|s}}\frac{f(b_2)c_s(b_2)}{b_2}\log^{\tilde{\alpha}}(M/b_2)&=\sum_{k=0}^{K}\binom{-\tilde{\alpha}+k-1}{k}(\log M)^{\tilde{\alpha}-k}\sum_{\substack{p|b_2\Rightarrow p|s}}\frac{f(b_2)c_s(b_2)}{b_2}\log^{k} b_2\\
&+O_{\kappa,K}\bigg((\log M)^{\Re(\tilde{\alpha})-K-1}\sum_{\substack{p|b_2\Rightarrow p|s}}\frac{d_{\kappa}(b_2)(b_2,s)}{b_2}(\log b_2)^{K+1}\bigg)\nonumber\\
&+O_{\delta,\kappa,A_1,D,K}\bigg(\frac{\log^{\Re(\tilde{\alpha})}N}{N^{3\delta/16}}d_{\kappa+1}(s)\bigg).\nonumber
\end{align}
Let us define the following Dirichlet series 
\begin{align*}
\Theta(\sigma):=\sum_{\substack{p|b\Rightarrow p|s}}\frac{f(b)c_s(b)}{b^{\sigma}},\ \ \tilde{\Theta}(\sigma):=\sum_{\substack{p|b\Rightarrow p|s}}\frac{d_{\kappa}(b)(b,s)}{b^{\sigma}},\ \ \ (\sigma\geq 1).
\end{align*} 
In order to find a better and manageable form for them, we will prove the following lemma.
\begin{lem}
\label{lemdirichletseries}
For squarefree values of $s$, we have
\begin{align*}
\Theta(\sigma)&=\prod_{p|s}\bigg(-p+(p-1)\sum_{\nu\geq 0}\frac{f(p^{\nu})}{p^{\nu \sigma}}\bigg),\\
\tilde{\Theta}(\sigma)&=\prod_{p|s}\left(1-p+p(1-1/p^{\sigma})^{-\kappa}\right).
\end{align*}
\end{lem}
\begin{proof}
For a general multiplicative function $f(n)$ we have
\begin{equation*}
\sum_{n}\frac{f(n)c_s(n)}{n^{\sigma}}=\sum_{n}\frac{f(n)}{n^{\sigma}}\sum_{d|n,d|s}\mu(s/d)d=\sum_{d|s}\mu(s/d)d^{1-\sigma}\sum_{k}\frac{f(dk)}{k^{\sigma}},
\end{equation*}
by \eqref{mainpropc_q}. Let $F(\sigma)$ indicate the Dirichlet series of $f$. We denote with $v_p(n)$ the $p$-adic valuation of $n$. Then we get
\begin{align*}
\sum_{k}\frac{f(dk)}{k^{\sigma}}&=\prod_p \sum_{\nu\geq 0}\frac{f(p^{\nu+v_p(d)})}{p^{\nu \sigma}}\\
&=\prod_{p\nmid d}\left(1+\frac{f(p)}{p^{\sigma}}+\frac{f(p^2)}{p^{2\sigma}}+\cdots\right)\prod_{\substack{p^{a}||d\\a\geq 1}}\sum_{\nu\geq 0}\frac{f(p^{\nu+a})}{p^{\nu \sigma}}.
\end{align*}
Therefore, we can write 
$$\sum_{n}\frac{f(n)c_s(n)}{n^{\sigma}}=F(\sigma)\sum_{d|s}\mu(s/d)d^{1-\sigma}\ell(d),$$
where $\ell(d)$ is the multiplicative function given by 
$$\ell(d)=\prod_{\substack{p^{a}||d\\a\geq 1}}\left(1+\frac{f(p)}{p^{\sigma}}+\frac{f(p^2)}{p^{2\sigma}}+\cdots\right)^{-1}\sum_{\nu\geq 0}\frac{f(p^{\nu+a})}{p^{\nu \sigma}}.$$
From this, we immediately find that
\begin{equation*}
\sum_{n}\frac{f(n)c_s(n)}{n^{\sigma}}=F(\sigma)F_s(\sigma),
\end{equation*}
with $F_s(\sigma)$ equal to
$$\prod_{p|s}\left(1+\frac{f(p)}{p^{\sigma}}+\frac{f(p^2)}{p^{2\sigma}}+\cdots\right)^{-1}\bigg(-1+(p-1)\sum_{\nu\geq 1}\frac{f(p^{\nu})}{p^{\nu \sigma}}\bigg),$$
since $s$ is square-free. Therefore, it follows that
$$\Theta(\sigma)=\sum_{\substack{p|b\Rightarrow p|s}}\frac{f(b)c_s(b)}{b^{\sigma}}=\prod_{p|s}\bigg(-p+(p-1)\sum_{\nu\geq 0}\frac{f(p^{\nu})}{p^{\nu \sigma}}\bigg).$$
This concludes the search for the Euler product form of $\Theta(\sigma)$. Regarding $\tilde{\Theta}(\sigma)$ instead, if we indicate with $G(\sigma)$ the Dirichlet series of $d_\kappa(n)\textbf{1}_{p|n\Rightarrow p|s}$, by using the identity $(n,s)=\sum_{d|n,d|s}\phi(d)$ we get
\begin{align*}
\sum_{\substack{n:\\ p|n\Rightarrow p|s}}\frac{d_\kappa(n)(n,s)}{n^{\sigma}}&=G(\sigma)\prod_{p|s}\bigg(1+\frac{(p-1)}{p^{\sigma}}\left(1+\frac{d_\kappa(p)}{p^{\sigma}}+\frac{d_\kappa(p^2)}{p^{2\sigma}}+\cdots\right)^{-1}\sum_{\nu\geq 0}\frac{d_\kappa(p^{\nu+1})}{p^{\nu \sigma}}\bigg)\\
&=G(\sigma)\prod_{p|s}\bigg(p-(p-1)\bigg(1-\frac{1}{p^{\sigma}}\bigg)^{\kappa}\bigg)\\
&=\prod_{p|s}\bigg(1-p+p\bigg(1-\frac{1}{p^{\sigma}}\bigg)^{-\kappa}\bigg),
\end{align*}
since $s$ is squarefree. The proof of the lemma is completed.
\end{proof}
We now show that each term in the sum on the right hand side of \eqref{asymptoticsumb_2} corresponding to a $k\geq 1$ gives a smaller contribution compared to the $k=0$ term. Let us start by noticing that 
$$\Theta^{(k)}(\sigma)=(-1)^{k}\sum_{\substack{p|b\Rightarrow p|s}}\frac{f(b)c_s(b)}{b^{\sigma}}\log^{k}b$$
and defining
\begin{equation}
\label{defsmalltheta}
\theta(\sigma):=\frac{\Theta'}{\Theta}(\sigma)=\sum_{p|s}\frac{\gamma_{p}'(\sigma)}{\gamma_{p}(\sigma)},
\end{equation}
where 
\begin{equation}
\label{eq16}
\gamma_{p}(\sigma):=-p+(p-1)\sum_{\nu\geq 0}\frac{f(p^{\nu})}{p^{\nu \sigma}}.
\end{equation}
Using the Fa\`a di Bruno's formula \cite[p. 807, Theorem 2]{R} we see that
\begin{equation}
\label{FaadeBruno}
\Theta^{(k)}(1)=(e^{\log\Theta(\sigma)})^{(k)}\big|_{\sigma=1}=\Theta(1)k!\sum_{m_1+2m_2+...+km_k=k}\frac{\prod_{j=1}^{k}(\theta^{(j-1)}(1))^{m_j}}{1!^{m_1}m_1!\cdots k!^{m_k}m_k!}.
\end{equation}
Consequently, we need an estimate for the logarithmic derivative of $\Theta$ and its derivatives. To this aim we first note that $(\gamma_{p}'/\gamma_{p})^{(h)}(1)=(\log \gamma_{p}(\sigma))^{(h+1)}|_{\sigma=1}$, which again by the Fa\`a di Bruno's formula is
\begin{align}
\label{FaadeBruno2}
=(h+1)!\sum_{m_1+2m_2+...+(h+1)m_{h+1}=h+1}\frac{(-1+m_1+m_2+...+m_{h+1})!}{1!^{m_1}m_1!\cdots (h+1)!^{m_{h+1}}m_{h+1}!}\prod_{j=1}^{h+1}\bigg(\frac{-\gamma_{p}^{(j)}(1)}{\gamma_{p}(1)}\bigg)^{m_j}.
\end{align}
We observe that
$$\gamma_{p}(1)=-1+f(p)-\frac{f(p)}{p}+\frac{f(p^2)}{p}-\cdots=\sum_{\nu\geq 0}\frac{g(p^{\nu+1})}{p^{\nu}},$$
where $g(n)$ is the multiplicative function defined by $f(n)=g\ast \textbf{1}(n)$.
Hence, 
$$|\gamma_p(1)|\geq |g(p)|+O_{\kappa}(1/p),$$
since for any $j\geq 2$ we have $g(p^{j})=f(p^{j})-f(p^{j-1})$, from which 
$|g(p^{j})|\leq d_{\kappa}(p^{j})+d_{\kappa}(p^{j-1})$.
We note that $|g(p)|$ coincides exactly with the absolute value of the previously defined function $g$ at $p$, without notational issues. Moreover, thanks to restriction $(5.b)$ we get $|\gamma_p(1)|>0$, if we choose $C=C(\kappa)$ large enough, thus making \eqref{FaadeBruno2} well defined.

On the other hand, we can rewrite $\gamma_{p}'(\sigma)$ as
$$\gamma_{p}'(\sigma)=(p-1)\sum_{\nu\geq 1}\frac{f(p^{\nu})}{p^{\nu\sigma}}(-\nu\log p)$$
from which we immediately deduce that
$$\gamma_{p}^{(j)}(1)=(p-1)\sum_{\nu\geq 1}\frac{f(p^{\nu})}{p^{\nu}}(-\nu\log p)^{j}\ \ \ (j\geq 1).$$
Clearly, $|\gamma_{p}^{(j)}(1)|\leq C_j (\log p)^{j},$ for fixed values of $j\geq 1$ and a certain constant $C_j=C_j(\kappa)>0$. 

Inserting the above estimates in \eqref{FaadeBruno2} we obtain
$$\bigg|\bigg(\frac{\gamma_{p}'}{\gamma_{p}}\bigg)^{(h)}(1)\bigg|\leq\tilde{C}_h\bigg(\frac{\log p}{\min\{|\gamma_p(1)|,1\}}\bigg)^{h+1},$$
for fixed values of $h$ and suitable constants $\tilde{C}_h=\tilde{C}_h(\kappa)>0$, which inserted in \eqref{defsmalltheta} gives
$$|\theta^{(j-1)}(1)|\leq \tilde{C}_j \sum_{p|s}\bigg(\frac{\log p}{\min\{|\gamma_p(1)|,1\}}\bigg)^j\leq \tilde{C}_j \tilde{\gamma}_s^{j}\max_{p|s}\{(\log p)^{j}\}\omega(s)\leq \tilde{C}_j B{^j}\tilde{\gamma}_s^{j}\omega(s)(\log\log N)^{j},$$
defining $\tilde{\gamma}_s:=\max_{p|s}\min\{|\gamma_p(1)|,1\}^{-1}.$
Finally, by restriction $(2)$ on $q$, we deduce
$$|\theta^{(j-1)}(1)|\leq A\tilde{C}_j B^j\tilde{\gamma}_s^{j}(\log\log N)^{j+1}.$$
Inserting this into \eqref{FaadeBruno} we obtain
\begin{equation}
\label{estimatederivtheta}
\Theta^{(k)}(1)\ll_{k,\kappa} |\Theta(1)|\xi^k\tilde{\gamma}_s^{k}(\log \log N)^{2k},
\end{equation}
for fixed values of $k$ and a constant $\xi=\xi(A,B,\tilde{C}_{1}(\kappa),\dots, \tilde{C}_{h}(\kappa))$. For future reference we observe that the explicit multiplicative form of $\Theta(1)$ is given by
\begin{equation}
\label{formoftheta1}
\Theta(1)=\prod_{p|s}\bigg(-p+(p-1)\sum_{\nu\geq 0}\frac{f(p^{\nu})}{p^{\nu}}\bigg)=\prod_{p|s}\bigg(\sum_{\nu\geq 0}\frac{g(p^{\nu+1})}{p^{\nu}}\bigg)=\prod_{p|s}\bigg(g(p)+O_\kappa\bigg(\frac1{p}\bigg)\bigg).
\end{equation}
We conclude this section by estimating the series in the first error term in \eqref{asymptoticsumb_2}. First, since $q\in\mathcal{A}$, we also have
$$\tilde{\Theta}(1)=\prod_{p|s}\bigg(1+\kappa+O_{\kappa}\bigg(\frac1{p}\bigg)\bigg)\ll_{\kappa,D}d_{\kappa+1}(s).$$
Second, by Lemma \ref{lemdirichletseries} and arguing as above, we find
\begin{equation}
\label{estimatederivthetatilde}
\tilde{\Theta}^{(k)}(1)\ll_{k,\kappa} |\tilde{\Theta}(1)|\tilde{\xi}^{k}(\log \log N)^{2k}\ll_{\kappa,D} \tilde{\xi}^{k}d_{\kappa+1}(s)(\log \log N)^{2k},
\end{equation}
for a suitable $\tilde{\xi}=\tilde{\xi}(A,B,\kappa)>0$. Inserting the bound for $\tilde{\Theta}^{(K+1)}(1)$ inside the first error term in \eqref{asymptoticsumb_2}, we obtain that this last one is
\begin{equation}
\label{errortildetheta}
\ll_{\kappa,D,K}\tilde{\xi}^{K+1}d_{\kappa+1}(s)(\log N)^{\Re(\tilde{\alpha})-K-1}(\log \log N)^{2K+2},
\end{equation}
using that $\sqrt{N}\ll M\ll N$, which exceeds the second error term in \eqref{asymptoticsumb_2}, if $N$ large enough in terms of $\delta,\kappa,A,A_1,B,D,K$. Collecting together \eqref{eq14}, \eqref{asymptoticsumb_2} and \eqref{errortildetheta}, we conclude that
\begin{align*}
\sum_{\substack{b\leq N/4\\ p|b\Rightarrow p|q}}\frac{f(b)c_q(b)}{b}\log^{\tilde{\alpha}}(N/b)&=\sum_{k=0}^{K}(-1)^{k}\binom{-\tilde{\alpha}+k-1}{k}\Theta^{(k)}(1)\sum_{\substack{b_1\leq \sqrt{N}\\ p|b_1\Rightarrow p|r}}\frac{f(b_1)c_r(b_1)}{b_1}(\log (N/b_1))^{\tilde{\alpha}-k}\\
&+O\bigg(\tilde{\xi}^{K+1}\sum_{\substack{b_1\leq \sqrt{N}\\ p|b_1\Rightarrow p|r}}\frac{d_{\kappa}(b_1)|c_r(b_1)|}{b_1}d_{\kappa+1}(s)(\log N)^{\Re(\tilde{\alpha})-K-1}(\log\log N)^{2K+2}\bigg)\\
&+O_{\kappa,D}\bigg(\frac{(\log N)^{\kappa-1}}{N^{3\delta/16}}d_{\kappa+1}(q)\bigg)\nonumber,
\end{align*}
which, by Lemma \ref{lemmatail}, can be rewritten as
\begin{align}
\label{rawformsumf(n)c_q(n)}
&\sum_{k=0}^{K}(-1)^{k}\binom{-\tilde{\alpha}+k-1}{k}\Theta^{(k)}(1)\sum_{\substack{b_1\leq \sqrt{N}\\ p|b_1\Rightarrow p|r}}\frac{f(b_1)c_r(b_1)}{b_1}(\log(N/b_1))^{\tilde{\alpha}-k}\\
&+O_{\delta,\kappa,D,K}\bigg(\tilde{\xi}^{K+1} d_{\kappa+1}(q)(\log N)^{\Re(\tilde{\alpha})-K-1}(\log\log N)^{2K+2}\bigg),\nonumber
\end{align}
if we choose $N$ large enough compared to $\delta,\kappa,A,A_1,B,D$ and $K$.
\section{The contribution from large prime factors}
We now need to compute the innermost sums in \eqref{rawformsumf(n)c_q(n)}. In order to simplify the calculations we observe that the main contribution comes only from squarefree values. Indeed, since $\kappa>1$, we have
\begin{align*}
\sum_{\substack{b_1\leq \sqrt{N}\\ p|b_1\Rightarrow p|r\\ b_1\ \text{not-squarefree}}}\frac{f(b_1)c_r(b_1)}{b_1}(\log(N/b_1))^{\tilde{\alpha}-k}&\ll \sum_{e|r}e\sum_{\substack{b_1\leq \sqrt{N}\\ p|b_1\Rightarrow p|r\\ b_1\ \text{not-squarefree}\\ e|b_1}}\frac{d_{\kappa}(b_1)}{b_1}(\log(N/b_1))^{\Re(\tilde{\alpha})-k}\\
&\ll(\log N)^{\Re(\tilde{\alpha})-k}\sum_{e|r}d_{\kappa}(e)\sum_{\substack{t\leq \sqrt{N}/e\\ p|t\Rightarrow p|r\\ t\neq 1}}\frac{d_{\kappa}(t)}{t}\\
&\leq (\log N)^{\Re(\tilde{\alpha})-k-B/4}d_{\kappa+1}(r)\sum_{\substack{p|t\Rightarrow p|r}}\frac{d_{\kappa}(t)}{t^{3/4}}\\
&\ll_{\kappa,D}(\log N)^{\Re(\tilde{\alpha})-k-B/4}d_{\kappa+1}(r),
\end{align*}
by condition $(6)$ on $q$. Using \eqref{estimatederivtheta} we find an overall contribution to \eqref{rawformsumf(n)c_q(n)} of at most
\begin{equation}
\label{contributionnonsqurefree}
(\log N)^{\Re(\tilde{\alpha})-B/4}|\Theta(1)|d_{\kappa+1}(r)\sum_{k=0}^{K}\bigg|\binom{-\tilde{\alpha}+k-1}{k}\bigg|\xi^{k}\tilde{\gamma}_s^{k}\bigg(\frac{(\log\log N)^2}{\log N}\bigg)^k.
\end{equation}
Now, by conditions $(4)$ and $(5.a)$ on $s$, we have $\tilde{\gamma}_s\ll \sqrt{\log\log N}.$ Moreover 
$$|\Theta(1)|\leq \prod_{p|s}\bigg(\kappa+1+O_\kappa\bigg(\frac1{p}\bigg)\bigg)\ll_{\kappa,D}d_{\kappa+1}(s),$$ by condition $(6)$ on $q$. Therefore, taking e.g. $B=4(K+2)$ and remembering that $\xi=\xi(A,B,\kappa)$, where we will be taking $A$ as a function of only $\kappa$ and $A_1$, we may conclude that \eqref{contributionnonsqurefree} will contribute $\ll_{\kappa,A_1,D,K}d_{\kappa+1}(q)(\log N)^{\Re(\tilde{\alpha})-K-2},$ which will be absorbed into 
the error term of \eqref{rawformsumf(n)c_q(n)}, if we choose $N$ sufficiently large with respect to $\delta,\kappa,A_1,D$ and $K$.

We are left with the estimate of 
\begin{align}
\label{sumdivisorr}
\sum_{\substack{b_1\leq \sqrt{N}\\ p|b_1\Rightarrow p|r\\ b_1\ \text{squarefree}}}\frac{f(b_1)c_r(b_1)}{b_1}(\log (N/b_1))^{\tilde{\alpha}-k}&=\sum_{\substack{b_1|r}}\frac{f(b_1)c_r(b_1)}{b_1}(\log(N/b_1))^{\tilde{\alpha}-k}\\
&=\mu(r)\sum_{\substack{b_1|r}}\frac{f(b_1)\phi(b_1)\mu(b_1)}{b_1}(\log(N/b_1))^{\tilde{\alpha}-k}\nonumber
\end{align}
since $r$ is square-free and $r\leq N^{1/2-3\delta/4}$. Note that we can replace the last sum with 
\begin{equation}
\label{newformsumovers}
\mu(r)\sum_{\substack{b|r}}f(b)\mu(b)(\log(N/b))^{\tilde{\alpha}-k}
\end{equation}
at the cost of a small error term. Indeed
\begin{align*}
\bigg|\frac{\phi(b)}{b}-1\bigg|=\bigg|\exp\bigg(\sum_{p|b}\log\bigg(1-\frac1{p}\bigg)\bigg)-1\bigg|&=\bigg|\exp\bigg(O\bigg(\sum_{p|b}\frac{1}{p}\bigg)\bigg)-1\bigg|\\
&\ll\sum_{p|b}\frac{1}{p}\ll \frac{1}{(\log N)^{B/4}}\sum_{p|b}\frac{1}{p^{3/4}}\ll_D \frac{1}{(\log N)^{B/4}}.
\end{align*}
Arguing as before, its overall contribution to \eqref{rawformsumf(n)c_q(n)} will be absorbed in the big-Oh term there.

Now, assuming $r$ of the form $r=ts'$, with $t$ and $s'$ as in restrictions $(3.b)-(3.d)$ on $q$, we can rewrite \eqref{newformsumovers} as 
\begin{equation}
\label{17}
\mu(r)\sum_{\substack{b|s'}}f(b)\mu(b)(\log(N/b))^{\tilde{\alpha}-k}-f(t)\mu(r)\sum_{\substack{b|s'}}f(b)\mu(b)(\log(N/tb))^{\tilde{\alpha}-k}.
\end{equation}
For $M\in\{N,N/t\}$, we write 
\begin{equation}
\label{lastbinomexp}
\sum_{\substack{b|s'}}f(b)\mu(b)(\log(M/b))^{\tilde{\alpha}-k}=\sum_{h=0}^{\infty}\binom{-\tilde{\alpha}+k+h-1}{h}(\log M)^{\tilde{\alpha}-k-h}\sum_{\substack{b|s'}}f(b)\mu(b)\log^{h} b.
\end{equation}
In the next section we will need estimates for $\sum_{\substack{b|s'}}f(b)\mu(b)\log^{h}b$ when $h\geq 1$. This is what we achieve next.
\begin{lem}
\label{lemmultinomialthm}
For any $h\geq 1$ and $s'$ as before, satisfying in particular condition $(3.c)$ and $(5.a)$, we have 
$$\bigg|\sum_{\substack{b|s'}}f(b)\mu(b)\log^{h}b\bigg|\leq |g(s')|(\eps \log N)^h.$$
\end{lem}
\begin{proof}
With the same spirit of what was previously done in section 7, we can write
\begin{align*}
\sum_{\substack{b|s'}}f(b)\mu(b)\log^{h} b&=(-1)^{h}\frac{d^{h}}{d\sigma^{h}}\bigg(\prod_{p|s'}\bigg(1-\frac{f(p)}{p^{\sigma}}\bigg)\bigg)\bigg|_{\sigma=0}\\
&=(-1)^{h}\sum_{j_1+j_2+\cdots+j_{\omega(s')}=h}\binom{h}{j_1,j_2,\dots,j_{\omega(s')}}\prod_{i=1}^{\omega(s')}\bigg(1-\frac{f(p_i)}{p_i^{\sigma}}\bigg)^{(j_i)}\bigg|_{\sigma=0}.
\end{align*}
We have
\[\bigg(1-\frac{f(p_i)}{p_i^{\sigma}}\bigg)^{(j_i)}\bigg|_{\sigma=0}= \left\{ \begin{array}{ll}
         1-f(p_i)& \mbox{if $j_i=0$};\\
         -f(p_i)(-\log p_i)^{j_i} & \mbox{if $j_i\neq 0$}.\end{array} \right. \]
Hence, we can rewrite the above expression as
$$(-1)^{h}\sum_{j_1+j_2+\cdots+j_{\omega(s')}=h}\binom{h}{j_1,j_2,\dots,j_{\omega(s')}}\prod_{\substack{i=1,\dots,\omega(s')\\ j_i\neq 0}}(-f(p_i))\prod_{\substack{i=1,\dots, \omega(s')\\ j_i=0}}(1-f(p_i))\prod_{i=1}^{\omega(s')}(-\log p_i)^{j_i}.$$
Since $s'$ satisfies condition $(5.a)$ and in particular for any prime $p|s'$ we have $f(p)\neq 1$, we may rewrite the above as
$$\prod_{p|s'}(1-f(p))\sum_{j_1+j_2+\cdots+j_{\omega(s')}=h}\binom{h}{j_1,j_2,\dots,j_{\omega(s')}}\prod_{\substack{i=1,\dots,\omega(s')\\ j_i\neq 0}}\bigg(\frac{-f(p_i)}{1-f(p_i)}\bigg)\prod_{i=1}^{\omega(s')}(\log p_i)^{j_i}.$$
Now, observe the above expression is upper bounded in absolute value by
\begin{align}
\label{estbinom}
&\leq |g(s')|\sum_{j_1+j_2+\cdots+j_{\omega(s')}=h}\binom{h}{j_1,j_2,\dots,j_{\omega(s')}}\prod_{i=1}^{\omega(s')}\bigg(\max\bigg\{\bigg|\frac{f(p_i)}{g(p_i)}\bigg|, 1\bigg\}\log p_i\bigg)^{j_i}\\
&=|g(s')|\bigg(\sum_{p|s'}\max\bigg\{\bigg|\frac{f(p)}{g(p)}\bigg|, 1\bigg\}\log p\bigg)^{h},\nonumber
\end{align}
by the multinomial theorem \cite{N}. Finally, note that 
$$\max\bigg\{\bigg|\frac{f(p)}{g(p)}\bigg|, 1\bigg\}\leq \max\bigg\{\frac{\kappa}{|g(p)|}, 1\bigg\}\leq \frac{\kappa}{\min\{|g(p)|,1\}}.$$
Since $s'$ satisfies restriction $(3.c)$, the second line of \eqref{estbinom} is 
\begin{equation*}
\label{estbinom2}
\leq |g(s')|(\eps \log N)^h,
\end{equation*}
which proves the lemma.
\end{proof}
\section{Combining the different pieces}
Collecting the results \eqref{firtsestimatef(n)c_q(n)}, \eqref{rawformsumf(n)c_q(n)} and \eqref{17}--\eqref{lastbinomexp}, we can see that $\sum_{n\leq N}f(n)c_q(n)$ equals to a main term of 
\begin{align}
\label{finalmainterm}
&N(\log N)^{\alpha-1}\sum_{j=0}^{J}\frac{-\lambda_j}{(\log N)^j}\sum_{k=0}^{K}(-1)^{k}\binom{-\alpha+j+k}{k}\frac{\Theta^{(k)}(1)}{(\log N)^{k}}\times\\
&\sum_{h=0}^{\infty}\binom{-\alpha+j+k+h}{h}\sum_{b|s'}f(b)\mu(s'/b)\frac{(\log b)^h}{(\log N)^h}\bigg(1-f(t)\bigg(1-\frac{\log t}{\log N}\bigg)^{\alpha-1-j-k-h}\bigg),\nonumber
\end{align}
since $\mu(r)=\mu(t)\mu(s')=-\mu(s'),$ plus an error term of
\begin{align}
\label{finalerror0}
&+O\bigg(d_{\kappa+1}(q)N\sum_{j=0}^{J}|\lambda_j|(\log N)^{\Re(\alpha)-j-K-2}(\log\log N)^{2(K+1)}\bigg)\\
&+O\bigg(d_{\kappa+1}(q)N(\log N)^{\kappa-1-A_1}(\log\log N)\bigg),\nonumber
\end{align}
where the big-Oh terms may depend on $\delta,\kappa,A_1,D,K$ and the implicit constant in \eqref{mainstatistic1} and the $\lambda_j$ are as in Theorem \ref{GKvariant}. We remind that $t$ indicates a prime number in the interval 
$$[N^{1/2-3\delta/4-\eps}, N^{1/2-3\delta/4-\eps/2}].$$

In order to estimate the contribution of the sum of the $\lambda_j$'s we are going to prove the following lemma.
\begin{lem}
\label{lemderivativeseulerproduct2}
Let $f$ be a multiplicative function such that $|f(n)|\leq d_{\kappa}(n)$, for a certain real positive constant $\kappa>1$ and any $N$-smooth integer $n$. Let $q$ be a squarefree positive integer smaller than $N$ satisfying condition $(4)$, with a large $C=C(\kappa,A_1)>\kappa+1$, and $(6)$. 

Then the coefficients $\lambda_j'=\Gamma(\alpha-j)\lambda_j$, where $\lambda_j$ are as defined in the statement of Theorem \ref{GKvariant}, satisfy $\lambda_j'\ll 1$, for $j=0,\dots,J$, with an implicit constant depending on $\kappa,A_1,D$ and that one in \eqref{mainstatistic1}.
\end{lem}
\begin{proof}
We remind that the coefficients $\lambda_j'$ are defined as
$$\lambda_j'=\lambda_j'(f,\alpha,q)=\sum_{l+h=j}\frac{(H_q^{-1})^{(h)}(1)c_l}{h!},$$
where
$$H_q(z):=\prod_{p|q}\bigg(1+\frac{f(p)}{p^z}+\frac{f(p^2)}{p^{2z}}+\cdots\bigg)\ \ \ (\Re(z)\geq 1)$$
and for any $0\leq l\leq J$ the $c_l$ are as in the statement of Theorem \ref{GKthm}. By Lemma \ref{lemderivativeseulerproduct0} each $c_l$ is uniformly bounded by a constant possibly depending on $\kappa,A_1,l$ and that in \eqref{mainstatistic1}. 

Therefore, to conclude the proof of the lemma we only need to show that each derivative $(H_q^{-1})^{(h)}(1)$ is bounded. 
However, we can write
$$H_q(z)=\prod_{p|q}\bigg(1+\frac{f(p)}{p^z}\bigg)\prod_{p|q}\bigg(1+\frac{f(p)}{p^z}+\frac{f(p^2)}{p^{2z}}+\cdots\bigg)\bigg(1+\frac{f(p)}{p^z}\bigg)^{-1}=:\tilde{H}_q(z)\tilde{\tilde{H}}_q(z).$$
Now it is not difficult to show that all the derivatives of $\tilde{\tilde{H}}_q(z)^{-1}$ at $z=1$ are uniformly bounded in $q$ and by Lemma \ref{lemderivativeseulerproduct} the same is true for those of $\tilde{H}_q(z)^{-1}$ at $z=1$. Finally, since we have
$$ \frac{d^h}{dz^h}H_q^{-1}(z)|_{z=1}=\sum_{l+k=h}\binom{h}{l}\frac{d^l}{dz^l}\tilde{H}_q^{-1}(z)|_{z=1}\frac{d^k}{dz^k}\tilde{\tilde{H}}_q^{-1}(z)|_{z=1}$$
we obtain the desired conclusion.
\end{proof}
By Lemma \ref{lemderivativeseulerproduct2}, choosing $K:=A_1$ and taking $N$ large enough in terms of $\delta,\kappa,A_1,D$ and the implicit constant in \eqref{mainstatistic1}, we see that the error term \ref{finalerror0} reduces to 
$$\ll d_{\kappa+1}(q)N(\log N)^{\kappa-1-A_1}(\log\log N).$$
Let us now focus on the main term \eqref{finalmainterm}. In the following we will make use several times of the following trivial estimates:
$$\bigg|1-f(t)\bigg(1-\frac{\log t}{\log N}\bigg)^{\alpha-1-j-k-h}\bigg|\leq 1+\kappa (1/2+3\delta/4+\eps/2)^{-\kappa-1-j-k-h}\leq 1+\kappa 2^{\kappa+1+j+k+h},$$
since $t\in [N^{1/2-3\delta/4-\eps},N^{1/2-3\delta/4-\eps/2}]$, if $\delta,\eps$ small, 
and 
\begin{align*}
\bigg|\binom{-\alpha+j+k+h}{h}\bigg|&=\frac{|(-\alpha+j+k+h)(-\alpha+j+k+h-1)\cdots (-\alpha+j+k+1)|}{h!}\\
&\leq \frac{(|\alpha|+j+k+h)(|\alpha|+j+k+h-1)\cdots (|\alpha|+j+k+1)}{h!}\nonumber\\
&=\binom{|\alpha|+j+k+h}{h}\leq \binom{\kappa+J+k+h}{h}\nonumber
\end{align*}
as well as similarly 
$$\bigg|\binom{-\alpha+j+k}{k}\bigg|\leq \binom{\kappa+J+k}{k}.$$
The contribution of the sum over $j\geq 1$ in \eqref{finalmainterm} is
$$\ll N(\log N)^{\Re(\alpha)-1}|g(s')\Theta(1)|\sum_{j=1}^{J}\frac{2^j |\lambda_j|}{(\log N)^j}\ll N(\log N)^{\Re(\alpha)-2}\frac{|g(s')\Theta(1)(\alpha-1)|}{|\Gamma(\alpha)|},$$
by using in sequence Lemma \ref{lemmultinomialthm}, the upper bound \eqref{estimatederivtheta}, conditions $(4)$ and $(5.a)$ on $s$, to estimate $\tilde{\gamma}_s$, and Lemma \ref{lemderivativeseulerproduct2}. Thus, the main term in \eqref{finalmainterm} reduces to 
\begin{align}
\label{finalmainterm2}
&-N(\log N)^{\alpha-1}\lambda_0\sum_{k=0}^{K}(-1)^{k}\binom{-\alpha+k}{k}\frac{\Theta^{(k)}(1)}{(\log N)^{k}}\times\\
&\sum_{h=0}^{\infty}\binom{-\alpha+k+h}{h}\sum_{b|s'}f(b)\mu(s'/b)\frac{(\log b)^h}{(\log N)^h}\bigg(1-f(t)\bigg(1-\frac{\log t}{\log N}\bigg)^{\alpha-1-k-h}\bigg).\nonumber
\end{align}
Working in a similar way as before, the contribution of the sum over $k\geq 1$ in \eqref{finalmainterm2} is
$$\ll N(\log N)^{\Re(\alpha)-2}(\log\log N)^{5/2}\frac{|c_0 g(s')\Theta(1)(\alpha-1)|}{|\Gamma(\alpha)|},$$
where we noticed that
$$\binom{-\alpha+k}{k}=\frac{(-\alpha+1)}{k}\binom{-\alpha+1+k-1}{k-1},$$
for any $k\geq 1$, and replaced $\lambda_0$ with
$$\lambda_0=\frac{c_0H_q^{-1}(1)}{\Gamma(\alpha)},$$
as in Theorem \ref{GKvariant}, and used Lemma \ref{lemderivativeseulerproduct} to estimate $|H_q^{-1}(1)|$. 

Thus, \eqref{finalmainterm2} reduces to 
\begin{align*}
&-N(\log N)^{\alpha-1}\lambda_0\Theta(1)\sum_{h=0}^{\infty}\binom{-\alpha+h}{h}\sum_{b|s'}f(b)\mu(s'/b)\frac{(\log b)^h}{(\log N)^h}\bigg(1-f(t)\bigg(1-\frac{\log t}{\log N}\bigg)^{\alpha-1-h}\bigg).
\end{align*}
Again, similar considerations lead to the following estimate for the contribution of the sum over $h\geq 1$ above:
$$\ll N(\log N)^{\Re(\alpha)-1}\frac{\eps|c_0 g(s')\Theta(1)(\alpha-1)|}{|\Gamma(\alpha)|}.$$
Assume now $\eps$ of the form $\eps:=\delta/V$, with $V\geq 5$ sufficiently large in terms of $\kappa$ in order to make the above binomial series convergent and to be determined later on in terms of the other parameters. Collecting the above error terms together, thanks to \eqref{gammaassumpt} and taking $N\geq N_0$, with $N_0(\delta,\kappa,A_1,V)$ sufficiently large, we have obtained
\begin{align}
\label{finalestimatef(n)c_q(n)}
&\sum_{n\leq N}f(n)c_q(n)=-\frac{c_0}{\Gamma(\alpha)}(f\ast \mu)(s')H_q^{-1}(1)\Theta(1)\theta_{N,\alpha}(t)N(\log N)^{\alpha-1}\\
&+O\bigg(N(\log N)^{\Re(\alpha)-1}\frac{\eps|c_0 g(s')\Theta(1)(\alpha-1)|}{|\Gamma(\alpha)|}+d_{\kappa+1}(q)N(\log N)^{\kappa-1-A_1}(\log\log N)\bigg)\nonumber
\end{align}
where the big-Oh term may depend on $\delta,\kappa,A_1,D$ and the implicit constant in \eqref{mainstatistic1} and we define
$$\theta_{N,\alpha}(p):=1-f(p)\bigg(1-\frac{\log p}{\log N}\bigg)^{\alpha-1},$$
for any prime $p\leq N$. 
\section{A Mertens' type estimate with $\theta_{N,\alpha}$}
In this section we insert a lemma about $\theta_{N,\alpha}$ which will play a fundamental role in the next section, where we will produce a lower bound for \eqref{finalestimatef(n)c_q(n)} on average over $q$. Indeed, to do this, we will need to lower bound an expression involving $\theta_{N,\alpha}(t)$ on average over $t\in [N^{1/2-3\delta/4-\delta/V}, N^{1/2-3\delta/4-\delta/2V}].$ This is what we achieve next.
\begin{lem}
\label{lemhyposumovert}
Let $f(n):\mathbb{N}\longrightarrow \mathbb{C}$ be a generalized divisor function as in Definition \ref{defgendivfun}, for parameters $\alpha,\beta,\kappa,A_1,A_2$ satisfying \eqref{relationbetaA_1}. Then there exists a small $\delta_0=\delta_0(\kappa)$, such that either for $\delta\leq \delta_0$ or for $\delta/2$, we have
\begin{equation}
\label{hyposumovert}
\sum_{\substack{t\ \textrm{prime}:\\ N^{1/2-3\delta/4-\delta/V}\leq t\leq N^{1/2-3\delta/4-\delta/2V}}}\frac{|\theta_{N,\alpha}(t)(f(t)-1)|}{t}\geq \eta\beta\frac{\delta}{V},
\end{equation}
for a certain $\eta=\eta(\delta,\kappa)>0$, if $V$ large enough with respect to $\delta,\kappa,A_1,A_2$ and the implicit constants in \eqref{mainstatistic1}--\eqref{mainstatistic2} and $N$ sufficiently large in terms of all these parameters.
\end{lem}
\begin{proof}
We split the proof into three main cases. First of all, if $\alpha=1$ then
$$\theta_{N,1}(t)=1-f(t)$$
and the sum \eqref{hyposumovert} reduces to 
\begin{align*}
\sum_{\substack{t\ \textrm{prime}:\\ N^{1/2-3\delta/4-\delta/V}\leq t\leq N^{1/2-3\delta/4-\delta/2V}}}\frac{|f(t)-1|^2}{t}&=\int_{N^{1/2-3\delta/4-\delta/V}}^{N^{1/2-3\delta/4-\delta/2V}}\frac{d(\beta t+R(t))}{t\log t}\\
&=\beta\log\bigg(1+\frac{\delta/2V}{1/2-3\delta/4-\delta/V}\bigg)+\int_{N^{1/2-3\delta/4-\delta/V}}^{N^{1/2-3\delta/4-\delta/2V}}\frac{d(R(t))}{t\log t}\\
&=\beta\log\bigg(1+\frac{\delta/2V}{1/2-3\delta/4-\delta/V}\bigg)+O((\log N)^{-A_2}),
\end{align*}
by partial summation from \eqref{mainstatistic2}, where $R(t)=O(t(\log t)^{-A_2})$ and the implicit constant in the big-Oh error term may depend on $\kappa,A_2$ and that of \eqref{mainstatistic2}.

By taking $N,V$ sufficiently large, $\delta$ small enough and thanks to \eqref{relationbetaA_1}, the above reduces to 
$$\beta\bigg(\frac{\delta}{V}(1+O(\delta))+O\bigg(\frac{\delta}{V}\bigg)^2\bigg)+O\bigg(\frac{\beta\delta^2}{V}\bigg)=\frac{\beta\delta}{V}+O\bigg(\frac{\beta\delta^2}{V}\bigg),$$
where now the implicit constant in the big-Oh error term is absolute. It is clear that now \eqref{hyposumovert} follows with $\eta=1/2$, say, by choosing $\delta$ suitably small. 

In particular, we have proved that
\begin{align}
\label{averageshortint1}
&\sum_{\substack{t\ \textrm{prime}:\\ N^{1/2-3\delta/4-\delta/V}\leq t\leq N^{1/2-3\delta/4-\delta/2V}}}\frac{|f(t)-1|^2}{t}=\frac{\beta\delta}{V}+O\bigg(\frac{\beta\delta^2}{V}\bigg),\\
\label{averageshortint2}
&\sum_{\substack{t\ \textrm{prime}:\\ N^{1/2-3\delta/4-\delta/V}\leq t\leq N^{1/2-3\delta/4-\delta/2V}}}\frac{1}{t}=\frac{\delta}{V}+O\bigg(\frac{\delta^2}{V}\bigg),
\end{align}
where \eqref{averageshortint2} follows as a special case of \eqref{averageshortint1}.

Let us now assume $\alpha\neq 1$ and $|\alpha-1|\leq \delta$. Then we can Taylor expand $\theta_{N,\alpha}(t)$ as 
\begin{align*}
\theta_{N,\alpha}(t)&=1-f(t)(1+(\alpha-1)\log(1/2+3\delta/4+\theta(t))+O(|\alpha-1|^2))\\
&=1-f(t)-f(t)(\alpha-1)\log(1/2+3\delta/4+\theta(t))+O_\kappa(|\alpha-1|^2)\\
&=(1-f(t))(1+(\alpha-1)\log(1/2+O(\delta)))-(\alpha-1)\log(1/2+O(\delta))+O_\kappa(|\alpha-1|^2),
\end{align*}
where $\theta(t)\in [\delta/V,\delta/2V]$ is defined by 
$t=N^{1/2-3\delta/4-\theta(t)}$ and we take $V\geq 5$ and $\delta$ small.

Inserting this into \eqref{hyposumovert} and using the triangle inequality, we get a lower bound for \eqref{hyposumovert} of
\begin{align*}
&\geq |1+(\alpha-1)\log(1/2+O(\delta))|\sum_{\substack{t\ \textrm{prime}:\\ N^{1/2-3\delta/4-\delta/V}\leq t\leq N^{1/2-3\delta/4-\delta/2V}}}\frac{|f(t)-1|^2}{t}\\
&-|(\alpha-1)\log(1/2+O(\delta))|\sum_{\substack{t\ \textrm{prime}:\\ N^{1/2-3\delta/4-\delta/V}\leq t\leq N^{1/2-3\delta/4-\delta/2V}}}\frac{|f(t)-1|}{t}\\
&+O_\kappa\bigg(|\alpha-1|^2\sum_{\substack{t\ \textrm{prime}:\\ N^{1/2-3\delta/4-\delta/V}\leq t\leq N^{1/2-3\delta/4-\delta/2V}}}\frac{|f(t)-1|}{t}\bigg).
\end{align*}
By Cauchy--Schwarz inequality and equations \eqref{averageshortint1}--\eqref{averageshortint2} we immediately deduce
\begin{equation}
\label{averageshortint3}
\sum_{\substack{t\ \textrm{prime}:\\ N^{1/2-3\delta/4-\delta/V}\leq t\leq N^{1/2-3\delta/4-\delta/2V}}}\frac{|f(t)-1|}{t}\leq \frac{\sqrt{\beta}\delta}{V}+O\bigg(\frac{\sqrt{\beta}\delta^2}{V}\bigg)
\end{equation}
and by Lemma \ref{lemmaarithmquadrineq}, taking $N$ sufficiently large with respect to $\delta,\kappa,A_1,A_2$ and the implicit constants \eqref{mainstatistic1}--\eqref{mainstatistic2}, we also have
\begin{equation}
\label{consequencearithmquadrineq}
|\alpha-1|\leq \sqrt{\beta}+O(\sqrt{\beta}\delta),
\end{equation}
thanks to \eqref{relationbetaA_1}. Hence, using \eqref{averageshortint1} and \eqref{averageshortint3} we can further lower bound \eqref{hyposumovert} with
\begin{align*}
&\geq |1+(\alpha-1)\log(1/2+O(\delta))|\bigg(\frac{\beta\delta}{V}+O\bigg(\frac{\beta\delta^2}{V}\bigg)\bigg)\\
&-|(\alpha-1)\log(1/2+O(\delta))|\bigg(\frac{\sqrt{\beta}\delta}{V}+O\bigg(\frac{\sqrt{\beta}\delta^2}{V}\bigg)\bigg)\\
&+O_\kappa\bigg(|\alpha-1|^2\bigg(\frac{\sqrt{\beta}\delta}{V}+O\bigg(\frac{\sqrt{\beta}\delta^2}{V}\bigg)\bigg)\bigg),
\end{align*}
which thanks to \eqref{consequencearithmquadrineq} and the hypothesis that $|\alpha-1|\leq \delta$ becomes
\begin{align*}
&\geq \frac{\beta\delta}{V}(1+O(\delta))-\log(2+O(\delta))\frac{\beta\delta}{V}(1+O(\delta))+O_\kappa\bigg(\frac{\beta\delta^2}{V}(1+O(\delta))\bigg)\\
&=\frac{\beta\delta}{V}(1-\log(2+O(\delta))+O(\delta)),
\end{align*}
which proves the lemma with $\eta=1/10$, say, if we take $\delta$ small enough.

Finally, we are left with the case $|\alpha-1|>\delta$. In this case we split the set of prime numbers into three sets:
\begin{align*}
&\mathcal{A}_1:=\{p: |\theta_{N,\alpha}(p)|\leq \delta^5\}\\
&\mathcal{A}_2:=\{p: |f(p)-1|\leq \delta^5\}\\
&\mathcal{A}_3:=\{p: |\theta_{N,\alpha}(p)|> \delta^5,|f(p)-1|>\delta^5 \}.
\end{align*}
\begin{rmk}
We expect the set $\mathcal{A}_3$, i.e. the set of primes where $\theta_{N,\alpha}$ and $f$ are respectively bounded away from $0$ and $1$, to contain a positive proportion of primes, at least on a small scale. Indeed, their complementary conditions should force $\alpha$ to be either very close to $1$ (which case we handled before) or very close to $2$, in which case we will succeed by adjusting the value of $\delta$.
\end{rmk}

We cover the interval $I:=[N^{1/2-3\delta/4-\delta/V}, N^{1/2-3\delta/4-\delta/2V}]$ with dyadic subintervals 
$$I=I'\cup \bigcup_{k=0}^{\lfloor \frac{\delta \log N}{2V\log 2} \rfloor-1} [N^{1/2-3\delta/4-\delta/V}2^k,N^{1/2-3\delta/4-\delta/V}2^{k+1}),$$
with $I'$ the possible rest of the above dyadic dissection. However, since we are looking for just a lower bound for \eqref{hyposumovert}, we can forget about $I'$.

Let us first suppose that for any $[x,2x)$ in the above union we have
$$|\mathcal{A}_3\cap [x,2x)|\geq \delta^5\frac{x}{\log x}.$$
Hence, in accordance with the prime number theorem, we are asking for a proportion of at least $\delta^5$ primes in the intersection $\mathcal{A}_3\cap [x,2x)$, for \emph{any} such $x$. From here it is easy to conclude, since \eqref{hyposumovert} will follow with a constant $\eta$ proportional to $\delta^{15}/(\kappa+1)^2$, since $\beta\leq (\kappa+1)^2$. 

Suppose now that there exists an interval 
\begin{equation}
\label{intcontradiction}
[x,2x):=[N^{1/2-3\delta/4-\delta/V}2^k,N^{1/2-3\delta/4-\delta/V}2^{k+1}),
\end{equation}
for a certain $k=0,\dots,\lfloor (\delta \log N)/(2V\log 2) \rfloor-1$, for which 
$$|\mathcal{A}_3\cap [x,2x)|< \delta^5 \frac{x}{\log x}.$$
This clearly implies that
$$|(\mathcal{A}_1\cup \mathcal{A}_2)\cap [x,2x)|\geq (1-\delta^5) \frac{x}{\log x}$$
and we let
$$|\mathcal{A}_1\cap [x,2x)|=d_1 \frac{x}{\log x},$$
for a certain $d_1\in [0,1]$. 
\begin{rmk}
One specific dyadic interval does not in general supply us with enough information on a function $f$ verifying \eqref{mainstatistic1}--\eqref{mainstatistic2} for single fixed values of $x$. However, we ask statistics \eqref{mainstatistic1}--\eqref{mainstatistic2} to hold uniformly on $2\leq x\leq N$. This imposes a \emph{rigidity} on the distribution of $f$ along the prime numbers, from which the ``local'' behaviour of $f$ is determined by the ``global'' one. In particular, the information that $f$ on average over all the primes smaller than \emph{any} $x\leq N$ is roughly $\alpha$, which we are now supposing to be bounded away from $1$, forces $f$ to be on \emph{any} dyadic interval $[x,2x)$, for large $x$, not too close to $1$, apart for a small proportion of primes. This, together with some structural information on $f$ over the primes that will be deducted from the definition of the sets $\mathcal{A}_1$ and $\mathcal{A}_2$, will negate the assumption that the almost totality of primes lies now in the union $\mathcal{A}_1\cup\mathcal{A}_2$.
\end{rmk}

We note that for primes in $\mathcal{A}_2$ we have
$$f(p)=1+O(\delta^5)$$
and for those in $\mathcal{A}_1$ we have instead
\begin{align*}
f(p)=\bigg(\frac{1}{2}+\frac{3\delta}{4}+\theta(p)\bigg)^{1-\alpha}+O(\delta^5)=\bigg(\frac{1}{2}+\frac{3\delta}{4}\bigg)^{1-\alpha}+O_\kappa(\delta^5+\delta/V).
\end{align*}
Therefore, from \eqref{mainstatistic1} and choosing $V=V(\delta,\kappa)$ sufficiently large, we get
\begin{align*}
\alpha x+O\bigg(\frac{x}{(\log x)^{A_1}}\bigg)&=\sum_{p\in (\mathcal{A}_1\cup \mathcal{A}_2\cup \mathcal{A}_3)\cap [x,2x)}f(p)\log p\\
&=\sum_{p\in \mathcal{A}_1\cap [x,2x)}\bigg(\bigg(\frac{1}{2}+\frac{3\delta}{4}\bigg)^{1-\alpha}+O_\kappa(\delta^5)\bigg)\log p+\sum_{p\in (\mathcal{A}_2\setminus \mathcal{A}_1)\cap [x,2x)}f(p)\log p\\
&+\sum_{p\in \mathcal{A}_3\cap [x,2x)}f(p)\log p\\
&=\sum_{p\in \mathcal{A}_1\cap [x,2x)}\bigg(\frac{1}{2}+\frac{3\delta}{4}\bigg)^{1-\alpha}\log p+\sum_{p\in (\mathcal{A}_2\setminus \mathcal{A}_1)\cap [x,2x)}\log p+O(\delta^5 x),
\end{align*}
from which, since $|(\mathcal{A}_1\cup \mathcal{A}_2)\cap [x,2x)|=(1+O(\delta^5))x/\log x$, we deduce that
\begin{align*}
\alpha x+O\bigg(\frac{x}{(\log x)^{A_1}}\bigg)&=\sum_{p\in \mathcal{A}_1\cap [x,2x)}\bigg(\bigg(\frac{1}{2}+\frac{3\delta}{4}\bigg)^{1-\alpha}-1\bigg)\log p+\sum_{p\in (\mathcal{A}_1\cup \mathcal{A}_2)\cap [x,2x)}\log p+O(\delta^5 x)\\
&=\bigg(\bigg(\frac{1}{2}+\frac{3\delta}{4}\bigg)^{1-\alpha}-1\bigg)\log(x+O(1))d_1\frac{x}{\log x}\\
&+\log(x+O(1))(1+O(\delta^5))\frac{x}{\log x}+O(\delta^5 x)\\
&=\bigg(\bigg(\frac{1}{2}+\frac{3\delta}{4}\bigg)^{1-\alpha}-1\bigg)d_1 x+x+O_\kappa(\delta^5 x+1/\log x).
\end{align*}
Finally, since $x\in [N^{1/2-3\delta/4-\delta/V},N^{1/2-3\delta/4-\delta/2V})$, by choosing $N$ sufficiently large with respect to $\delta,\kappa,A_1$ and the implicit constant in \eqref{mainstatistic1}, and dividing through by $x$, we conclude that
\begin{equation}
\label{contradictionfclose1}
\alpha-1+O_\kappa(\delta^5)=\bigg(\bigg(\frac{1}{2}+\frac{3\delta}{4}\bigg)^{1-\alpha}-1\bigg)d_1.
\end{equation}
Similar computations, but working with \eqref{mainstatistic2} instead, lead to
\begin{equation}
\label{contradictionfclose2}
\beta+O_\kappa(\delta^5)=\bigg|\bigg(\frac{1}{2}+\frac{3\delta}{4}\bigg)^{1-\alpha}-1\bigg|^2 d_1,
\end{equation}
if now $N$ is also sufficiently large with respect to $A_2$ and the implicit constant in \eqref{mainstatistic2}.

By substituting the value of $d_1$ from \eqref{contradictionfclose1} into \eqref{contradictionfclose2}, we find
\begin{equation}
\label{finalcontradiction0}
\beta+O_\kappa(\delta^5)=(\alpha-1)\bigg(\bigg(\overline{\frac{1}{2}+\frac{3\delta}{4}\bigg)^{1-\alpha}}-1\bigg).
\end{equation}
By dividing through by $\alpha-1$, remembering that $\delta<|\alpha-1|\leq \kappa+1$, and taking the conjugate, we can rewrite the above as
\begin{equation*}
\frac{\beta}{\bar{\alpha}-1}+1=\bigg(\frac{1}{2}+\frac{3\delta}{4}\bigg)^{1-\alpha}+O_\kappa(\delta^4).
\end{equation*}
If the left hand side in the above equation vanishes, we have
$$ \bigg(\frac{1}{2}+\frac{3\delta}{4}\bigg)^{1-\alpha}=O_\kappa(\delta^4),$$
which already leads to a contradiction, since $\delta$ can be chosen sufficiently small with respect to $\kappa$. Otherwise, we can pass to the logarithm on both sides and deduce that
$$\frac{1}{1-\alpha}\log\bigg(\frac{\beta}{\bar{\alpha}-1}+1\bigg)=\log\bigg(\frac{1}{2}+\frac{3\delta}{4}\bigg)+O_{\kappa}(\delta^3).$$
By Taylor expanding the logarithmic factor on the right hand side above as
\begin{equation*}
\log\bigg(\frac{1}{2}+\frac{3\delta}{4}\bigg)=-\log 2+\frac{3\delta}{2}+O(\delta^2),
\end{equation*}
and considering $\delta$ small enough in terms of $\kappa$, we finally get
\begin{equation*}
\frac{1}{1-\alpha}\log\bigg(\frac{\beta}{\bar{\alpha}-1}+1\bigg)+\log 2=\frac{3\delta}{2}+O(\delta^2).
\end{equation*}
A consequence of this is that, by shrinking $\delta$ if necessary, we should have:
\begin{equation}
\label{finalcontradiction}
R(\alpha,\beta):=\bigg|\frac{1}{1-\alpha}\log\bigg(\frac{\beta}{\bar{\alpha}-1}+1\bigg)+\log 2\bigg|\in \bigg[\frac{7\delta}{5}, \frac{8\delta}{5}\bigg].
\end{equation}
Now, we either have $R(\alpha, \beta)=0$ and \eqref{finalcontradiction} fails for any $\delta>0$ or it does by possibly replacing $\delta$ with $\delta/2$. In both cases we reach a contradiction, thus concluding the proof of the lemma.
\end{proof}
\section{The lower bound for the variance}
\subsection{Collecting the main results}
Plugging \eqref{finalestimatef(n)c_q(n)} into \eqref{startingpoint1}, we find for $V(Q,f)$ a lower bound of:
\begin{align}
\label{finalestimlowerboundvariance1}
&\gg QN(\log N)^{-\beta+2(\Re(\alpha)-1)}\bigg(\sum_{KQ_0\leq q\leq RN^{-\delta/4}}^{'}\bigg(\bigg|\frac{c_0}{\Gamma(\alpha)}\bigg| h_1(q)+R_{\alpha}(N,q)\bigg)+E(N)\bigg)^{2}\\
&+O_\kappa\bigg(\frac{N^2(\log N)^{\kappa^2+4\kappa+2}}{Q_0}+\frac{N^2(\log N)^{\beta+2\Re(\alpha)-2}}{Q_0}\bigg),\nonumber
\end{align}
where we let 
\begin{align*}
&E(N):=O_{\delta,\kappa}\bigg(\frac{(\log N)^{-\Re(\alpha)+1}}{N^{\delta/11}}\bigg)\\
&R_{\alpha}(N,q):=O_{\delta,\kappa,A_1,A_2,D}\bigg(\frac{\eps|c_0(\alpha-1)|h_2(q)}{|\Gamma(\alpha)|}+\frac{h_3(q)(\log\log N)}{(\log N)^{-\kappa+\Re(\alpha)+A_1}}\bigg),
\end{align*}
with
\begin{align}
\label{defh}
& h_1(q):=\frac{|g(q)g(s')H_q^{-1}(1)\Theta(1)\theta_{N,\alpha}(t)|}{q};\\
& h_2(q):=\frac{|g(q)g(s')\Theta(1)|}{q};\nonumber\\ 
& h_3(q):=\frac{|g(q)|d_{\kappa+1}(q)}{q}.\nonumber
\end{align}
Here $\sum^{'}$ indicates a sum over all numbers $q$ satisfying restrictions $(1)$ to $(6)$ and the $\gg$ constant may depend on $\delta,\kappa,A_1,A_2,D$ and the implicit constants in \eqref{mainstatistic1}--\eqref{mainstatistic2}.
\subsection{The sum of $R_{\alpha}(N,q)$}
We can easily estimate the sum of $R_{\alpha}(N,q)$ by using Lemma \ref{lemrankinestimate}. For the sum involving $h_3$ the contribution will be 
$$\ll_\kappa(\log N)^{(\kappa+1)^2+\kappa-\Re(\alpha)-A_1}(\log\log N).$$
Regarding the sum involving $h_2$ instead, it may be bounded by
\begin{align}
\label{sumofh2}
\ll |c_0|\frac{\eps|\alpha-1|}{|\Gamma(\alpha)|}\sum_{\substack{s'\leq N^{\eps}\\ p|s'\Rightarrow p>(\log N)^{B}}}\frac{|g(s')|^2}{s'}\sum_{\substack{t\ \textrm{prime}:\\ t\in I_{\delta}(N)}}\frac{|g(t)|}{t}\sum_{\substack{s\leq N^{1/2}\\ p|s\Rightarrow p\leq (\log N)^{B}}}\frac{|\Theta(1)g(s)|}{s},
\end{align}
where $I_\delta(N):=[N^{1/2-3\delta/4-\delta/V},N^{1/2-3\delta/4-\delta/2V}]$.

Now, observe from \eqref{formoftheta1} that $|\Theta(1)|\leq \prod_{p|s}(|g(p)|+O_\kappa(1/p))$ and trivially $|g(s)|= \prod_{p|s}|g(p)|$. Hence, by Rankin's trick the innermost sum in \eqref{sumofh2} is 
\begin{align*}
\leq \prod_{\substack{p\leq (\log N)^{B}}}\bigg(1+\frac{|g(p)|^2}{p}+O_{\kappa}\bigg(\frac{1}{p^2}\bigg)\bigg)\asymp_{\kappa}\prod_{\substack{p\leq (\log N)^{B}}}\bigg(1+\frac{|g(p)|^2}{p}\bigg).
\end{align*} 

Regarding the sum over $s'$, arguing similarly and since $(\log N)^{B}\leq N^{\eps}$, if we take $N$ large enough with respect to $\eps$ and $A_1$, we see it is
$$\ll_\kappa \prod_{(\log N)^{B}<p\leq N^{\eps}}\bigg(1+\frac{|g(p)|^2}{p}\bigg)=\frac{\prod_{p\leq N^{\eps}}\bigg(1+\frac{|g(p)|^2}{p}\bigg)}{\prod_{p\leq (\log N)^{B}}\bigg(1+\frac{|g(p)|^2}{p}\bigg)}\ll \frac{\eps^{\beta}(\log N)^{\beta}}{\prod_{p\leq (\log N)^{B}}\bigg(1+\frac{|g(p)|^2}{p}\bigg)},$$
by partial summation from \eqref{mainstatistic2}, made possible thanks to the hypothesis \eqref{relationbetaA_1} on $A_2$. Here the implicit constant depends on $\kappa,A_1,A_2$ and the implicit constants in  \eqref{mainstatistic1}--\eqref{mainstatistic2} and we take $N$ large enough with respect to these parameters. 

Finally we come to the sum over the primes $t$. By Cauchy--Schwarz and equations \eqref{averageshortint1}--\eqref{averageshortint2} it is $\ll\sqrt{\beta}\eps.$ 

Hence, overall we get a bound for \eqref{sumofh2} of
$$\ll \eps^{2+\beta}\frac{|c_0(\alpha-1)|}{|\Gamma(\alpha)|}\sqrt{\beta}(\log N)^{\beta}\ll |c_0|\frac{\beta\eps^{2+\beta}}{|\Gamma(\alpha)|}(\log N)^{\beta},$$ 
where we used $|\alpha-1|\ll \sqrt{\beta}$ from Lemma \ref{lemmaarithmquadrineq} and where again the implicit constant depends on $\kappa,A_1,A_2$ and those in \eqref{mainstatistic1}--\eqref{mainstatistic2} and we take $N$ sufficiently large with respect to these parameters. 
\begin{rmk}
It is essential here to relate $\alpha-1$ to $\sqrt{\beta}$ by means of the tight bound supplied by Lemma \ref{lemmaarithmquadrineq}, otherwise the above error coming from the sum involving $h_2$ could potentially overcome the main term coming from the sum of $h_1$. 
\end{rmk}
\subsection{The main term}
By expanding the definition of $h_1(q)$ and all the conditions $q$ is subject to, we see that the precise shape of $\Sigma^{'}_q h_1(q)$ is
\begin{equation}
\label{preciseshapesum}
\sum_{\substack{s'\leq N^{\eps}\\ p|s'\Rightarrow p>(\log N)^{B},\\ p|s'\Rightarrow |g(p)|>(\log\log N)^{-1/2}\\ s'\in\mathcal{A}'\\ s'\ \textrm{squarefree}}}\frac{|g(s')|^2|H_{s'}^{-1}(1)|}{s'}\sum_{\substack{t\ \textrm{prime}:\\ N^{1/2-3\delta/4-\eps}\leq t\leq N^{1/2-3\delta/4-\eps/2}\\ f(t)\neq 1}}\frac{|\theta_{N,\alpha}(t)g(t)H_t^{-1}(1)|}{t}
\end{equation}  
$$\times\sum_{\substack{KQ_0/ts'\leq s\leq RN^{-\delta/4}/ts'\\ p|s\Rightarrow C<p\leq (\log N)^B,\ p>C/|g(p)|,\ |g(p)|>(\log\log N)^{-1/2}\\ \omega(tss')\leq A\log\log N\\ tss'\in\mathcal{A}\\ s\ \textrm{squarefree}}}\frac{|\Theta(1)g(s)H_s^{-1}(1)|}{s}.$$ 
We now insert a series of observations to simplify its estimate. 
\subsection{Removal of some extra conditions}
To begin with, by Lemma \ref{lemderivativeseulerproduct} we have 
$$|H_q^{-1}(1)|\gg_{\kappa,D} 1.$$
In the following we then replace $h_1(q)$ with the value in $\eqref{defh}$ without the factor $H_q^{-1}(1).$ 

Let us now focus on the condition $(2)$. To this aim we note that
$$\sum_{\substack{KQ_0\leq q\leq RN^{-\delta/4}\\ \omega(q)\leq A\log\log N}}^{'}h_1(q)=\sum_{KQ_0\leq q\leq RN^{-\delta/4}}^{'}h_1(q)-\sum_{\substack{KQ_0\leq q\leq RN^{-\delta/4}\\ \omega(q)> A\log\log N}}^{'}h_1(q),$$
where now $\sum^{'}$ indicates the sum over all of the other remaining restrictions on $q$. The last sum on the right hand side above can be upper bounded by
$$\ll \frac{1}{(\log N)^{A}}\sum_{q\leq N}h_1(q)e^{\omega(q)}.$$
Using again 
$$|\Theta(1)|\leq \prod_{p|s}\bigg(|g(p)|+O_\kappa\bigg(\frac1{p}\bigg)\bigg)$$
and 
$$|g(s')|= \prod_{p|s'}|g(p)|,$$
as well as the trivial 
$$|\theta_{N,\alpha}(t)|\ll_{\kappa} 1,$$
we are left to estimate
$$\frac{1}{(\log N)^A}\sum_{\substack{t\leq \sqrt{N}\\ t\ \textrm{prime}}}\frac{1}{t}\sum_{q'\leq N}\frac{\prod_{p|q'}(e(\kappa+1)^2+O(1/p))}{q'}.$$
This can be done by means of Lemma \ref{lemrankinestimate} and Mertens' theorem and the result will be 
$$\ll_{\kappa}(\log\log N)(\log N)^{e(\kappa+1)^2-A}.$$
So far, if we collect together all the error terms inside the parenthesis in \eqref{finalestimlowerboundvariance1}, we have got an overall error of
\begin{align}
\label{finalerror}
&\ll(\log\log N)(\log N)^{e(\kappa+1)^2-A}+(\log N)^{(\kappa+1)^2+\kappa-\Re(\alpha)-A_1}(\log\log N)\\
&+|c_0|\frac{\beta\eps^{2+\beta}}{|\Gamma(\alpha)|}(\log N)^{\beta}+\frac{(\log N)^{-\Re(\alpha)+1}}{N^{\delta/11}}\nonumber\\
&\ll (\log N)^{(\kappa+1)^2+\kappa-\Re(\alpha)-A_1}(\log\log N)+|c_0|\frac{\beta\eps^{2+\beta}}{|\Gamma(\alpha)|}(\log N)^{\beta}\nonumber,
\end{align}
choosing 
$$A:=A_1+e(\kappa+1)^2+1$$
and taking $N$ sufficiently large in terms of $\delta$ and $\kappa$, where the implicit constant above depends on $\delta,\kappa,A_1,A_2,D$ and those in  \eqref{mainstatistic1}--\eqref{mainstatistic2}. 

We now concentrate on the condition $(6)$. It is certainly equivalent to
$$\frac{(\log t)^{A_1+1}}{t^{3/4}}+\sum_{p|s}\frac{(\log p)^{A_1+1}}{p^{3/4}}+\sum_{p|s'}\frac{(\log p)^{A_1+1}}{p^{3/4}}\leq D.$$
Since $t$ is extremely large and all the primes dividing $s'$ are at least $(\log N)^{B}$, with $B=4(K+2)=4(A_1+2)$, it is actually equivalent to the fact that the corresponding sum over the prime factors of $s$ must be slightly smaller than $D$. So we can lower bound \eqref{preciseshapesum} with the same expression but having the innermost sum switched with that over those numbers $s$ satisfying:
$$\sum_{p|s}\frac{(\log p)^{A_1+1}}{p^{3/4}}\leq D-1.$$
Now, this is the complete sum minus that under the complementary condition. This last one is upper bounded by
\begin{align*}
&\leq\sum_{\substack{KQ_0/ts'\leq s\leq RN^{-\delta/4}/ts'\\ p|s\Rightarrow C<p\leq (\log N)^B\\ \sum_{p|s}\frac{(\log p)^{A_1+1}}{p^{3/4}}>D-1}}\frac{\prod_{p|s}(|g(p)|^2+O(1/p))}{s}\\
&\leq \sum_{\substack{C<r\leq (\log N)^B\\ r\ \textrm{prime}}}\frac{(\log r)^{A_1+1}}{(D-1)r^{3/4}}\sum_{\substack{s\leq RN^{-\delta/4}\\ p|s\Rightarrow C<p\leq (\log N)^B\\ r|s}}\frac{\prod_{p|s}(|g(p)|^2+O(1/p))}{s}\\
&\ll_\kappa \sum_{\substack{r\leq (\log N)^B}}\frac{(\log r)^{A_1+1}}{(D-1)r^{7/4}}\sum_{\substack{s\leq RN^{-\delta/4}\\ p|s\Rightarrow C<p\leq (\log N)^B}}\frac{\prod_{p|s}(|g(p)|^2+O(1/p))}{s}\\
&\ll_\kappa \sum_{\substack{r\leq (\log N)^B}}\frac{(\log r)^{A_1+1}}{(D-1)r^{7/4}}\prod_{\substack{C<p\leq (\log N)^B}}\bigg(1+\frac{|g(p)|^2+O_\kappa(1/p)}{p}\bigg)\\
&\ll_{A_1}\frac{1}{D-1}\prod_{\substack{C<p\leq (\log N)^B}}\bigg(1+\frac{|g(p)|^2+O_\kappa(1/p)}{p}\bigg),
\end{align*}
by Rankin's trick. By the arbitrariness of $D=D(\kappa,A_1)$, this term will be negligible. Indeed, we will now show that the complete sum over $s$ contributes
$$\gg \prod_{\substack{C<p\leq (\log N)^B}}\bigg(1+\frac{|g(p)|^2+O_\kappa(1/p)}{p}\bigg).$$
\subsection{The estimate of the sum over $s$}
We start by setting the value of $Q_0$ as 
$$Q_0:=\frac{N(\log N)^{\eta_0}}{Q\beta^2},$$
where
$$\eta_0:=(\kappa+2)^2-\beta-2(\Re(\alpha)-1)+3=(\kappa+1)^2-\beta+2(\kappa-\Re(\alpha))+8\geq 8$$
and $\beta$ is as in the statement of Theorem \ref{thmalpha}. Note that this choice satisfies the conditions in \eqref{eq0} and Lemma \ref{lemparseval-shiu}, if $N$ large enough in terms of $\delta,\kappa$ and $A_1$. By condition \eqref{relationbetaA_1}, we deduce that 
$$\frac{KQ_0}{ts'}\ll \frac{(\log N)^{\eta_0} N^{-\delta/4+\eps}}{\beta^2}\leq N^{-\delta/4+\eps}(\log N)^{\eta_0+2(A_1-\kappa(\alpha,\beta))},$$
with 
$$\kappa(\alpha,\beta)=(\kappa+1)^2+\kappa-\Re(\alpha)-\beta+4.$$
Thus, recalling that $\eps=\delta/V$, with $V\geq 5$, and taking $N$ large enough in terms of $\delta,\kappa$ and $A_1$, we have $KQ_0/ts'<1$. Thanks to this the sum over $s$ becomes a sum over a \emph{long} interval, which heavily simplifies its computation. In particular, it coincides with
$$\sum_{\substack{s\leq RN^{-\delta/4}/ts'\\ p|s\Rightarrow C<p\leq (\log N)^B,\ p>C/|g(p)|,\ |g(p)|>(\log\log N)^{-1/2}}}\frac{|g(s)\prod_{p|s}(g(p)+O(1/p))|}{s}.$$
Applying Lemma \ref{lemrankinestimate} we find it is 
$$\gg_{\kappa}\prod_{\substack{C<p\leq \min\{RN^{-\delta/4}/ts',(\log N)^B\}\\ p>C/|g(p)|,\ |g(p)|>(\log\log N)^{-1/2}}}\bigg(1+\frac{|g(p)|^2+O_\kappa(1/p)}{p}\bigg).$$
We restrict now the sum over $s'$ to those numbers $\leq N^{\eps/W}$, for a certain $W\geq 3$ to determine later. In this way, it is immediate to check that 
$$\frac{RN^{-\delta/4}}{ts'}=\frac{N^{1/2-3\delta/4}}{ts'}\geq \frac{N^{1/2-3\delta/4}}{N^{1/2-3\delta/4-\eps/2+\eps/W}}= \exp\bigg(\eps\bigg(\frac{1}{2}-\frac{1}{W}\bigg)\log N\bigg)\geq(\log N)^{B},$$
for $N$ large enough with respect to $\eps$ and $A_1$. Thus the product above is indeed only over the prime numbers $C<p\leq (\log N)^B$ and it equals $P_1/P_2$, where
\begin{align*}
&P_1:=\prod_{\substack{C<p\leq (\log N)^B\\ p>C/|g(p)|}}\bigg(1+\frac{|g(p)|^2+O_\kappa(1/p)}{p}\bigg),\\
&P_2:=\prod_{\substack{C<p\leq (\log N)^B\\ p>C/|g(p)|,\\ |g(p)|\leq(\log\log N)^{-1/2}}}\bigg(1+\frac{|g(p)|^2+O_\kappa(1/p)}{p}\bigg).
\end{align*}
However $P_2$ is of bounded order, since
\begin{equation*}
\sum_{\substack{C<p\leq (\log N)^B\\ p>C/|g(p)|,\ |g(p)|\leq(\log\log N)^{-1/2}}}\frac{|g(p)|^2+O_\kappa(1/p)}{p}\leq \sum_{\substack{p\leq (\log N)^B\\ |g(p)|\leq(\log\log N)^{-1/2}}}\bigg(\frac{|g(p)|^2}{p}+O_\kappa \bigg(\frac{1}{p^2}\bigg)\bigg)\ll_\kappa 1,
\end{equation*}
by Mertens' theorem, if $N$ large compared to $\kappa$ and $A_1$. Regarding $P_1$ instead, it coincides with $P_3/P_4$, where 
\begin{align*}
&P_3:=\prod_{\substack{C<p\leq (\log N)^B}}\bigg(1+\frac{|g(p)|^2+O_\kappa(1/p)}{p}\bigg),\\
&P_4:=\prod_{\substack{C<p\leq (\log N)^B\\ p\leq C/|g(p)|}}\bigg(1+\frac{|g(p)|^2+O_\kappa(1/p)}{p}\bigg).
\end{align*}
As before, one can show that $P_4$ is bounded, which makes the sum over $s$ at least of order
\begin{align*}
\gg_{\kappa} \prod_{\substack{C<p\leq (\log N)^B}}\bigg(1+\frac{|g(p)|^2+O_\kappa(1/p)}{p}\bigg)\asymp_{\kappa,C}\prod_{\substack{C<p\leq (\log N)^B}}\bigg(1+\frac{|g(p)|^2}{p}\bigg).
\end{align*}
\subsection{The estimate of the sum over $t$}
We remind that $\eps=\delta/V$ and we assume $\delta,N$ and $V$ to be as in Lemma \ref{lemhyposumovert}. We then make use of \eqref{hyposumovert} to lower bound the sum over $t$ in \eqref{preciseshapesum}.

\subsection{The estimate of the sum over $s'$}
By previous considerations, the sum over $s'$ is
\begin{equation}
\label{sumovers'}
\sum_{\substack{s'\leq N^{\eps/W},\\ p|s'\Rightarrow p>(\log N)^{B}\\ |g(p)|>(\log\log N)^{-1/2}\\ s'\in\mathcal{A}'}}\frac{|g(s')|^2}{s'}=\sum_{\substack{s'\leq N^{\eps/W},\\ p|s'\Rightarrow p>(\log N)^{B}\\ |g(p)|>(\log\log N)^{-1/2}}}\frac{|g(s')|^2}{s'}-\sum_{\substack{s'\leq N^{\eps/W},\\ p|s'\Rightarrow p>(\log N)^{B}\\ |g(p)|>(\log\log N)^{-1/2}\\ s'\not\in\mathcal{A}'}}\frac{|g(s')|^2}{s'}.
\end{equation}
We may deal with the second sum on the right hand side of \eqref{sumovers'} using the definition of the set $\mathcal{A}'$ of condition $(3.c)$ in the following way:
\begin{align*}
\sum_{\substack{s'\leq N^{\eps/W},\\ p|s'\Rightarrow p>(\log N)^{B}\\ |g(p)|>(\log\log N)^{-1/2}\\ s'\not\in\mathcal{A}'}}\frac{|g(s')|^2}{s'}&\ll_{\kappa} \frac{1}{\eps \log N}\sum_{\substack{s'\leq N^{\eps/W},\\ p|s'\Rightarrow p>(\log N)^{B}\\ |g(p)|>(\log\log N)^{-1/2}}}\frac{|g(s')|^2}{s'}\sum_{\substack{r|s':\\ r\ \textrm{prime}}}\frac{\log r}{\min\{|f(r)-1|,1\}}\\
&=\frac{1}{\eps \log N}\sum_{\substack{(\log N)^{B}<r\leq N^{\eps/W}\\|g(r)|>(\log\log N)^{-1/2}\\ r\ \textrm{prime}}}\frac{\log r}{\min\{|f(r)-1|,1\}}\sum_{\substack{s'\leq N^{\eps/W}\\ r|s'\\ \\ p|s'\Rightarrow p>(\log N)^{B}\\ |g(p)|>(\log\log N)^{-1/2}}}\frac{|g(s')|^2}{s'}\\
&\leq \frac{1}{\eps \log N}\sum_{\substack{(\log N)^{B}<r\leq N^{\eps/W}\\|g(r)|>(\log\log N)^{-1/2}\\ r\ \textrm{prime}}}\frac{|f(r)-1|^2\log r}{\min\{|f(r)-1|,1\}r}\sum_{\substack{s'\leq N^{\eps/W}\\ p|s'\Rightarrow p>(\log N)^{B}\\ |g(p)|>(\log\log N)^{-1/2}}}\frac{|g(s')|^2}{s'}\\
&\ll_{\kappa}\frac{1}{W}\sum_{\substack{s'\leq N^{\eps/W}\\ p|s'\Rightarrow p>(\log N)^{B}\\ |g(p)|>(\log\log N)^{-1/2}}}\frac{|g(s')|^2}{s'},
\end{align*}
where the fraction $|f(r)-1|^2/\min\{|f(r)-1|,1\}$ is easily seen to be bounded and we used Mertens' theorem to compute the sum over the primes.

Thus, choosing a value of $W=W(\kappa)\geq 3$ large enough we deduce that \eqref{sumovers'} is
\begin{align*}
\gg_{\kappa} \prod_{\substack{(\log N)^{B}<p\leq N^{\eps/W}\\ |g(p)|>(\log\log N)^{-1/2}}}\bigg(1+\frac{|g(p)|^2}{p}\bigg)&=\frac{\prod_{(\log N)^{B}<p\leq N^{\eps/W}}\bigg(1+\frac{|g(p)|^2}{p}\bigg)}{\prod_{\substack{(\log N)^{B}<p\leq N^{\eps/W}\\ |g(p)|\leq (\log\log N)^{-1/2}}}\bigg(1+\frac{|g(p)|^2}{p}\bigg)}\\
&\gg \prod_{(\log N)^{B}<p\leq N^{\eps/W}}\bigg(1+\frac{|g(p)|^2}{p}\bigg),
\end{align*}
by Lemma \ref{lemrankinestimate} and since the product in the denominator above is bounded.
\subsection{Completion of the proof of Theorem \ref{thmalpha}}
Collecting the above estimates together, we have found an overall lower bound for the sum involving $h_1(q)$ in \eqref{finalestimlowerboundvariance1} of 
$$\gg \bigg|\frac{\eta c_0}{\Gamma(\alpha)}\bigg|\frac{\beta\delta}{V}\prod_{C<p\leq N^{\eps/W}}\bigg(1+\frac{|g(p)|^2}{p}\bigg),$$
with $c_0$ as in the statement of Theorem \ref{thmalpha} and $\eta$ as in Lemma \ref{lemhyposumovert}.

The above product can be estimated through partial summation, giving a contribution of 
$$\gg\bigg(\frac{\eps}{W}\bigg)^{\beta}(\log N)^{\beta}\gg \eps^{\beta}(\log N)^{\beta},$$
where the $\gg $ constant depends on $\kappa,A_2,C$ and that in \eqref{mainstatistic2}.

Recalling that $\eps=\delta/V$, $C$ depends on $\delta, \kappa, A_1$ and the implicit constant in \eqref{mainstatistic1} and collecting the previous two estimates together, we have proved that the sum of $h_1(q)$ in \eqref{finalestimlowerboundvariance1} is
$$\gg \bigg|\frac{\eta c_0}{\Gamma(\alpha)}\bigg|\frac{\delta^{1+\beta}}{V^{1+\beta}}\beta(\log N)^{\beta},$$
where the above implicit constant may depend on $\delta,\kappa,A_1,A_2$ and the implicit constant in \eqref{mainstatistic1}--\eqref{mainstatistic2} and we consider $N$ as sufficiently large with respect to all these parameters. We deduce a lower bound for the term inside parenthesis in \eqref{finalestimlowerboundvariance1} of
\begin{align*}
&\gg \bigg|\frac{\eta c_0}{\Gamma(\alpha)}\bigg|\frac{\delta^{1+\beta}}{V^{1+\beta}}\beta(\log N)^{\beta}+O\bigg((\log N)^{(\kappa+1)^2+\kappa-\Re(\alpha)-A_1}(\log\log N)+|c_0|\frac{\beta\eps^{2+\beta}}{|\Gamma(\alpha)|}(\log N)^{\beta}\bigg)\\
&\gg \bigg|\frac{\eta c_0}{\Gamma(\alpha)}\bigg|\frac{\delta^{1+\beta}}{V^{1+\beta}}\beta(\log N)^{\beta},
\end{align*}
thanks to conditions \eqref{relationbetaA_1}--\eqref{gammaassumpt}, if we take $V$ large enough in terms of $\delta,\kappa,A_1,A_2$ and the implicit constants in  \eqref{mainstatistic1}--\eqref{mainstatistic2}, and $N$ sufficiently large in terms of all these parameters.
 
Remembering that 
$$Q_0=\frac{N(\log N)^{\eta_0}}{Q\beta^2},$$
where 
$$\eta_0=(\kappa+2)^2-\beta-2(\Re(\alpha)-1)+3,$$
as well as the relations \eqref{relationbetaA_1}, we have overall found that
\begin{equation}
\label{finalestimlowerboundvariance2}
V(Q,f) \gg \bigg|\frac{c_0\beta}{\Gamma(\alpha)}\bigg|^2\bigg(\frac{\delta}{V}\bigg)^{2(1+\beta)}QN(\log N)^{\beta+2(\Re(\alpha)-1)},
\end{equation}
where the implicit constant above may depend on $\delta,\kappa, A_1,A_2$ and those in \eqref{mainstatistic1}--\eqref{mainstatistic2} and $N\geq N_0$, with $N_0$ large enough depending on all these parameters. Since the term $(\delta/V)^{2(1+\beta)}$ is uniformly bounded in terms of the aforementioned parameters, it may be absorbed in the implicit constant in \eqref{finalestimlowerboundvariance2}. Finally, recalling the estimate \eqref{averagefsquare}, we notice that equation \eqref{finalestimlowerboundvariance2} is actually in the form stated in Theorem \ref{thmalpha}, thus concluding its proof. 

\section*{Acknowledgements}
The author is indebted to his supervisor Adam J. Harper for numerous conversations about this problem. In particular, he would like to thank him for a careful reading of this paper and for several suggestions that substantially improved the exposition and notably simplified various proofs.

\end{document}